\documentclass{article}
\usepackage{latexsym}
\usepackage{leftidx}
\usepackage{amsmath,amsthm, color}
\usepackage{enumerate}
\usepackage{amssymb}
\usepackage{upgreek}
\usepackage[margin=1.1in,dvips]{geometry}

\def\f12{\frac 1 2}

\def\a{\alpha}
\def\b{\beta}
\def\ga{\gamma}

\def\ep{\epsilon}
\def\la{\lambda}

\def\si{\sigma}
\def\Si{\Sigma}
\def\om{\omega}
\def\Om{\Omega}

\def\ub{\underline{u}}
\def\Lb{\underline{L}}

\def\ab{\underline{\a}}
\def\Hb{\underline{H}}

\def\pa{\partial}
\def\les{\lesssim}

\def\f12{\frac 1 2}
\def\div{\text{div}}

\newcommand{\vol}{\textnormal{vol}}

\newcommand{\D}{\mbox{$D \mkern-13mu /$\,}}
\newcommand{\J}{\mbox{$J \mkern-13mu /$\,}}
\newcommand{\nabb}{\mbox{$\nabla \mkern-13mu /$\,}}
\newcommand{\divs}{\mbox{$\div \mkern-13mu /$\,}}
\newtheorem{thm}{Theorem}

\newtheorem{prop}{Proposition}

\newtheorem{lem}{Lemma}
\newtheorem{cor}{Corollary}
\newtheorem{remark}{Remark}
\theoremstyle{definition}

\begin{document}

\title{Decay of solutions of Maxwell-Klein-Gordon equations with large Maxwell field}
\date{}
\author{Shiwu Yang}
\maketitle
\begin{abstract}
In the author's work \cite{yangILEMKG}, it has been shown that solutions of Maxwell-Klein-Gordon equations in $\mathbb{R}^{3+1}$ possess some form of global strong decay properties with data bounded in some weighted energy space. In this paper, we prove pointwise decay estimates for the solutions for the case when the initial data are merely small on the scalar field but can be arbitrarily large on the Maxwell field. This extends the previous result of Lindblad-Sterbenz \cite{LindbladMKG}, in which smallness was assumed both for the scalar field and the Maxwell field.
\end{abstract}

\section{Introduction}

In this paper, we study the pointwise decay of solutions to the Maxwell-Klein-Gordon equations on $\mathbb{R}^{3+1}$ with large Cauchy data. To define the equations, let $A=A_\mu dx^\mu$ be a $1$-form. The covariant derivative associated to this 1-form is
\begin{equation*}
D_\mu =\pa_\mu+\sqrt{-1}A_\mu,
\end{equation*}
which can be viewed as a $U(1)$ connection on the complex line bundle over $\mathbb{R}^{3+1}$ with the standard flat metric $m_{\mu\nu}$. Then the curvature $2$-form $F$ associated to this connection is given by
\begin{equation*}
F_{\mu\nu}=-\sqrt{-1}[D_{\mu}, D_{\nu}]=\pa_\mu A_\nu-\pa_\nu A_\mu=(dA)_{\mu\nu}.
\end{equation*}
This is a closed $2$-form, that is, $F$ satisfies the Bianchi identity
\begin{equation}
\label{bianchi}
 \pa_\ga F_{\mu\nu}+\pa_\mu F_{\nu\ga}+\pa_\nu F_{\ga\mu}=0.
\end{equation}
The Maxwell-Klein-Gordon equations (MKG) is a system for the connection field $A$ and the complex scalar field $\phi$:
\begin{equation}
 \label{EQMKG}\tag{MKG}
\begin{cases}
\pa^\nu F_{\mu\nu}=\Im(\phi \cdot\overline{D_\mu\phi})=J_\mu;\\
D^\mu D_\mu\phi=\Box_A\phi=0.
\end{cases}
\end{equation}
These are Euler-Lagrange equations of the functional
\[
L[A, \phi]=\iint_{\mathbb{R}^{3+1}}\frac{1}{4}F_{\mu\nu}F^{\mu\nu}+\frac{1}{2}D_{\mu}\phi\overline{D^{\mu}\phi}dxdt.
\]
A basic feature of this system is that it is gauge invariant under the following gauge transformation:
\[
 \phi\mapsto e^{i\chi}\phi; \quad A\mapsto A-d\chi.
\]
More precisely, if $(A, \phi)$ solves \eqref{EQMKG}, then $(A-d\chi, e^{i\chi}\phi)$ is also a solution for any potential function $\chi$. Note that $U(1)$ is abelian. The Maxwell field $F$ is invariant under the above gauge transformation and \eqref{EQMKG} is said to be an \textsl{abelian gauge theory}. For the more general theory when $U(1)$ is replaced by a compact Lie group, the corresponding equations are referred to as \textsl{Yang-Mills-Higgs equations}.

\bigskip

In this paper, we consider the Cauchy problem to \eqref{EQMKG}. The initial data set $(E, H, \phi_0, \phi_1)$ consists of the initial electric field $E$, the magnetic field $H$, together with initial data $(\phi_0, \phi_1)$ for the scalar field. In terms of the solution $(F, \phi)$, on the initial hypersurface, these are:
\begin{equation*}
F_{0i}=E_i,\quad \leftidx{^*}F_{0i}=H_i,\quad \phi(0, x)=\phi_0,\quad D_t\phi(0, x)=\phi_1,
\end{equation*}
where $\leftidx{^*}F$ is the Hodge dual of the 2-form $F$. In local coordinates $(t, x)$,
\[
(H_1, H_2, H_3)=(F_{23}, F_{31}, F_{12}).
\]
 The data set is said to be \textsl{admissible} if it satisfies the compatibility condition
\begin{equation}
\label{eq:comp:cond}
div(E)=\Im(\phi_0\cdot \overline{\phi_1})=\left.J_0\right|_{t=0},\quad div (H)=0,
\end{equation}
where the divergence is taken on the initial hypersurface $\mathbb{R}^3$. For solutions of \eqref{EQMKG}, the energy
\[
 E[F, \phi](t):=\int_{\mathbb{R}^3}|E|^2+|H|^2+|D\phi|^2dx
\]
is conserved. Another important conserved quantity is the total charge
\begin{equation}
\label{defcharge}
q_0=\frac{1}{4\pi}\int_{\mathbb{R}^3}\Im(\phi\cdot \overline{D_t\phi})dx=\frac{1}{4\pi}\int_{\mathbb{R}^3}div (E)dx,
\end{equation}
which can be defined at any fixed time $t$. The existence of nonzero charge plays a crucial role in the asymptotic behavior of solutions of \eqref{EQMKG}. It makes the analysis more complicated and subtle. This is obvious from the above definition as the electric field $E_i=F_{0i}$ has a tail $q_0r^{-3}x_i$ at any fixed time $t$.

\bigskip

The Cauchy problem to \eqref{EQMKG} has been studied extensively. One of the most remarkable results is due to Eardley-Moncrief in \cite{Moncrief1}, \cite{Moncrief2}, in which it was shown that there is always a global solution to the general Yang-Mills-Higgs equations for sufficiently smooth initial data.
 This has later been improved to data merely bounded in the energy space for MKG by Klainerman-Machedon in \cite{MKGkl} and for the non-abelian case of Yang-Mills equations in e.g. \cite{YMkl}, \cite{sungjinYM},
\cite{MKGtesfahun}. Since then there has been an extensive literature on generalizations and extensions of this classical result, aiming at improving the regularity of the initial data in order to construct a global solution, see \cite{MKGtataru}, \cite{MKGtao}, \cite{KriegerMKG4}, \cite{MKGmachedon}, \cite{OhMKG4}, \cite{MKGigor} and references therein. A common feature of all these works is to
 construct a local solution with rough data. Then the global well-posedness follows by establishing a priori bound for some appropriate norms of the solution. For example, a local solution was constructed in \cite{Moncrief1} while in \cite{Moncrief2}, they showed that the $L^\infty$ norm of the solution never blows up even though it may grow in time $t$. As a consequence the solution can be extended to all time; however the decay property of the solution is unknown. In view of this, although the solution of \eqref{EQMKG} exists
globally with rough initial data, very little is known about the decay properties.

\bigskip
Asymptotic behavior and decay estimates are well understood for linear fields (see e.g.\cite{asymLkl}) and nonlinear fields with sufficiently small initial data (see e.g.\cite{fieldschrist}, \cite{Shu}). These mentioned results rely on the conformal symmetry of the system, either by conformally compactifying the Minkowski space or by using the conformal killing vector field $(t^2+r^2)\pa_t+2tr\pa_r$ as multiplier. Nevertheless the use of the conformal symmetry requires strong decay of the initial data and thus in general does not allow the presence of nonzero charge except when the initial data are essentially compactly supported. For the case with nonzero charge, the first related work regarding the asymptotic properties was due to W. Shu in \cite{shu2}. However, that work only considered the case when the solution is trivial outside
a fixed forward light cone. Details for general case were not carried out. A complete proof towards this program was later contributed by Lindblad-Sterbenz in \cite{LindbladMKG}, also see a more recent work \cite{MKGLydia}.

 \bigskip

 The presence of nonzero charge has a long range effect on the asymptotic behavior of the solutions, at least in a neighbourhood of the spatial infinity. This can be seen from the conservation law
of the total charge as the electric field $E$ decays at most $ r^{-2}$ as $r\rightarrow\infty$ at any fixed time. This weak decay rate makes the analysis more complicated even for small initial data. To deal with this difficulty, Lindblad-Sterbenz constructed a global chargeless field and made use of the fractional Morawetz estimates obtained by using the vector fields $u^p\pa_u+v^p\pa_v$ as multipliers. The latter work \cite{MKGLydia} relied on the observation that the angular derivative of the Maxwell field has zero charge. The Maxwell field then can be estimated by using Poincar\'e inequality.

\bigskip

The asymptotic behavior of solutions of MKG with general large data remains unknown until recently in \cite{yangILEMKG} quantitative decay estimates have been obtained for solutions with data bounded in some weighted energy space. Pointwise decay requires the energy estimates for the derivatives of the solution. However, commuting the equations with derivatives generates nonlinear terms. The aim of this paper is to identify a class of large data for MKG equations such that we can derive the pointwise decay of the solutions.

\bigskip

We define some necessary notations in order to state our main result.
We use the standard polar local coordinate system $(t, r,
\om)$ of Minkowski space as well as the null coordinates $u=\frac{t-r}{2}$, $v=\frac{t+r}{2}$. Let $\nabla$ denote the derivative on $\mathbb{R}^3$ and $\Om$ be the set of angular momentum vector fields $\Om_{ij}=x_i\pa_j-x_j\pa_i$.
Without loss of generality we only prove estimates in the future, i.e., $t\geq 0$.
 Next we introduce a null frame $\{L, \Lb, e_1, e_2\}$, where
\[
L=\pa_v=\pa_t+\pa_r,\quad \Lb=\pa_u=\pa_t-\pa_r
\]
and $\{e_1, e_2\}$ is an orthonormal basis of the sphere with
constant radius $r$. We use $\D$ to denote the covariant derivative associated to the connection field $A$ on the sphere with radius $r$. For any 2-form $F$, denote the null decomposition under the above null frame by
\begin{equation}
\label{eq:curNull}
\a_i=F_{Le_i},\quad\underline{\a}_i=F_{\Lb e_i},\quad \rho=\f12 F_{\Lb L}, \quad \si=F_{e_1 e_2},\quad i\in{1, 2}.
\end{equation}
We assume that the initial data set $(E, H, \phi_0, \phi_1)$ is admissible. Let $q_0$ be the charge defined in \eqref{defcharge} which is
uniquely determined by the initial data of the scalar field $(\phi_0, \phi_1)$. We assume that the data for the scalar field is small but the data for the Maxwell field is large. However the data can not be assigned freely. They satisfy the compatibility condition \eqref{eq:comp:cond}. To measure the size of the initial data for the scalar field and the Maxwell field, let $(E^{df}, E^{cf})$ be the Hodge decomposition of the electronic field $E$ with $E^{df}$ the divergence free part and $E^{cf}$ the curl free part. Then the compatibility condition \eqref{eq:comp:cond} on $E$ is equivalent to
\[
\div E^{cf}=\Im(\phi_0\cdot \overline{\phi_1}).
\]
This implies that $E^{cf}$ can be uniquely determined by $(\phi_0, \phi_1)$ (with suitable decay assumption on $E$). Therefore for the initial data set $(E, H, \phi_0, \phi_1)$ for \eqref{EQMKG} we can freely assign $\phi_0$, $\phi_1$ and $E^{df}$, $H$ as long as $\div H=0$, $\div E^{df}=0$. The total charge $q_0$ is a constant determined by $(\phi_0, \phi_1)$.

We now define the norms of the initial data. For some positive constant $0<\ga_0<1$, we define the second order weighted Sobolev norm respectively for the initial data of the Maxwell field $(E, H)$ and the initial data of the scalar field $(\phi_0, \phi)$:
\begin{align*}
\mathcal{M}:&=\sum\limits_{l\leq 2}(1+r)^{1+\ga_0}(|\Om^l E^{df}|^2+|\Om^l H|^2+|\nabla^l E^{df}|^2+|\nabla^l H|^2)dx,\\
\mathcal{E}:&=\sum\limits_{l\leq 2}(1+r)^{1+\ga_0}(|\nabla\Om^l\phi_0|^2+|\Om^l \phi_1|^2+|\nabla^{l+1} \phi_0|^2+|\nabla^{l} \phi_1|^2+|\phi_0|^2)dx.
\end{align*}
We remark here that the definition for $\mathcal{E}$ is not gauge invariant. The gauge invariant norm depends on the connection field $A$ which up to to a gauge transformation can be determined by the initial data of the Maxwell field $(E^{df}, H)$. However in our setting $\mathcal{M}$ can be arbitrarily large while $\mathcal{E}$ is assumed to be small depending on $\mathcal{M}$. It it hence much clear to use a norm for the scalar field so that it does not depend on $\mathcal{M}$. However we will show later (see Lemma \ref{lem:IDbd} in Section \ref{sec:btstp}) that the gauge invariant norm is equivalent to the above Sobolev norm up to a constant depending on $\mathcal{M}$.

We now can state our main theorem:
\begin{thm}
 \label{thm:dMKG:small}
Consider the Cauchy problem to \eqref{EQMKG} with admissible initial data set $(E, H, \phi_0, \phi_1)$. Then there exists a positive constant $\ep_0$, depending on $\mathcal{M}$, $\ga_0$, such that for all $\mathcal{E}<\ep_0$, the solution $(F, \phi)$ of \eqref{EQMKG} satisfies the following decay estimates:
\begin{align*}
|D_{\Lb}(r\phi)|^2(u, v, \om)\leq C\mathcal{E}u_+^{-1-\ga_0},&\quad |r\ab|^2(u, v, \om)\leq C u_+^{-1-\ga_0};\\
r^p(|D_L(r\phi)|^2+|\D(r\phi)|^2)(u, v, \om)\leq C\mathcal{E} u_+^{p-1-\ga_0},&\quad 0\leq p\leq 1+\ga_0;\\
 r^p(|r\a|^2+|r\si|^2)(u, v, \om)\leq C u_+^{p-1-\ga_0},&\quad 0\leq p\leq 1+\ga_0;\\
r^{p+2}|\rho-q_0 r^{-2}\chi_{\{t+R\leq r\}}|^2(u, v, \om)\leq C u_+^{p-1-\ga_0},&\quad 0\leq p< 1,\\
r^p|\phi|^2(u, v,\om)\leq C\mathcal{E} u_+^{p-2-\ga_0},&\quad 1\leq p\leq 2\\
|D\phi|^2(t, x)+|\phi|^2(t, x)\leq C\mathcal{E}(1+t)_+^{-1-\ga_0},&\quad |F|^2(t, x)\leq C(1+t)_+^{-1-\ga_0} ,\quad \forall |x|\leq R
 \end{align*}
for all $(u, v, \om)\in\mathbb{R}^{3+1}\cap\{|x|\geq R\}$ and for some constant $C$ depending on $\mathcal{M}$, $\ga_0$, $p$. Here $q_0$ is the total charge and $\chi_{\{t+2\leq r\}}$ is the characteristic function on the exterior region $\{t+2\leq r\}$ and $u_+=2+|u|$.
\end{thm}
We make several remarks:
\begin{remark}
The second order derivatives of the initial data is the minimum regularity we need to derive the above pointwise decay of the solution. Similar decay estimates hold for the higher order derivatives of the solution if higher order weighted Sobolev norms of the initial data are known.
\end{remark}

\begin{remark}
The restriction on $\ga_0$, that is $0<\ga_0<1$, is merely for the sake of brevity. If $\ga_0\geq 1$, then the decay property of the solutions propagates in the exterior region ($t+2\leq r$). In other words, we have the same decay estimates as in the theorem for $\tau\leq 0$. However in the interior region where $\tau>0$, the maximal decay rate is $\tau_+^{-2}$ (corresponding to $\ga_0=1$), that is, the decay rate in the interior region for $\ga_0\geq 1$ in general can not be better than that of $\ga_0=1$.
\end{remark}

\begin{remark}
Since we assume the scalar field is small, the charge is also small by definition. Combined with the techniques in \cite{yangILEMKG}, our approach can be adapted to the case with large charge. There are two ways of generalizations. The first one is to relax the assumption on the scalar field as in \cite{yang5} so that the charge can be large. Secondly we can consider the following unphysical equations:
\begin{equation*}
\begin{cases}
\pa^\nu F_{\mu\nu}=\la\Im(\phi \cdot\overline{D_\mu\phi});\\
\Box_A\phi=0
\end{cases}
\end{equation*}
for some constant $\la$.
\end{remark}
Compared to the previous result of Lindblad-Sterbenz \cite{LindbladMKG}, we have made the following improvements: first of all, we obtain pointwise decay estimates for solutions of \eqref{EQMKG} for a class of large initial data. We only require smallness on the scalar field. In particular our initial data for \eqref{EQMKG} can be arbitrarily large. Combining the method in \cite{yang5}, we can even make the data on the scalar field large in the energy space. Secondly we have lower regularity on the initial data. In \cite{LindbladMKG}, it was assumed that the derivative of the initial data decays one order better, that is, $\nabla^{I}(E^{df}, H)$, $D^{I}(D\phi_0, \phi_1)$ belong to the weighted Sobolev space with weights $(1+r)^{1+\ga_0+2|I|}$, while in this paper we only assume that the angular derivatives of the data obey this improved decay (see the definition of $\mathcal{M}$, $\mathcal{E}$). For the other derivatives, the weights is merely $(1+r)^{1+\ga_0}$. This makes the analysis more delicate. Moreover, as the solution decays weaker initially, our decay rate is weaker than that in \cite{LindbladMKG} (only decay rate in $\tau$, the decay in $r$ is the same). However if we assume the same decay of the initial data as in \cite{LindbladMKG}, then we are able to obtain the same decay for the solution.

\bigskip

We use a new approach developed in \cite{yangILEMKG} to study the asymptotic behavior of solutions of \eqref{EQMKG}. This new method was originally introduced by Dafermos-Rodnianski in \cite{newapp} for the study of decay of linear waves on black hole spacetimes. This novel method starts by proving the energy flux decay of the solutions of linear equations through the forward light cone $\Si_{\tau}$ (see definitions in Section \ref{sec:notation}). The pointwise decay then follows by commuting the equation with $\pa_t$ and the angular momentum $\Om$. In the abstract framework set by Dafermos-Rodnianski in \cite{newapp}, the energy flux decay relies on three kinds of basic ingredients and estimates: a uniform energy bound, an integrated local energy decay estimate and a hierarchy of $r$-weighted energy estimates in a neighbourhood of the null infinity, which can be obtained by using the vector fields $\pa_t$, $f(r)\pa_r$, $r^p(\pa_t+\pa_r)$ as multipliers respectively. Combining these three estimates, a pigeon hole argument then leads to the energy flux decay.

\bigskip

As the initial data for the scalar field is small, we can use perturbation method to prove the pointwise decay of the solution. With a suitable bootstrap assumption on the nonlinearity $J=\Im(\phi\cdot \overline{D\phi})$, we first can use the new method to prove energy decay estimates for the Maxwell field up to the second order derivatives. Once we have these decay estimates for the Maxwell field, we then can show the energy decay as well as pointwise decay for the scalar field, which can then be used to improve the bootstrap assumption. The smallness of the scalar field is used here to close the bootstrap assumptions.

\bigskip

The existence of nonzero chargel has a long range effect on the asymptotic behavior of the solution in the exterior region $\{t+R\leq r\}$, which has been discussed in \cite{yangILEMKG} when the charge is large. To deal with this difficulty, we define the chargeless 2-form
\[
\bar F=F-\chi_{\{t+R\leq r\}}q_0 r^{-2}dt\wedge dr.
\]
We first carry out estimates for $\bar F$ on the exterior region $\{t+R\leq r\}$, which in particular controls the energy flux through $\{t+R=r\}$ (the intersection of the interior region and the exterior region). We then can use the new method to obtain estimates for the Maxwell field $F$ in the interior region. The Maxwell equation commutes with the Lie derivatives of $F$ (see Lemma \ref{lem:commutator}). It is not hard to obtain energy decay estimates for the derivatives of the Maxwell field under suitable bootstrap assumptions on the nonlinearity $J$ by using the new approach.

\bigskip

The main difficulty lies in showing the energy decay estimates for the scalar field due to fact that the covariant derivative $D$  is not commutable. The interaction terms of the Maxwell field and the scalar field arise from the commutator. To control those interaction terms, previous results (\cite{LindbladMKG}, \cite{MKGLydia}) rely on the smallness of the Maxwell field and those terms could be absorbed. The key observation that the Maxwell field is allowed to be large in this paper is that the robust new method makes use of the decay in $\tau$ (equivalent to $1+|t-r|$ up to a constant) and those terms could be controlled by using Gronwall's inequality without smallness assumption on the Maxwell field. Traditionally, the Gronwall's inequality is used with respect to the foliation $t=constant$. Therefore strong decay in $t$ is necessary. As the new method foliates the spacetime by using the null hypersurfaces $S_{\tau}$, it enables us to make use of the weaker decay in $\tau$ in order to apply Gronwall's inequality.

\bigskip

The paper is organized as follows: we define additional notations and derive the transport equations for the curvature components in Section \ref{sec:notation}; we then review the energy identities respectively for the scalar field and the Maxwell field in Section \ref{sec:energyID}; in Section \ref{sec:decay:lin}, we use the new method to obtain energy decay estimates first for the Maxwell field and then for the scalar field; finally in the last section, we improved the bootstrap assumption and conclude our main theorem.

\textbf{Acknowledgments}
The author would like to thank Pin Yu for helpful discussions.

\section{Preliminaries and notations}
\label{sec:notation}
We define some additional notations used in the sequel. Recall the null frame $\{L, \Lb, e_1, e_2\}$ defined in the introduction. At any fixed point $(t, x)$, we may choose $e_1$, $e_2$ such that
\begin{equation*}
 [L, e_i]=-\frac{1}{r}e_i,\quad [\Lb, e_i]=\frac{1}{r}e_i,\quad
\left.[e_1, e_2]\right|_{(t, x)}=0,\quad i\in\{1, 2\}.
\end{equation*}
This helps to compute those geometric quantities which are
independent of the choice of the local coordinates. We then can compute the covariant derivatives for the null frame at any fixed point:
\begin{equation}
 \label{eq:nullderiv}
 \begin{split}
&\nabla_L L=0,\quad   \nabla_L \Lb=0, \quad \nabla_L e_i=0,\quad \nabla_{\Lb}\Lb=0, \quad \nabla_{\Lb}e_i=0,\\
&\nabla_{e_i}L=r^{-1}e_i,\quad \nabla_{e_i}\Lb=-r^{-1}e_i, \quad \nabla_{e_1}e_2=\nabla_{e_2}e_1=0,\quad \nabla_{e_i}e_i=-r^{-1}\pa_r.
 \end{split}
\end{equation}
We use $\pa$ to abbreviate $(\pa_t,
\pa_{1}, \pa_{2}, \pa_{3}) =(\pa_t, \nabla)$ and $\nabb$ to denote the covariant derivative on the sphere with radius $r$.

Now we define the foliation of the spacetime. Let $H_u$ be the outgoing null hypersurface $\{t-r=2u\}$ and $\underline{H}_{\underline{u}}$ as the incoming null hypersurface $\{t+r=2 \underline{u}\}$. Let $R>1$ be a fixed constant. We now use this fixed constant $R$
to define the foliation of the future of the initial hypersurface $t=0$. Let $\tau^*=\frac{\tau-R}{2}$. In the exterior region where $t+R\leq r$, we use the foliation
\[
 \Si_{\tau}:=H_{\tau^*}\cap \{t\geq 0\},\quad \tau\leq 0.
\]
In the interior region where $t+R\geq r$, let $\Si_{\tau}$ be the foliation defined as follows:
\begin{align*}
 \Si_{\tau}:=\{t=\tau, \quad |x|\leq R\}\cup (H_{\tau^*}\cap \{|x|\geq R\}).
\end{align*}
In particular, the future spacetime $t\geq 0$ is foliated by $\Si_{\tau}$, $\tau\in \mathbb{R}$ where $\tau\leq 0$ foliates the exterior region and $\tau\geq 0$ gives the foliation in the interior region.

Note that the boundary of the region bounded by $\Si_{\tau_1}$ and $\Si_{\tau_2}$ is part of the future null infinity where the decay behavior of the solution is unknown. To make the energy estimates rigorous, we consider the finite truncated region. For any $v_0\geq \frac{\tau+R}{2}$, denote the truncated $\Si_{\tau}$ as
\[
 \Si_{\tau}^{v_0}=\Si_{\tau}\cap\{v\leq v_0\}.
\]
Unless we specify it, in the following the outgoing null hypersurface $H_u$ stands for $H_{u}\cap \{t\geq 0\}$ in the exterior region and
$H_{u}\cap \{|x|\geq R\}$ in the interior region. On the initial hypersurface $t=0$, we denote the annulus with radii $r_1<r_2$ as
\[
 B_{r_1}^{ r_2}=\{r_1\leq |x|\leq r_2\}, \quad B_{r}=B_{r}^{\infty}.
\]
Next we define the domains.
In the exterior region, for $\tau_2\leq \tau_1\leq 0$, denote $\mathcal{D}_{\tau_1}^{ \tau_2}$ to be the Cauchy development of the annulus $\{R-\tau_1\leq |x|\leq R- \tau_2\}$ or more precisely
\[
 \mathcal{D}_{\tau_1}^{ \tau_2}=\{(t, x)|||x|+\tau_1^*+\tau_2^*|+t\leq \tau_1^*-\tau_2^*\}.
\]
The boundary of this domain consists of the spacelike initial surface $B_{R-\tau_1}^{R-\tau_2}$ and the truncated outgoing and incoming null hypersurfaces which we denote as:
\begin{align*}
 H_{\tau_1^*}^{\tau_2^*}=H_{\tau_1^*}\cap \mathcal{D}_{\tau_1}^{ \tau_2},\quad\underline{H}_{-\tau_2^*}^{\tau_1^*}=\underline{H}_{-\tau_2^*}\cap \mathcal{D}_{\tau_1^*}^{ \tau_2^*}.
\end{align*}
In the interior region, for any $\tau_2\geq \tau_1\geq 0$, we denote $\mathcal{D}_{\tau_1}^{\tau_2}$ to be the region bounded by $\Si_{\tau_1}$ and $\Si_{\tau_2}$:
\begin{align*}
\mathcal{D}_{\tau_1}^{\tau_2}=\{(t, x)|(t, x)\in \Si_\tau, \tau_1\leq \tau\leq \tau_2\}.
\end{align*}
We use
$\bar{\mathcal{D}}_{\tau_1}^{\tau_2}=\mathcal{D}_{\tau_1}^{\tau_2}\cap\{|x|\geq R\}$ to denote the region outside the cylinder.l
The incoming null boundary of this finite region is denoted by $\underline{H}_{v}^{\tau_1, \tau_2}=\underline{H}_v\cap \{\tau_1^*\leq u\leq \tau_2^*\}$.
 Finally, for $\tau\in\mathbb{R}$, let $\mathcal{D}_{\tau}=\mathcal{D}_{\tau}^{+\infty}$ when $\tau\geq 0$ and $\mathcal{D}_{\tau}=\mathcal{D}_{\tau}^{-\infty}$ when $\tau< 0$.

 We use $E[\phi](\Si)$ to denote the energy flux of the complex scalar field $\phi$ and $E[F](\Si)$ the energy flux of the 2-form $F$
 through the hypersurface $\Si$ in Minkowski space. The derivative on the scalar field is with respect to the covariant derivative $D$. For our interested hypersurfaces, we can comlpute
\begin{align*}
 E[\phi](\mathbb{R}^3)&=\int_{\mathbb{R}^3}|D\phi|^2dx,\quad E[F](\mathbb{R}^3)=\int_{\mathbb{R}^3}\rho^2+|\si|^2+\frac{1}{2}(|\a|^2+|\underline{\a}|^2)dx,\\
 E[\phi](H_u)&=\int_{H_u}(|D_L\phi|^2+|\D\phi|^2)r^2dvd\om,\quad E[F](H_u)=\int_{H_u}(\rho^2+\si^2+|\a|^2) r^2dvd\om,\\
 E[\phi](\underline{H}_{\underline{u}})&=\int_{\underline{H}_{\underline{u}}}(|D_{\Lb}\phi|^2+|\D\phi|^2)r^2dvd\om,\quad E[F](\Hb_{\ub})=\int_{S_\tau}(\rho^2+\si^2+|\ab|^2) r^2dvd\om.
\end{align*}
Here $\rho$, $\si$, $\a$, $\ab$ are the null components of the 2-form $F$ defined in line \eqref{eq:curNull}.

Next we define weighted Sobolev norm either on domains or on surfaces. For any function $f$ (scalar or vector valued or tensors) we denote the spacetime integral on $\mathcal{D}$ in Minkowski space
\[
 I^{p}_{q}[f](\mathcal{D}):=\int_{\mathcal{D}}u_+^{q}r_+^{p}|f|^2,\quad r_+=1+r,\quad u_+=1+|u|.
\]
for any reall numbers $p$, $q$.
Here $\mathcal{D}$ can be the domain or hypersurface in the Minkowski space. For example, when $\mathcal{D}$ is $H_{u}$, then
\[
 I^{p}_{q}[f](H_u):=\int_{H_u} r_+^{p}u_+^{q}|f|^2 r^2 dvd\om.
\]
To define the norms of the derivatives of the solution, we need vector fields used as commutators which, in this paper,
are the killing vector field $\pa_t$ together with the angular momentum $\Om$ with components $\Om_{ij}=x_i\pa_j-x_j\pa_i$. We denote the set
\[
Z=\{\pa_t, \Om_{ij}\}.
\]
For the scalar field, it is nature to take the
covariant derivative $D_Z$ associated to the connection $A$ for any vector field $Z$. This covariant derivative has already been defined for the purpose of defining the equations in the beginning of the introduction. For the Maxwell field $F$ which is a 2-form, we define the Lie derivative
\begin{align*}
 (\mathcal{L}_Z F)_{\mu\nu}=Z(F_{\mu\nu})-F(\mathcal{L}_Z \pa_{\mu}, \pa_\nu)-F(\pa_\mu, \mathcal{L}_{Z}\pa_\nu), \quad (\mathcal{L}_Z J)_\mu= Z(J_\mu)-J(\mathcal{L}_Z \pa_\mu)
\end{align*}
respectively for any two form $F$ and any one form $J$. Here $\mathcal{L}_Z X=[Z, X]$.

If the vector field $Z$ is killing, that is, $\pa^\mu Z^\nu+\pa^\nu Z^\mu=0$,
then we can show that
\begin{align*}
 \pa^\mu(\mathcal{L}_Z F)_{\mu\nu}&=Z(\pa^\mu F_{\mu\nu})+\pa^\mu Z^\a\pa_\a F_{\mu\nu}+\pa_\mu Z^\a\pa^\mu F_{\a\nu}+\pa_\nu Z^\a\pa^\mu F_{\mu\a}\\
&=Z(\pa^\mu F_{\mu\nu})+\pa_\nu Z^\a\pa^\mu F_{\mu\a}=(\mathcal{L}_Z \delta F)_\nu.
\end{align*}
Here $\delta F_{\nu}=\pa^\mu F_{\mu\nu}$ is a one form which is the divergence of the $2$-form $F$. We use $\mathcal{L}_Z^k$ or $D^k_Z$ to denote the $k$-th derivatives, that is ,
\[
 \mathcal{L}_Z^k=\mathcal{L}_{Z^{1}}\mathcal{L}_{Z^{2}}\ldots\mathcal{L}_{Z^k}.
\]
Similarly for $D_Z^k$. The vector fields $Z^j$ is $\pa_t$ or the angular momentum $\Om_{ij}$.

Based on these calculations, we have the following commutator lemma:
\begin{lem}
 \label{lem:commutator}
For any killing vector field $Z$, we have
\begin{align*}
[\Box_A, D_Z]\phi&=2i Z^\nu F_{\mu \nu}D^\mu \phi+i \pa^\mu(Z^\nu F_{\mu\nu})\phi,\\
\pa^\mu(\mathcal{L}_Z G)_{\mu\nu}&=(\mathcal{L}_Z \delta G)_\nu
\end{align*}
for any complex scalar field $\phi$ and any two form $G$.
\end{lem}

For the energy estimates of the solutions of \eqref{EQMKG}, the initial energies $\mathcal{M}$, $\mathcal{E}$ defined in the introduction can not be used directly as $\mathcal{E}$ is not gauge invariant. Note that the vector fields used as commutators are $Z=\{\pa_t, \Om\}$, for any two form $F$ satisfying the Bianchi identity and any scalar field $\phi$, for the given connection field $A$, we define the following weighted $k$-th order initial energies:
\begin{equation}
\label{eq:def4E0kFphi}
 E^k_0[F]:=\sum\limits_{l\leq k}\int_{\mathbb{R}^3}r_+^{1+\ga_0}|\mathcal{L}_{Z}^l F|^2(0, x) dx,\quad E^k_0[\phi]:=\sum\limits_{l\leq k, j\leq 3}\int_{\mathbb{R}^3}r_+^{1+\ga_0}|D_{Z}^l D_j\phi|^2(0, x) dx.
\end{equation}
Here $D_j$ denotes the spatial covariant derivative and $0<\ga_0<1$ is the constant in the main Theorem. We remark here that $F$ may not be the full Maxwell field of the solution of \eqref{EQMKG}. In application, it can be the chargeless part of the full solution. However, the connection field $A$ is associated to the full Maxwell field. In fact the full Maxwell field does not belong to this weighted Sobolev space due to the existence of nonzero charge.

We end this section by writing the Maxwell equation under the null frame $\{L, \Lb, e_1, e_2\}$. In other words, we derive the transport equations for the curvature components. Let $F_{\mu\nu}$ be the two-form verifying the Bianchi identity. Denote $J=\delta F$, that is, $J_{\mu}=\pa^\nu F_{\mu \nu}$. Then we have
\begin{lem}l
 \label{lem:nullMKG}
 Under the null frame $\{L, \Lb, e_1, e_2\}$, the MKG equations are the following transport equations for the curvature components:
 \begin{align}
\label{eq:eq4rhoCu}
&\Lb(r^2\rho)-\divs(r^2\ab)=r^2 J_{\Lb}, \quad L(r^2\rho)+\divs(r^2\a)=r^2 J_L,\\
\label{eq:eq4ab}
&\nabla_{L}(r{\ab}_i)-r\nabb_{e_i}\rho-r\nabb_{e_j}F_{e_ie_j}=rJ_{e_i},\quad i=1, 2,\\
\label{eq:eq4si}
&\Lb(r^2\si)=r^2 (e_2\ab_1-e_1\ab_2),\quad L(r^2\si)=r^2 (e_2\a_1-e_1\a_2),\\
\label{eq:eq4a}
&\nabla_{\Lb}(r{\a}_i)+r\nabb_{e_i}\rho-r\nabb_{e_j}F_{e_ie_j}=rJ_{e_i},\quad i=1, 2.
\end{align}
Here $\divs$ is the divergence operator on the sphere with radius $r$.
\end{lem}
\begin{proof}
From the Maxwell equation $J_L=(\delta F)(L)$. Use the formula
\[
(\nabla_{X}F)(Y, Z)=XF(Y, Z)-F(\nabla_{X}Y, Z)-F(Y,  \nabla_{X}Z)
\]
for all vector fields $X$, $Y$, $Z$. By using \eqref{eq:nullderiv}, we then can compute
\begin{align*}
-(\delta F)(\Lb)&=-\f12(\nabla_{L}F)(\Lb, \Lb)-\f12 (\nabla_{\Lb}F)(L, \Lb)+(\nabla_{e_i}F)(e_i, \Lb)\\
&=\Lb\rho-e_i\ab_i-F(-2r^{-1}\pa_r, \Lb)-F(e_i, -r^{-1}e_i)\\
&=\Lb\rho-2r^{-1}\rho-\divs(\ab).
\end{align*}
Multiple both sides by $r^2$. We then get the first equation of \eqref{eq:eq4rhoCu}. The second equation follows similarly.

For \eqref{eq:eq4ab} and \eqref{eq:eq4a}, we need to use the Bianchi identity \eqref{bianchi} which is equivalent to
\[
(\nabla_X F)(Y, Z)+(\nabla_Y F)(Z, X)+(\nabla_Z F)(X, Y)=0
\]
for all vector fields $X$, $Y$, $ Z$. Let's only prove \eqref{eq:eq4ab}. We can show that
\begin{align*}
-(\delta F)(e_i)&=-\f12(\nabla_{L}F)(\Lb, e_i)-\f12 (\nabla_{\Lb}F)(L, e_i)+(\nabla_{e_j}F)(e_j, e_i)\\
&=-\f12 L\ab_i+\f12 (\nabla_{L}F)(e_i, \Lb)+\f12(\nabla_{e_i}F)(\Lb, L)+e_j F_{e_je_i}-F(-2r^{-1}\pa_r, e_i)-F(e_i, -r^{-1}\pa_r)\\
&=-L\ab_i+e_i\rho-\f12 F(-r^{-1}e_i, L)-\f12 F(\Lb, r^{-1}e_i)+e_j F_{e_je_i}+r^{-1}F(\pa_r, e_i)\\
&=-L\ab_i+e_i\rho+e_j F_{e_je_i}-r^{-1}\ab_i.
\end{align*}
This leads to \eqref{eq:eq4ab}.

The first transport equation \eqref{eq:eq4si} for $\si$ follows from the Bianchi identity:
\begin{align*}
0&=(\nabla_{\Lb} F)(e_1, e_2)+(\nabla_{e_1}F)(e_2, \Lb)+(\nabla_{e_2}F)(\Lb, e_1)\\
&=\Lb\si-e_1\ab_2-F(e_2, -r^{-1}e_1)+e_2\ab_1-F(-r^{-1}e_2, e_1)\\
&=\Lb\si-e_1\ab_2+e_2\ab_1-2r^{-1}\si.
\end{align*}
The dual one follows if we replace $\Lb$ with $L$.
\end{proof}

\section{Energy method}
\label{sec:energyID}

In this section, we review the energy method for solutions of the covariant linear wave equations and Maxwell equations by using the new method developed in  \cite{yangILEMKG}.
Denote $d\vol$ the
volume form in the Minkowski space. In the local coordinate system $(t, x)$, we have $
d\vol=dxdt
$. Here we have chosen $t$ to be the time orientation.
\subsection{Eergy identity for the scalar field}
For any complex scalar field $\phi$, we define the associated energy momentum tensor
\begin{equation*}
\begin{split}
 T[\phi]_{\mu\nu}=\Re\left(\overline{D_\mu\phi}D_\nu\phi\right)-\f12 m_{\mu\nu}\overline{D^\ga\phi}D_\ga\phi.
\end{split}
\end{equation*}
Here $m_{\mu\nu}$ is the flat metric of Minkowski spacetime and the covariant derivative $D$ is defined with respect to the given connection field $A$. For any vector field $X$, we have the following identity
\[
\pa^\mu(T[\phi]_{\mu\nu}X^\nu) = \Re(\Box_A \phi X^\nu \overline{D_\nu\phi})+X^\nu F_{\nu\ga}J^\ga[\phi]+T[\phi]^{\mu\nu}\pi^X_{\mu\nu},
\]
where $\pi_{\mu\nu}^X=\f12 \mathcal{L}_X m_{\mu\nu}$ is the deformation tensor of the vector field $X$ in Minkowski space, $\Box_A$ is the covariant wave operator associated to the connection $A$, $F=dA$ is the exterior derivative
of the one-form $A$ which gives us a two-form and $J^\ga[\phi]=Im(\phi\cdot \overline{D^\ga \phi})$.
For any function $\chi$, we have
\begin{align*}
 \f12\pa^{\mu}\left(\chi \pa_\mu|\phi|^2-\pa_\mu\chi|\phi|^2\right)= \chi \overline{D_\mu\phi}D^\mu\phi -\f12\Box\chi\cdot|\phi|^2+\chi \Re(\Box_A\phi\cdot \overline\phi).
\end{align*}
We now define the vector field $\tilde{J}^X[\phi]$ with components
\begin{equation}
\label{eq:mcurent}
\tilde{J}^X_\mu[\phi]=T[\phi]_{\mu\nu}X^\nu -
\f12\pa_{\mu}\chi \cdot|\phi|^2 + \f12 \chi\pa_{\mu}|\phi|^2+Y_\mu
\end{equation}
for some vector field $Y$ which may also depend on the complex scalar field $\phi$. We then have the equality
\[
\pa^\mu \tilde{J}^X_\mu[\phi] =\Re(\Box_A \phi(\overline{D_X\phi}+\chi\overline \phi))+div(Y)+X^\nu F_{\nu\a}J^\a[\phi]+T[\phi]^{\mu\nu}\pi^X_{\mu\nu}+
\chi \overline{D_\mu\phi}D^\mu\phi -\f12\Box\chi\cdot|\phi|^2.
\]
Here the operator $\Box$ is the wave operator in Minkowski space. Now for any region $\mathcal{D}$ in $\mathbb{R}^{3+1}$, using Stokes' formula, we derive the following energy identity:
\begin{align}
\notag &\iint_{\mathcal{D}}\Re(\Box_A \phi(\overline{D_X\phi}+\chi\overline \phi))+div(Y)+X^\nu F_{\nu\ga}J^\ga+T[\phi]^{\mu\nu}\pi^X_{\mu\nu}+
\chi \overline{D_\mu\phi}D^\mu\phi -\f12\Box\chi\cdot|\phi|^2d\vol\\
&=\iint_{\mathcal{D}}\pa^\mu \tilde{J}^X_\mu[\phi]d\vol=\int_{\pa \mathcal{D}}i_{\tilde{J}^X[\phi]}d\vol,
\label{energyeq}
\end{align}
where $\pa\mathcal{D}$ denotes the boundary of the domain $\mathcal{D}$ and $i_Z d\vol$ denotes the contraction of the volume form $d\vol$
with the vector field $Z$ which gives the surface measure of the
boundary. For example, for any basis $\{e_1, e_2,
\ldots, e_n\}$, we have
$$i_{e_1}( de_1\wedge de_2\wedge\ldots
de_k)=de_2\wedge de_3\wedge\ldots\wedge de_k.
$$
Throughout this paper, the domain $\mathcal{D}$ will be regular regions bounded by the $t$-constant slices, the outgoing null hypersurfaces $u=constant$, the incoming null
hypersurfaces $v=constant$ or the surface with constant $r$. We now compute $i_{\tilde{J}^{X}[\phi]}d\vol$ respectively on these hypersurfaces.

On
$t=constant$ slice, the surface measure is a function times $dx$. Recall the volume form
\[
d\vol=dxdt=-dtdx.
\]
Here note that $dx$ is a $3$-form. We thus can show that
\begin{equation}
\label{eq:curlessR}
\begin{split}
 i_{\tilde{J}^{X}[\phi]}d\vol&=-(\tilde{J}^{X}[\phi])^0dx=-(\Re(\overline{D^t\phi}
D_X\phi)-\f12 X^0\overline{D^\ga\phi}D_\ga\phi-\f12 \pa^t\chi
|\phi|^2+\f12\chi\pa^t|\phi|^2+Y^0l)dx.
\end{split}
\end{equation}
On the surface with constant $r$, the surface measure is $r^2dtd\om$. Therefore we have
\begin{equation}
\label{eq:currcon}
i_{\tilde{J}^{X}[\phi]}d\vol=(\Re(\overline{D^{r}\phi}
D_X\phi)-\f12 X^{r}\overline{D^\ga\phi}D_\ga\phi-\f12
\pa^{r}\chi|\phi|^2+\f12\chi \pa^{r}|\phi|^2+Y^{r})r^2dt d\om.
\end{equation}
On the outgoing null hypersurface $H_u$, we can write the volume form
\[
d\vol=dxdt=r^2drdt d\om=2r^2dvdud\om=-2dudvd\om.
\]
Here $d\om$ is the standard surface measure on the unite sphere.
Notice that $\Lb=\pa_u$. We can compute
\begin{equation}
\label{eq:curStau}
i_{\tilde{J}^{X}[\phi]}d\vol=-2(\Re(\overline{D^{\Lb}\phi}
D_X\phi)-\f12 X^{\Lb}\overline{D^\ga\phi}D_\ga\phi-\f12
\pa^{\Lb}\chi|\phi|^2+\f12\chi \pa^{\Lb}|\phi|^2+Y^{\Lb}l)r^2dvd\om.
\end{equation}
Similarly, on the $v$-constant incoming null hypersurfaces $\Hb_{\ub}$, we
have
\begin{equation}
\label{eq:curnullinfy}
i_{\tilde{J}^{X}[\phi]}d\vol=2(\Re(\overline{D^{L}\phi}
D_X\phi)-\f12 X^{L}\overline{D^\ga\phi}D_\ga\phi-\f12
\pa^{L}\chi|\phi|^2+\f12\chi \pa^{L}|\phi|^2+Y^{L})r^2dud\om.
\end{equation}
We remark here that the above formulae hold for any vector fields $X$, $Y$ and any function $\chi$.

\subsection{Energy identities for the Maxwell field}
Let $F$ be any 2-form satisfying the Bianchi identity \eqref{bianchi}. The associated the energy momentum tensor is
\[
 T[F]_{\mu\nu}=F_{\mu\ga}F_{\nu}^{\;\ga}-\frac{1}{4}m_{\mu\nu}F_{\ga\b}F^{\ga\b}.
\]
For any vector field $X$, we have the divergence formula
\[
 \pa^\mu{T[F]_{\mu\nu}X^\nu}=\pa^\mu F_{\mu\ga}F_{\nu}^{\;\ga}X^\nu+T[F]^{\mu\nu}\pi^X_{\mu\nu},
\]
where as defined previously, $\pi_{\mu\nu}^X=\f12 \mathcal{L}_X m_{\mu\nu}$ is the deformation tensor of the vector field $X$ in Minkowski space.
Define the vector field $J^X[F]$ as follows:
\[
 J^X[F]_\mu=T[F]_{\mu\nu}X^\nu.
\]
Then for any domain $\mathcal{D}$ in $\mathbb{R}^{3+1}$, we have the following energy identity for the Maxwell field $F$:
\begin{align}
\iint_{\mathcal{D}}\pa^\mu F_{\mu\a}F_{\nu}^{\;\a}X^\nu+T[F]^{\mu\nu}\pi^X_{\mu\nu}d\vol=\iint_{\mathcal{D}}\pa^\mu J^X_\mu[F]d\vol=\int_{\pa \mathcal{D}}i_{J^X[F]}d\vol.
\label{energyeqcur}
\end{align}
For the terms on the boundary, similar to \eqref{eq:curlessR} to \eqref{eq:curnullinfy}, we can compute
\begin{equation}
\label{eq:curF}
\begin{split}
 i_{J^{X}[F]}d\vol&=-(F^{0\mu}F_{\nu\mu}X^\nu-\frac{1}{4}X^0 F_{\mu\nu}F^{\mu\nu})dx;\\
i_{J^{X}[F]}d\vol&=(F^{r\mu}F_{\nu\mu}X^\nu-\frac{1}{4}X^r F_{\mu\nu}F^{\mu\nu})r^2dtd\om;\\
i_{J^{X}[F]}d\vol&=-2(F^{\Lb \mu}F_{\nu\mu}X^\nu-\frac{1}{4}X^{\Lb}F_{\mu\nu}F^{\mu\nu})r^2dvd\om;\\
i_{J^{X}[F]}d\vol&=2( F^{L \mu}F_{\nu\mu}X^\nu-\frac{1}{4}X^{L}F_{\mu\nu}F^{\mu\nu})r^2dud\om
\end{split}
\end{equation}
respectively on the $t=constant$ slice, surface with constant $r$, the outgoing null hypersurface $H_u$ and the incoming null hypersurface $\Hb_{\ub}$.

\subsection{The integrated local energy estimates using the multiplier $f(r)\pa_r$}
For the full solution $(\phi, F)$ of the Maxwell Klein-Gordon equations, including the case with large charge, the integrated local energy estimates together with the $r$-weighted energy estimates
in the next subsection have been studied in the author's work \cite{yangILEMKG}. To obtain estimates for higher order derivatives of the solutions, we need to commute the equations with derivatives and
nonlinear terms hence arise. Furthermore, in our setting, the data for the Maxwell field are large while the data for the complex scalar field are small. We thus need to obtain estimates separately for the
Maxwell field and the scalar field.

We first consider the integrated local energy estimates for the scalar field. In the energy identity \eqref{energyeq} for the scalar field,
we choose the vector fields
$$X=f(r)\pa_r,\quad Y=0$$
for some function $f(r)$. We then can compute
\begin{align*}
 &T[\phi]^{\mu\nu}\pi^X_{\mu\nu}+ \chi \overline{D_\mu\phi}D^\mu\phi-\f12 \Box\chi |\phi|^2\\
=&(r^{-1}f-\chi+\f12 f')|D_t\phi|^2+(\chi+\f12 f'-r^{-1}f)|D_r\phi|^2+(\chi-\f12 f')|\D\phi|^2-\f12 \Box\chi |\phi|^2.
\end{align*}
The idea is to choose the functions $f$, $\chi$ so that the coefficients are positive. Let $\ep$ be a small positive constant, depending only on $\ga_0$. Construct the functions $f$ and $\chi $ as follows:
$$f=2\ep^{-1}-\frac{2\ep^{-1}}{(1+r)^{\ep}},\quad \chi=r^{-1}f.$$
We can compute
\begin{align*}
 \chi-r^{-1}f + \f12 f'=r^{-1}f + \f12 f' - \chi=\frac{1}{(1+r)^{1+\ep}},\quad-\f12 \Box \chi=\frac{1+\ep}{r(1+r)^{2+\ep}}.
\end{align*}
When $r>1$, we have the following improved estimate for
$\chi-\f12 f'$
\begin{equation}\label{improvnabb}\chi- \f12
f'\geq \frac{2\ep^{-1}}{r} - \frac{1+2\ep^{-1}}{r(1+r)^\ep}\geq \frac{1}{r},
\quad r\geq 1.
\end{equation}
This improved estimate will be used to show the improved integrated local
energy estimate for the covariant angular derivative of the
scalar field $\phi$.

From the above calculation, we see that for this particular choice of vector field $X$ and the function $\chi$, the last three terms in the first line of equation \eqref{energyeq} have positive signs. We
treat the first two terms as nonlinear terms. To get an integrated local energy estimate for the scalar field $\phi$, it the suffices to control the boundary terms arising from the Stokes' formula \eqref{energyeq}. This  requires a version of Hardy's inequality. Before stating the lemma, we make a convention
that the notation $A\les_{\Gamma} B$ means that there exists a constant $C$, depending only on the constants $R$, $\ga_0$, $\ep$ and the set $\Gamma$ such that $A\leq CB$. For the particular case when $\Gamma$ is empty, we omit the index $\Gamma$.
\begin{lem}
 \label{lem:Hardyga}
 Assume $0\leq\ga<1$ and the complex scalar field $\phi$ vanishes at null infinity, that is,
\[
 \lim_{v\rightarrow\infty}\phi(v, u, \om)=0
\]
for all $u$, $\om$. Then we have
 \begin{equation}
  \label{eq:Hardyga}
  \int_{H_u}r^{\ga}|\phi|^2dvd\om \les \int_{\om}r^{1+\ga}|\phi|^2(u, v(u), \om)d\om  +\int_{H_u} r^{\ga}|D_L(r\phi)|^2dvd\om
 \end{equation}
 for all $u\in\mathbb{R}$. Here $v(u)=-u$ when $u\leq -\frac{R}{2}$ that is in the exterior region and $v(u)=2R+u$ when $u>-\frac{R}{2}$ that is in the interior region. In particular, we have
 \begin{equation}
 \label{eq:Hardy}
 \int_{H_u}|\phi|^2dvd\om\les E[\phi](H_u),\quad \int_{\Si_{\tau}}|\phi|^2dv'd\om \les E[\phi](\Si_{\tau}).
 \end{equation}
 Here $v'=v$ when $r\geq R$ and $v'=r$ otherwise.
\end{lem}
\begin{proof}
It suffices to notice that the connection $D$ is compatible with the inner product $<,>$ on the complex plane. Then the proof when $\ga=0$ goes the same as the case when the connection field $A$ is trivial,
see e.g. Lemma 2 of \cite{yang1} or Proposition 11.2 of \cite{dr3}. Another quick way to reduce the proof of the Lemma to the case with trivial connection field $A$ is to choose a particular gauge such that the scalar field
$\phi$ is real. We can do this is due to the fact that all the norms we considered in this paper are gauge invariant.
For general $\ga$, based on the above argument, the proof goes similar to the proof of the standard Hardy's inequality. Let $\psi=r\phi$. Note that $\ga <1$. We can show that
 \begin{align*}
  \int_{v_0}^{\infty}\int_{\om }r^{\ga-2}|\psi|^2dvd\om &=\frac{1}{\ga-1}\int_{v_0}^{\infty}\int_{\om}|\psi|^2d\om d r^{\ga -1}\\
  &=\frac{1}{\ga-1}\left. r^{\ga -1}\int_{\om }|\psi|^2d\om \right|_{v_0}^{\infty}+\frac{2}{1-\ga}\int_{v_0}^{\infty}\int_{\om }r^{\ga -1}D_L\psi \cdot \psi dv d\om \\
  &\leq \frac{1}{1-\ga} \int_{\om }r^{1+\ga}|\phi|^2(u, v_0, \om ) d\om +\frac{1}{2}\int_{v_0}^{\infty}\int_{\om}r^{\ga-2}|\psi|^2dvd\om \\
  &\quad \quad +\frac{8}{(1-\ga)^2}\int_{v_0}^{\infty}\int_{\om}r^{\ga}|D_L\psi|^2dvd\om.
 \end{align*}
Estimate \eqref{eq:Hardyga} then follows by absorbing the second term and taking $v_0=v(u)$.
\end{proof}

We then can derive the following integrated local energy estimate for the scalar field $\phi$.
\begin{prop}
 \label{prop:ILEs}
Assume the complex scalar field $\phi$ vanishes at null infinity and the spatial infinity initially. Then in the interior region $\{r\leq R+t\}$, we have the following energy estimates:
\begin{equation}
 \label{eq:ILE:Sca:in}
\begin{split}
& I^{-1-\ep}_0[\bar D\phi](\mathcal{D}_{\tau_1}^{\tau_2})+E[\phi](\Si_{\tau_2})+E[\phi](\Hb^{\tau_1, \tau_2}_v)+\iint_{\mathcal{D}_{\tau_1}^ {\tau_2}}\frac{|\D\phi|^2}{1+r} dxdt\\
&\les E[\phi](\Si_{\tau_1})+I^{1+\ep}_0[\Box_A\phi](\mathcal D_{\tau_1}^{\tau_2})+
\iint_{\mathcal{D}_{\tau_1}^{\tau_2}}|F_{L\nu}J^\nu[\phi]|+|F_{\Lb\nu}J^\nu[\phi]|dxdt
\end{split}
\end{equation}
for all $0\leq \tau_1<\tau_2$, $v\geq\frac{\tau_2+R}{2}$, where we denote $\bar D\phi=(D\phi, r_+^{-1}\phi)$ and $F=dA$. Similarly, in the exterior region $\{r>t+R\}$, we have
\begin{equation}
 \label{eq:ILE:Sca:ex}
\begin{split}
& I^{-1-\ep}_0[\bar D\phi](\mathcal D_{\tau_1}^{\tau_2})+E[\phi](H_{\tau_1^*}^{-\tau_2^*})+E[\phi](\Hb_{-\tau_2^*}^{\tau_1^*})\\
&\les E[\phi](B_{R-\tau_1})+I^{1+\ep}_0[\Box_A\phi](\mathcal D_{\tau_1}^{\tau_2})+
\iint_{\mathcal{D}_{\tau_1}^{ \tau_2}}|F_{L\nu}J^\nu[\phi]|+|F_{\Lb\nu}J^\nu(\phi)| dxdt
\end{split}
\end{equation}
for all $\tau_2\leq \tau_1\leq 0$. Here the notations have been defined in Section \ref{sec:notation} and $J^\mu[\phi]=\Im(\phi\cdot \overline{D^\mu \phi})$.
\end{prop}
\begin{proof}
For all $v_0\geq \frac{\tau_2+R}{2}$, take the region $\mathcal{D}$ to be $\mathcal{D}_{\tau_1}^{\tau_2}\cap\{v\leq v_0\}$ which is bounded by the surfaces $\Si_{\tau_1}$, $\Si_{\tau_2}$
, $\Hb_{v_0}^{\tau_1, \tau_2}$ and the functions $f$, $\chi$ as above and the vector field $Y=0$ in the energy identity \eqref{energyeq}.
The boundary terms can be controled by the energy flux according to Hardy's inequality of Lemma \ref{lem:Hardyga}. For more details regarding this bound, we refer to e.g. Proposition 1 in \cite{yang1}.
Therefore the above calculations lead to the following integrated local energy estimate:
\begin{align*}
 &\iint_{\mathcal{D}_{\tau_1}^{\tau_2}\cap\{v\leq v_0\}}\frac{|D\phi|^2}{(1+r)^{1+\ep}}+\frac{|\D\phi|^2}{1+r}+\frac{|\phi|^2}{r(1+r)^{2+\ep}}dxdt\\
&\les E[\phi](\Si_{\tau_1}^{ v_0})+E[\phi](\Si_{\tau_2}^{v_0})+E[\phi](\Hb^{\tau_1, \tau_2}_{v_0})
+\iint_{\mathcal{D}_{\tau_1}^{\tau_2}}|\Box_A \phi (\overline{D_X\phi}+\chi\overline{\phi})|+|F_{r\nu}J^\nu[\phi]|dxdt
\end{align*}
Next, we take the vector fields $X=\pa_t$, $Y=0$ and the function $\chi=0$ in the energy identity \eqref{energyeq} for the scalar field. Consider the region
$\mathcal{D}_{\tau_1}^{\tau_2}\cap\{v\leq v_0\}$. We retrieve the classical energy estimate
\begin{equation*}
 E[\phi](\Si_{\tau_2}^{v_0})+E[\phi](\Hb_{v_0}^{\tau_1, \tau_2})=E[\phi](\Si_{\tau_1}^{v_0})-2\iint_{\mathcal{D}_{\tau_1}^{\tau_2}\cap\{v\leq v_0\}}\Re\left(\Box_A\phi \overline{D_t\phi}\right)+F_{0\nu}J^\nu[\phi] dxdt.
\end{equation*}
Combined with the previous integrated local energy estimate and letting $v_0\rightarrow \infty$, we derive that
\begin{align*}
 I^{-1-\ep}_0[\bar D\phi](\mathcal{D}_{\tau_1}^{\tau_2})\les E[\phi](\Si_{\tau_1})+\iint_{\mathcal{D}_{\tau_1}^{\tau_2}}|\Box_A\phi \overline{\bar D\phi}|+|F_{L\nu}J^\nu[\phi]|+|F_{\Lb\nu}J^\nu[\phi]| dxdt.
\end{align*}
We apply Cauchy-Schwarz's inequality to the integral of $\Box_A\phi \overline{\bar D\phi}$:
\[
 2|\Box_A\phi \overline{\bar D\phi}|\leq \ep_1 (1+r)^{-1-\ep}|\bar D\phi|^2+\ep_1^{-1}(1+r)^{1+\ep}|\Box_A\phi|^2,\quad \forall \ep_1>0.
\]
Choose $\ep_1$ to be sufficiently small depending only on $\ep$, $\ga_0$, $R$ so that the integral of the first term can be absorbed. We thus can derive the integrated local
energy estimate for the scalar field. Then in the above classical energy estimate, we can use Cauchy-Schwarz's inequality again to bound $\Re\left(\Box_A\phi \overline{D_t\phi}\right)$ which gives
control of the energy flux $E[\phi](H_{\tau_2^*})$. This energy estimate together with the previous integrated local energy estimate imply the energy estimate \eqref{eq:ILE:Sca:in} of the Proposition in the interior region.
The improved estimate for the angular covariant derivative is due to the improve estimate \eqref{improvnabb}.

The proof for the estimate \eqref{eq:ILE:Sca:ex} in the exterior region is similar. The only point we need to emphasize is that we use the fact that the
$\phi$ goes to zero as $r\rightarrow \infty$ on the initial
hypersurface. We thus can use the Hardy's inequality to control the integral of $\frac{|\phi|^2}{(1+r)^{2}}$. This is also the reason that we have $E[\phi](B_{R-\tau_1})$
instead of $E[\phi](B_{R-\tau_1}^{R-\tau_2})$ on the right hand side of \eqref{eq:ILE:Sca:ex}.
\end{proof}
In our setting, $F$ is the Maxwell field which is no longer small. In particular this means that the integral of $|F_{L\nu}J^\nu[\phi]|$ on the right hand side of \eqref{eq:ILE:Sca:in}, \eqref{eq:ILE:Sca:ex} could not be absorbed. The key to control those terms is to use the $r$-weighted energy estimates in the next section.

Let $F$ be any 2-form satisfying the Bianchi identity \eqref{bianchi}. Denote $J=\delta F$ or $J_{\mu}=\pa^\nu F_{\nu\mu}$ be the divergence of $F$. This notation $J$ can be viewed as inhomogeneous term of the linear Maxwell equation. In \eqref{EQMKG}, this $J$ is identity to $J[\phi]$ which is quadratic in the scalar field $\phi$. Under the null frame $\{L, \Lb, e_1, e_2\}$, ldenote $\J=(J_{e_1}, J_{e_2})$. We derive an analogue of Proposition \ref{prop:ILEs}.
\begin{prop}
 \label{prop:ILEcur}
Then in the interior region $\{r\leq t+R\}$, we have the integrated local energy estimates
\begin{equation}
 \label{eq:ILE:cur:in}
\begin{split}
&I^{-1-\ep}_0[F](\mathcal{D}_{\tau_1}^{\tau_2})+ \int_{\tau_1}^{\tau_2}\int_{\Si_{\tau}}\frac{\rho^2+\si^2}{1+r} dxd\tau+E[F](\Hb_{v_0}^{\tau_1,\tau_2})+E[F](\Si_{\tau_2})\\
\les &  E[F](\Si_{\tau_1})+I^{1+\ep}_0[|J_{L}|+|\J|](\mathcal{D}_{\tau_1}^{\tau_2})+\iint_{\mathcal{D}_{\tau_1}^{\tau_2}}|J_{\Lb}||\rho|dxdt.
\end{split}
\end{equation}
for all $0\leq \tau_1<\tau_2$, $v_0\geq \frac{\tau_2+R}{2}$. In the exterior region $\{r\leq R+t\}$ for all $\tau_2\leq \tau_1\leq 0$, we have
\begin{equation}
 \label{eq:ILE:cur:ex}
\begin{split}
&I^{-1-\ep}_0[F](\mathcal{D}_{\tau_1}^{\tau_2})+E[F](\Hb_{-\tau_2^*}^{\tau_1^*})+E[F](H_{\tau_1^*}^{-\tau_2^*})\\
&\les  E[F](B_{R-\tau_1}^{R-\tau_2})+I^{1+\ep}_0[|J_{L}|+|\J|](\mathcal{D}_{\tau_1}^{\tau_2})+\iint_{\mathcal{D}_{\tau_1}^{\tau_2}}|J_{\Lb}||\rho|dxdt.
\end{split}
\end{equation}
\end{prop}
 \begin{proof}
 The idea to prove this proposition is the same as that of the previous proposition for the scalar field. However, the calculations is different for the Maxwell field $F$.
In the energy identity \eqref{energyeqcur} for the Maxwell field, we take the vector field
\[
 X=f(r)\pa_r=2\ep^{-1}(1-r_+^{-\ep})\pa_r.
\]
Denote $\om_i=r^{-1}x_i$. We then can compute
\begin{align*}
 T[F]^{\mu\nu}\pi_{\mu\nu}^X&=T[F]^{ij}(f'\om_i\om_j+r^{-1}f\delta_{ij}-r^{-1}f\om_i\om_j)\\
&=\frac{1}{4}(2r^{-1}f-f')F_{\mu\nu}F^{\mu\nu}+(f'-r^{-1}f)F_{r\nu}F^{r\nu}-r^{-1}fF_{0\nu}F^{0\nu},
\end{align*}
where the Latin indices $\mu$, $\nu$ run from $0$ to $3$ and the Greek indices $i$, $j$ run from 1 to 3. Using the null decomposition of the 2-form under the null frame
$\{L, \Lb, e_1, e_2\}$ defined in line \eqref{eq:curNull}, we can show that
\begin{align*}
 F_{\mu\nu}F^{\mu\nu}&=-2 \rho^2-2\a\cdot \underline{\a}+2\si^2,\\
F_{0\nu}F^{0\nu}&=-\frac{1}{4}(4\rho^2+2\a\cdot \underline{\a}+|\a|^2+|\underline\a|^2),\\
F_{r\nu}F^{r\nu}&=-\frac{1}{4}(4\rho^2+2\a\cdot\underline{\a}-|\a|^2-|\underline{\a}|^2).
\end{align*}
Therefore we have
\begin{equation}
 \label{TFXr}
T[F]^{\mu\nu}\pi_{\mu\nu}^X=(r^{-1}f-\f12 f')(\rho^2+\si^2)+\frac{1}{4}f'(|\a|^2+|\underline\a|^2).
\end{equation}
The calculations before line \eqref{improvnabb} imply that the coefficients $r^{-1}f-\f12 f'$, $f'$ have positive signs.
To obtain the similar integrated local energy estimates for the Maxwell field $F$, we need to control the boundary terms arising from the Stokes' formula \eqref{energyeqcur}.
Using the formulae \eqref{eq:curF}, we can compute that
\begin{align*}
 2| i_{J^{X}[F]}d\vol|&= f\left||\a|^2-|\underline{\a}|^2\right|dx\leq |F|^2 dx=2f i_{J^{\pa_t}[F]}d\vol,\\
2|i_{J^{X}[\phi]}d\vol|&=f \left|-\rho^2+|\a|^2-\si^2\right|r^2dvd\om\leq f(\rho^2+|\a|^2+\si^2)r^2dvd\om=2fi_{J^{\pa_t}[\phi]}d\vol,\\
2|i_{J^{X}[\phi]}d\vol|&=f\left|-\rho^2+|\underline{\a}|^2-\si^2\right|r^2dud\om\leq f(\rho^2+|\underline{\a}|^2+\si^2)r^2dud\om=2fi_{J^{\pa_t}[\phi]}d\vol
\end{align*}
respectively on the $t=constant$ slice, the outgoing null hypersurface and the incoming null hypersurface for all positive function $f$. This in particular implies that
the boundary terms corresponding to the vector field $f\pa_r$ can be bounded by the energy flux for all positive bounded function $f$. Therefore for the particular choice of vector field $X$,
the energy identity \eqref{energyeqcur} on the domain $\mathcal{D}_{\tau_1}^{\tau_2}\cap\{v\leq v_0\}$ for all $0\leq \tau_1<\tau_2$, $v_0\geq \frac{\tau_2+R}{2}$ leads to
\begin{align*}
 \int_{\tau_1}^{\tau_2}\int_{\Si_{\tau}^{v_0}}\frac{|F|^2}{(1+r)^{1+\ep}}+\frac{\rho^2+\si^2}{1+r} dxd\tau\les & E[F](\Si_{\tau_1}^{v_0})+E[F](\Si_{\tau_2}^{v_0})+E[F](\Hb_{v_0}^{\tau_1, \tau_2})\\
&+\int_{\tau_1}^{\tau_2}\int_{\Si_{\tau}}|J^{\ga}||F_{L\ga}-F_{\Lb \ga})|dxd\tau.
\end{align*}
Here notice that we have the improved estimate \eqref{improvnabb} for the coefficient of $\rho^2+\si^2$. If we take the vector field $X=\pa_t$ on the same domain, we then can derive the classical energy identity
\begin{align*}
\int_{\tau_1}^{\tau_2}\int_{\Si_{\tau}^{v_0}}J^{\ga}(F_{L \ga}+F_{\Lb\ga})dxd\tau=E[F](\Si_{\tau_1}^{v_0})-E[F](\Hb_{v_0}^{\tau_1,\tau_2})-E[F](\Si_{\tau_2}^{v_0}).
\end{align*}
Let $v_0\rightarrow \infty$ and apply Cauchy-Schwarz to the inhomogeneous term $J^{\mu} (|F_{L\mu}|+|F_{\Lb \mu}|)$ for $\mu=\Lb$, $e_1$, $e_2$:
 \[
 |J^{\Lb}| |F_{L\Lb}|+|J^{e_i}| (|F_{L e_i}|+|F_{\Lb e_i}|)\les \ep_1^{-1}(|J_{L}|+|\J|)r_+^{1+\ep}+\ep_1 |F|^2r_+^{-1-\ep},\quad \ep_1>0.
 \]
 The integral of second term could be absorbed for sufficiently small $\ep_1$. For the component when $\mu=L$, we estimate:
 \[
 |J^{L}| |F_{L\Lb}|\les |J_{\Lb}||\rho|.
 \]
 Then the above energy identity together with the integrated local energy estimates imply the integrated local energy estimate \eqref{eq:ILE:cur:in} in the interior region. The energy estimate
\eqref{eq:ILE:cur:ex} in the exterior region follows in a same way.
 \end{proof}

\subsection{The $r$-weighted energy estimates using the multiplier $r^pL$}
In this section, we establish the robust $r$-weighted energy estimates both for the scalar field and the Maxwell field. This estimate for solutions of linear wave equation in Minkowski space
was first introduced by Dafermos-Rodnianski in \cite{newapp}. We study the $r$-weighted energy estimate either in the exterior region $\{r\geq R+t\}$ for the domain $\mathcal{D}_{\tau_1}^{\tau_2}$ for $\tau_2\leq \tau_1\leq 0$ or in the interior region for
domain $\bar{\mathcal{D}}_{\tau_1}^{\tau_2}$ for $0\leq \tau_1<\tau_2$ which is bounded by the outgoing null hypersurfaces $H_{\tau_1^*}$, $H_{\tau_2^*}$ and the cylinder $\{r=R\}$.

We have the following $r$-weighted energy estimates for the complex scalar field.
\begin{prop}
 \label{prop:pWE:sca}
 Assume that the complex scalar field $\phi$ vanishes at null infinity. Then in the interior region, for all $0\leq \tau_1\leq \tau_2$, $v_0\geq \frac{\tau_2+R}{2}$, we have the $r$-weighted energy estimate:
 \begin{equation}
  \label{eq:pWE:sca:in}
  \begin{split}
 &\int_{H_{\tau_2^*}}r^p|D_L\psi|^2dvd\om+\int_{\tau_1}^{\tau_2}\int_{H_{\tau^*}}r^{p-1}(p|D_L\psi|^2+(2-p)|\D\psi|^2)dvd\om d\tau+\int_{\Hb_{v_0}^{\tau_1, \tau_2}}r^p|\D\psi|^2dud\om\\
\les &\int_{H_{\tau_1^*}}r^p|D_L\psi|^2dvd\om+I^{\max\{1+\ep, p\}}_{\min\{1+\ep, p\}}[\Box_A\phi](\mathcal{D}_{\tau_1}^{\tau_2})+E[\phi](\Si_{\tau_1})+I^{1+\ep}_0[\Box_A\phi](\mathcal{D}_{\tau_1}^{\tau_2})\\
&\qquad+\iint_{\mathcal{D}_{\tau_1}^{\tau_2}}|F_{\Lb\mu}J^\mu[\phi]|+|F_{L\mu}J^\mu[\phi]|dxdt+
\iint_{\bar{\mathcal{D}}_{\tau_1}^{\tau_2}}r^{p}|F_{L\mu}J^\mu[\phi]|dxdt.
\end{split}
 \end{equation}
for all $0\leq p\leq 2$. Similarly in the exterior region, for all $\tau_2\leq \tau_1\leq 0$, we have
\begin{equation}
 \label{eq:pWE:sca:ex}
 \begin{split}
  &\int_{H_{\tau_1^*}^{-\tau_2^*}}r^p|D_L\psi|^2dvd\om+\iint_{\mathcal{D}_{\tau_1}^{\tau_2}}r^{p-1}(p|D_L\psi|^2+(2-p)|\D\psi|^2)dvd\om du+\int_{\Hb_{-\tau_2^*}^{\tau_1^*}}r^p|\D\psi|^2dud\om\\
\les & \int_{B_{R-\tau_1}^{R-\tau_2}}r^p(|D_L\psi|^2+|\D\psi|^2)drd\om+I^{\max\{p, 1+\ep\}}_{\min\{1+\ep, p\}}[\Box_A \phi](\mathcal{D}_{\tau_1}^{\tau_2})+\iint_{\mathcal{D}_{\tau_1}^{\tau_2}}r^{p}|F_{L\mu}J^\mu[\phi]| dxdt.
\end{split}
\end{equation}
\end{prop}
\begin{proof}
Apply the energy identity \eqref{energyeq} to the region $\bar{\mathcal{D}}_{\tau_1}^{\tau_2}\cap\{v\leq v_0\}$ which
bounded by $H_{\tau_1^*}$, $H_{\tau_2^*}$, $\{r=R\}$ and $\Hb_{v_0}^{\tau_1, \tau_2}$ with the vector fields $X$, $Y$ and the function $\chi$ as follows:
\[
 X=r^{p}L, \quad Y=\frac{p}{2}r^{p-2}|\phi|^2L,\quad  \chi=r^{p-1}.
\]
Denote $\psi=r\phi$ as the weighted scalar field. Note that we have the equality
\[
 r^2|D_L\phi|^2=|D_L\psi|^2-L(r|\phi|^2),\quad r^2|\D\phi|^2=|\D\psi|^2,\quad r^2|D_{\Lb}\phi|^2=|D_{\Lb}\psi|^2+\Lb(r|\phi|^2).
\]
We then can compute
\begin{align*}
 & div(Y)+T[\phi]^{\mu\nu}\pi_{\mu\nu}^X+\chi \overline{D^{\mu}\phi}D_\mu\phi-\f12\Box\chi |\phi|^2\\
&=\frac{p}{2}r^{-2}L(r^{p}|\phi|^2)+\f12r^{p-1}\left(p|D_L\phi|^2+(2-p)|\D\phi|^2\right)-\f12 p(p-1)r^{p-3}|\phi|^2\\
&=\f12 r^{p-3}\left(p|D_L\psi|^2+(2-p)|\D\psi|^2\right).
\end{align*}
We next compute the boundary terms according to the formula \eqref{eq:curF}. We have
\begin{align*}
 \int_{H_{\tau^*}}i_{\tilde{J}^X[\phi]}d\vol&=\int_{H_{\tau^*}}r^{p}|D_L\psi|^2-\f12 L(r^{p+1}\phi) \quad dvd\om,\\
\int_{\Hb_{v_0}^{\tau_1, \tau_2}}i_{\tilde{J}^X[\phi]}d\vol&=-\int_{\Hb_{v_0}^{\tau_1, \tau_2}}r^{p}|\D\psi|^2+\f12\Lb(r^{p+1}|\phi|^2)\quad dud\om,\\
\int_{\{r=R\}\cap\{ \tau_1\leq t\leq \tau_2\} }i_{\tilde{J}^X[\phi]}d\vol&=\int_{\tau_1}^{\tau_2}\int_{\om}\f12 r^p(|D_L\psi|^2-|\D\psi|^2)-\f12\pa_t(r^{p+1}|\phi|^2) \quad d\om dt.
\end{align*}
Now notice that there is a cancellation for the boundary terms:
\begin{align*}
 &-\int_{H_{\tau_1^*}}L(r^{p+1}|\phi|^2)dvd\om-\int_{\Hb_{v_0}^{\tau_1, \tau_2}}\Lb(r^{p+1}|\phi|^2)dud\om\\
&+\int_{H_{\tau_2^*}}L(r^{p+1}|\phi|^2)dvd\om+\int_{\tau_1}^{\tau_2}\int_{\om}\pa_t(r^{p+1}|\phi|^2)d\om dt=0.
\end{align*}
Therefore in the interior region for the domain $\bar{\mathcal{D}}_{\tau_1}^{\tau_2}\cap\{v\leq v_0\}$, the above calculations lead to the following $r$-weighted energy identity:
\begin{equation}
\label{eq:pWE:sca:in:id}
\begin{split}
 &\int_{H_{\tau_2^*}^{v_0}}r^p|D_L\psi|^2dvd\om+\int_{\tau_1}^{\tau_2}\int_{H_{\tau^*}^{v_0}}r^{p-1}(p|D_L\psi|^2+(2-p)|\D\psi|^2)dvd\om d\tau+\int_{\Hb_{v_0}^{\tau_1, \tau_2}}r^p|\D\psi|^2dud\om\\
=&\int_{H_{\tau_1^*}^{v_0}}r^p|D_L\psi|^2dvd\om-\f12\int_{\tau_1}^{\tau_2}\int_{\om}r^p(|D_L\psi|^2-|\D\psi|^2) d\om dt\\
&-\int_{\tau_1}^{\tau_2}\int_{H_{\tau^*}^{v_0}}r^{p-1}\Re\left(\Box_A\phi \overline{D_L\psi}\right)+r^{p}F_{L\mu}J^\mu[\phi] dxdt.
\end{split}
\end{equation}
Similarly, in the exterior region $\{r\leq R+t\}$ for the domain $\mathcal{D}_{\tau_1}^{\tau_2}$ for all $\tau_2\leq \tau_1\leq 0$, we have
\begin{equation}
\label{eq:pWE:sca:ex:id}
\begin{split}
  &\int_{H_{\tau_1^*}^{-\tau_2^*}}r^p|D_L\psi|^2dvd\om+\iint_{\mathcal{D}_{\tau_1}^{\tau_2}}r^{p-1}(p|D_L\psi|^2+(2-p)|\D\psi|^2)dvd\om du+\int_{\Hb_{-\tau_2^*}^{\tau_1^*}}r^p|\D\psi|^2dud\om\\
=&\f12\int_{B_{R-\tau_1}^{R-\tau_2}}r^p(|D_L\psi|^2+|\D\psi|^2)drd\om-\iint_{\mathcal{D}_{\tau_1}^{\tau_2}}r^{p-1}\Re\left(\Box_A\phi \overline{D_L\psi}\right)+r^{p}F_{L\mu}J^\mu[\phi] dxdt.
\end{split}
\end{equation}
For the inhomogeneous term, when $p\geq 1+\ep$, we apply Cauchy-Schwarz inequality directly
\[
 2r^{p+1}|\Box_A\phi \cdot \overline{D_L\psi}|\les r^{p} u_+^{-1-\ep}|D_L\psi|^2+r^{p+2}u_+^{1+\ep}|\Box_A\phi|^2.
\]
The integral of the first term in the above inequality can be controlled by using Gronwall's inequality both in \eqref{eq:pWE:sca:in:id} and \eqref{eq:pWE:sca:ex:id}. In particular this shows that estimate \eqref{eq:pWE:sca:ex}
follows from \eqref{eq:pWE:sca:ex:id}.

When $p<1+\ep$, we note that
\[
 2p-1-\ep<p\cdot \frac{p}{1+\ep}+(p-1)(1-\frac{p}{1+\ep}).
\]
Then we can estimate the inhomogeneous term as follows:
\begin{align*}
2r^{p+1}|\Box_A\phi \cdot \overline{D_L\psi}|&\leq \ep_1 r^{2p- 1-\ep} u_+^{-p}|D_L\psi|^2+\ep_1^{-1}r^{1+\ep+2}u_+^{ p}|\Box_A\phi|^2\\
&\leq \ep_1 (r^p u_+^{-1-\ep})^{\frac{p}{1+\ep}}(r^{p-1})^{ 1-\frac{p}{1+\ep}}|D_L\psi|^2+\ep_1^{-1}r^{1+\ep+2}u_+^{ p}|\Box_A\phi|^2\\
&\leq \ep_1 r^p u_+^{-1-\ep}|D_L\psi|^2+\ep_1 r^{p-1}|D_L\psi|^2+\ep_1^{-1}r^{1+\ep+2}u_+^{ p}|\Box_A\phi|^2
\end{align*}
for all $\ep_1>0$. The integral of the first term lcan be controlled by using Gronwall's inequality. The integral of
the second term can be absorbed for sufficiently small $\ep_1$. Then estimate \eqref{eq:pWE:sca:ex} follows.

For the $r$-weighted energy estimate \eqref{eq:pWE:sca:in} in the interior region, we need to control the boundary term on $\{r=R\}$. It suffices to estimate it for $p=0$ in \eqref{eq:pWE:sca:in:id} by making use of the energy estimates \eqref{eq:ILE:Sca:in}. From Hardy's inequality in Lemma \ref{lem:Hardyga},l we note that
\[
 \int_{H_{\tau}}|D_L\psi|^2d\om dv \les E[\phi](\Si_\tau).
\]
By using the itegrated local energy estimate \eqref{eq:ILE:Sca:in}, we therefore can show that
\begin{align*}
 &\left|\int_{\tau_1}^{\tau_2}\int_{\om}r^p(|D_L\psi|^2-|\D\psi|^2) d\om dt\right|\\
 \les & R^p \int_{H_{\tau_2^*}^{v_0}}|D_L\psi|^2dvd\om+\int_{\tau_1}^{\tau_2}\int_{H_{\tau^*}^{v_0}}r^{-1}|\D\psi|^2 dvd\om d\tau+\int_{\Hb_{v_0}^{\tau_1, \tau_2}}|\D\psi|^2dud\om\\
&\qquad +\int_{H_{\tau_1^*}^{v_0}}|D_L\psi|^2dvd\om+\int_{\tau_1}^{\tau_2}\int_{H_{\tau^*}^{v_0}}r^{-1}|\Re\left(\Box_A\phi \overline{D_L\psi}\right)|+|F_{L\mu}J^\mu[\phi]| dxdt\\
\les & E[\phi](\Si_{\tau_1})+I^{1+\ep}_0[\Box_A\phi](\mathcal{D}_{\tau_1}^{\tau_2})+\iint_{\mathcal{D}_{\tau_1}^{\tau_2}}|F_{\Lb\mu}J^\mu[\phi]|+|F_{L\mu}J^\mu[\phi]| dxdt.
\end{align*}
The inhomogeneous term can be bounded by using Cauchy-Schwarz inequality and the integrated local energy estimates. Once we have bound for the boundary terms on $\{r=R\}$, the $r$-weighted energy estimate \eqref{eq:pWE:sca:in} follows from
the identity \eqref{eq:pWE:sca:in:id} and Gronwall's inequality.
\end{proof}
Next we establish the $r$-weighted energy estimate for the Maxwell field.
\begin{prop}
 \label{prop:pWE:cur}
 Let $F$ be any 2-form satisfying the Bianchi identity \eqref{bianchi}. Then in the interior region, for all $0\leq \tau_1\leq \tau_2$, $v_0\geq \frac{\tau_2+R}{2}$, we have the $r$-weighted energy estimate:
 \begin{align}
\notag
 &\int_{H_{\tau_2^*}}r^{p+2}|\a|^2dvd\om+\int_{\tau_1}^{\tau_2}\int_{H_{\tau^*}}r^{p+1}(p|\a|^2+(2-p)(\rho^2+\si^2))dvd\om d\tau+\int_{\Hb_{v_0}^{\tau_1, \tau_2}}r^{p+2}(\rho^2+\si^2)dud\om\\
   \label{eq:pWE:cur:in}
\les &\int_{H_{\tau_1^*}^{v_0}}r^{p+2}|\a|^2dvd\om+I^{\max\{p, 1+\ep\}}_{\min\{1+\ep, p\}}[\J](\mathcal{D}_{\tau_1}^{\tau_2})+(2-p)^{-1}I^{p+1}_0[J_{L}](\mathcal{D}_{\tau_1}^{\tau_2})\\
\notag
&+E[F](\Si_{\tau_1})+I^{1+\ep}_0[|J_{L}|+|\J|](\mathcal{D}_{\tau_1}^{\tau_2})+\iint_{\mathcal{D}_{\tau_1}^{\tau_2}}|J_{\Lb}||\rho|dxdt
 \end{align}
for all $0\leq p\leq 2$. Similarly in the exterior region, for all $\tau_2\leq \tau_1\leq 0$, $0\leq p\leq 2$, we have
\begin{equation}
 \label{eq:pWE:cur:ex}
 \begin{split}
  &\int_{H_{\tau_1^*}^{-\tau_2^*}}r^p|\a|^2r^2dvd\om+\iint_{\mathcal{D}_{\tau_1}^{\tau_2}}r^{p+1}(p|\a|^2+(2-p)(\rho^2+\si^2))dvd\om du+\int_{\Hb_{-\tau_2^*}^{\tau_1^*}}r^p(\rho^2+\si^2)r^2dud\om\\
\les &\int_{B_{R-\tau_1}^{R-\tau_2}}r^p |F|^2 dx+I^{\max\{p, 1+\ep\}}_{\min\{p,1+\ep\}}[\J ](\mathcal{D}_{\tau_1}^{\tau_2})+(2-p)^{-1}I^{p+1}_0[J_{L}](\mathcal{D}_{\tau_1}^{\tau_2}).
 \end{split}
\end{equation}
\end{prop}
\begin{proof}
Take the vector field
\[
 X=r^p L=f\pa_t+f\pa_r
\]
in the energy identity \eqref{energyeqcur} for the Maxwell field. Using the computations before line \eqref{TFXr}, we have
\begin{align*}
 T[F]^{\mu\nu}\pi^X_{\mu\nu}&=T[F]^{\mu\nu}\pi^{f\pa_r}_{\mu\nu}+T[F]^{\mu\nu}\pi^{f\pa_t}\\
&=(r^{-1}f-\f12 f')(\rho^2+\si^2)+\frac{1}{4}f'(|\a|^2+|\underline\a|^2)+\frac{f'}{4}(|\a|^2-|\underline{\a}|^2)\\
&=\f12 r^{p-1}\left((2-p)(\rho^2+\si^2)+ p|\a|^2\right).
\end{align*}
For the boundary terms corresponding to the vector field $X=r^p L$, we have
\begin{align*}
i_{J^{X}[F]}d\vol=\f12 r^p (|\a|^2+\rho^2+\si^2)dx,\quad & i_{J^{X}[F]}d\vol=\f12 r^p (|\a|^2-\rho^2-\si^2)r^2dtd\om,\\
i_{J^{X}[\phi]}d\vol=r^p|\a|^2 r^2dvd\om,\quad & i_{J^{X}[\phi]}d\vol=-r^p (\rho^2+\si^2)r^2dud\om
\end{align*}
respectively on the $t=constant$ slice, $r=constant$ surface, the outgoing null hypersurface and the incoming null hypersurface. Therefore for all $0\leq \tau_1\leq \tau_2$, $v_0\geq \frac{\tau_2+R}{2}$ if we take the region $\mathcal{D}$
 bounded by $H_{\tau_1^*}$, $H_{\tau_2^*}$, $\{r=R\}$, $\Hb_{v_0}^{\tau_1, \tau_2}$, we get the $r$-weighted energy identity in the interior region $\{r\leq R+t\}$:
\begin{align}
\notag
 &\int_{H_{\tau_2^*}^{v_0}}r^p|\a|^2r^2dvd\om+\int_{\tau_1}^{\tau_2}\int_{H_{\tau^*}^{v_0}}r^{p-1}(p|\a|^2+(2-p)(\rho^2+\si^2))r^2dvd\om d\tau+\int_{\Hb_{v_0}^{\tau_1, \tau_2}}r^p(\rho^2+\si^2)r^2dud\om\\
 \label{eq:pWE:cur:in:id}
=&\int_{H_{\tau_1^*}^{v_0}}r^p|\a|^2r^2dvd\om-\f12\int_{\tau_1}^{\tau_2}\int_{\om}r^p(|\a|^2-\rho^2-\si^2) r^2d\om dt-\int_{\tau_1}^{\tau_2}\int_{H_{\tau^*}^{v_0}}r^p \pa^\mu F_{\mu\nu}F_{L}^{\;\nu} dxdt.
\end{align}
Similarly, in the exterior region $\{r\geq R+t\}$, consider the region $\mathcal{D}_{\tau_1}^{\tau_2}$ for $\tau_2\leq \tau_1\leq 0$. We have the following identity:
\begin{align}
\notag
  &\int_{H_{\tau_1^*}^{-\tau_2^*}}r^p|\a|^2r^2dvd\om+\iint_{\mathcal{D}_{\tau_1}^{\tau_2}}r^{p-1}(p|\a|^2+(2-p)(\rho^2+\si^2))dvd\om du+\int_{\Hb_{-\tau_2^*}^{\tau_1^*}}r^p(\rho^2+\si^2)r^2dud\om\\
\label{eq:pWE:cur:ex:id}
=&\f12\int_{B_{R-\tau_1}^{R-\tau_2}}r^p(|\a|^2+\rho^2+\si^2)dx-\iint_{\mathcal{D}_{\tau_1}^{\tau_2}}r^{p} \pa^\mu F_{\mu\nu}F_{L}^{\;\nu}  dxdt.
\end{align}
To obtain \eqref{eq:pWE:cur:ex}, we note that under the null frame $\{L, \Lb, e_1, e_2\}$
\[
 F_{L}^{\;L}=-\f12 F_{L\Lb}=\rho,\quad F_{L}^{\;\Lb}=0, \quad F_{L}^{\;e_j}=\a_j,\quad j=1, 2.
\]
We can use the same method to treat the term $r^p|\pa^\mu F_{\mu e_j}F_{L}^{\;e_j}|$ as that for $\Box_A\phi \cdot \overline{D_L\psi}$ in the previous Proposition \ref{prop:pWE:sca} (simply replace $\Box_A\phi$ with $\pa^\mu F_{\mu e_j}$ and
$\overline{D_L}\psi$ with $r \a_j$).
For the term involving $\rho$, we estimate
\[
r^{p+2} |J_{L}\cdot \rho|\leq \f12 (2-p)r^{p+1}|\rho|^2+\frac{2}{2-p}r^{p+3}|J_{L}|^2.
\]
The integral of the first term could be absorbed. Then the $r$-weighted energy estimate \eqref{eq:pWE:cur:ex} follows from the above $r$-weighted energy identity \eqref{eq:pWE:cur:ex:id}.

We can treat the inhomogeneous term the same way for the $r$-weighted energy estimate in the interior region from the $r$-weighted energy identity \eqref{eq:pWE:cur:in:id}. Like the case for the scalar field, the boundary term on $\{r=R\}$ can be bounded by taking
$p=0$ in \eqref{eq:pWE:cur:in:id} and then by making use of the integrated local energy estimate \eqref{eq:ILE:cur:in}:
\begin{align*}
|\int_{\tau_1}^{\tau_2}\int_{\om}r^p(|\a|^2-\rho^2-\si^2) r^2d\om dt|&\les  \int_{\tau_1}^{\tau_2}\int_{H_{\tau^*}^{v_0}}(\rho^2+\si^2)rdvd\om d\tau+\int_{\Hb_{v_0}^{\tau_1, \tau_2}}(\rho^2+\si^2)r^2dud\om\\
 &+R^p\int_{H_{\tau_2^*}^{v_0}}|\a|^2r^2dvd\om+ \int_{H_{\tau_1^*}^{v_0}}|\a|^2r^2dvd\om+\int_{\tau_1}^{\tau_2}\int_{H_{\tau^*}^{v_0}}|\pa^\mu F_{\mu\nu}F_{L}^{\;\nu} |dxdt\\
& \les  E[F](\Si_{\tau_1})+I^{1+\ep}_0[|J_{L}|+|\J|](\mathcal{D}_{\tau_1}^{\tau_2})+\iint_{\mathcal{D}_{\tau_1}^{\tau_2}}|J_{\Lb}||\rho|dxdt.
\end{align*}
This together with the above argument by using Gronwall's inequality imply the $r$-weighted energy estimate for the Maxwell field in the interior region.
\end{proof}

\section{Decay estimates for the linear solutions}
\label{sec:decay:lin}

In this section we derive energy flux decay both for the linear Maxwell field and linear complex scalar field under appropriate assumptions.
We use bootstrap argument to construct global solutions of the nonlinear \eqref{EQMKG}. The first step is to study the decay properties of the linear solutions. Recall that $F=dA$ with $A$ the connection used to define the covariant derivative
$D$. Our strategy is that we make assumptions on $J_{\mu}=\pa^\nu F_{\mu\nu}$ to obtain estimates for the linear solution $F$. We then use these estimates to derive estimates for the solutions of the linear
covariant wave equation $\Box_A\phi=0$. As in the equation \eqref{EQMKG} the nonlinearity $J$ is quadratic in $\phi$, by making use of the smallness of the scalar field we then can close our bootstrap assumption on $J$. The difficulty is that
the Maxwell field $F$ is no longer small and the existence of nonzero charge.

 Assume that $F=dA$ has charge $q_0$ and splits into the charge part and chargeless part
  \[
   F=\chi_{\{r>t+R\}}q_0 r^{-2}dt\wedge dr+\bar F.
  \]
Let $\J=(J_{e_1}, J_{e_2})$ be the angular component of $J$. Denote
\begin{equation}
\label{eq:def4NkF}
\begin{split}
 m_k&=\sum\limits_{l\leq k}I^{1+\ga_0}_{1+\ep}[\mathcal{L}_Z^l \J](\{r\geq R\})+I^{2+\ga_0}_0[\mathcal{L}_Z^l J_L](\{ r\geq R\})+I^{1+\ep}_{1+\ga_0}[|\mathcal{L}_Z^l \J|+|\mathcal{L}_Z^l J_L|](\{t\geq 0\})\\
 & +I^{1-\ep}_{1+\ga_0+2\ep}[\mathcal{L}_Z^l J_{\Lb}](\{t\geq 0\})+I_{1+\ga_0}^0[\nabla \mathcal{L}_Z^{l-1}J](\{r\leq 2R\})+|q_0|\sup\limits_{\tau\leq 0}\tau_+^{1+\ga_0}\iint_{\mathcal{D}_{\tau}^{-\infty}}|J_{\Lb}|r^{-2}dxdt,\\
 M_k &=m_k+ \sum\limits_{l\leq k}E^k_0[\bar F]+1+|q_0|,
\end{split}
 \end{equation}
 where we recall from line \eqref{eq:def4E0kFphi} in Section \ref{sec:notation} that $E_0^k[F]$ denotes the weighted Sobolev norm of the Maxwell field with weights $r_+^{1+\ga_0}$ on the initial hypersurface $t=0$. The integral of $|J_{\Lb}|r^{-2}$ is used to control the interaction of the nonzero charge with the nonlinearity $J$ in the exterior region.

To derive the energy decay for the Maxwell field, we assume that $M_k$ is finite. Note that $m_k$ can be viewed as the assumption on the inhomogeneous term $J$ which depends on the scalar field $\phi$ in the equation \eqref{EQMKG}. According
to our assumption, it is small. $E^k_0[\bar F]$ denotes the size of the initial data for the chargeless part of the Maxwell field. Hence it is large in our setting. The charge $q_0$ is a constant depending on the initial data of the
scalar field.

\subsection{Energy decay for the Maxwell field}
We derive energy flux decay for the Maxwell field $F$ under the assumption that $M_k$ is finite.
\begin{prop}
 \label{prop:Endecay:cur}
 In the interior region for all $0\leq \tau_1\leq \tau_2$, $v_0\geq\frac{\tau_2+R}{2}$, we have the following energy flux decay for the Maxwell field:
  \begin{equation}
  \label{eq:Endecay:cur:in}
I^{-1-\ep}_0[F](\mathcal{D}_{\tau_1}^{\tau_2})+ \int_{\tau_1}^{\tau_2}\int_{\Si_{\tau}}\frac{\rho^2+\si^2}{1+r} dxd\tau+E[F](\Hb_{v_0}^{\tau_1,\tau_2})+E[F](\Si_{\tau_1})\les (\tau_1)_+^{-1-\ga_0}M_0.
\end{equation}
In the exterior region $\{r\leq R+t\}$ for all $\tau_2\leq \tau_1\leq 0$, $0\leq p\leq 1+\ga_0$, we have
\begin{align}
 \label{eq:Endecay:cur:ex}
&I^{-1-\ep}_0[\bar F](\mathcal{D}_{\tau_1}^{\tau_2})+E[\bar F](\Hb_{-\tau_2^*}^{\tau_1^*})+E[\bar F](H_{\tau_1^*})+(\tau_1)_+^{-p}\int_{H_{\tau_1^*}}r^{p+2}|\a|^2dvd\om
\les  (\tau_1)_+^{-1-\ga_0}M_0.
\end{align}
\end{prop}
\begin{proof}
Let's first consider the estimates in the exterior region. By the definition of $M_0$, we derive that
\begin{align*}
  \int_{B_{R-\tau_1}^{R-\tau_2}}r^p |\bar F|^2 dx+&I^{p}_{1+\ep}[\J](\mathcal{D}_{\tau_1}^{\tau_2})+I^{p+1}_0[J_{L}](\mathcal{D}_{\tau_1}^{\tau_2})\les (\tau_1)_+^{p-1-\ga_0}M_0,\quad 0\leq p\leq 1+\ga_0.
\end{align*}
Here note that in the exterior region $r\geq \frac{1}{2}u_+$. Then the $r$-weighted energy estimate \eqref{eq:pWE:cur:ex} implies that
\begin{equation*}
\begin{split}
 \int_{H_{\tau_1^*}^{-\tau_2^*}}r^{p+2}|\a|^2dvd\om+\iint_{\mathcal{D}_{\tau_1}^{\tau_2}}r^{p+1}(|\a|^2+\bar\rho^2+\si^2)dvd\om du
 \les (\tau_1)_+^{p-1-\ga_0}M_0.
 \end{split}
\end{equation*}
 This estimate can be used to bound the integral of $|J_{\Lb}||\rho|$ on the right hand side of \eqref{eq:ILE:cur:ex}. Recall that $\rho=q_0 r^{-2}+\bar \rho$ when $r\geq R+t$. We then can show thatl
\begin{align*}
\iint_{\mathcal{D}_{\tau_1}^{\tau_2}}|J_{\Lb}||\rho|dxdt\les \iint_{\mathcal{D}_{\tau_1}^{\tau_2}}|q_0||J_{\Lb}|r^{-2}+|\bar\rho|^2 r^{\ep-1}u_+^{-\ep}+
|J_{\Lb}|^2r^{1-\ep}u_+^{\ep}dxdt\les M_0(\tau_1)_+^{-1-\ga_0}.
\end{align*}
The decay estimate \eqref{eq:Endecay:cur:ex} then follows from the energy estimate \eqref{eq:ILE:cur:ex} as
 \[
E[\bar F](B_{R-\tau_1}^{R-\tau_2})+I^{1+\ep}_0[|\J|+|J_{L}|](\mathcal{D}_{\tau_1}^{\tau_2})\les (\tau_1)_+^{-1-\ga_0}M_0.
 \]
For the decay estimates in the interior region, we use the pigeon hole argument in \cite{newapp}. First by interpolation, we derive from the definition of $M_0$ that
\begin{align*}
I^{\max\{p, 1+\ep\}}_{\min\{p, 1+\ep\}}[\J](\mathcal{D}_{\tau_1}^{\tau_2})+I^{p+1}_0[J_{L}](\mathcal{D}_{\tau_1}^{\tau_2})\les (\tau_1)_+^{p-1-\ga_0}M_0
\end{align*}
for all $\ep\leq p\leq 1+\ga_0$. To bound $|J_{\Lb}||\rho|$, we use Cauchy-Schwaz:
\[
\iint_{\mathcal{D}_{\tau_1}^{\tau_2}}|J_{\Lb}||\rho|dxdt\les \iint_{\mathcal{D}_{\tau_1}^{\tau_2}}\ep_1|\rho|^2 r_+^{\ep-1}+\ep_1^{-1}r_+^{1-\ep}|J_{\Lb}|^2 dxdt,\quad \forall \ep_1>0.
\]
Here note that in the interior region $\rho=\bar \rho$. For $\ep\leq p\leq 1+\ga_0$ and sufficiently small $\ep_1$ the first term could be absorbed from the $r$-weighted energy estimates \eqref{eq:pWE:cur:in} and the second term is bounded above by $M_0(\tau_1)_+^{-1-\ga_0}$ by the definition of $M_2$.

To apply the pigeon hole argument, we need to control the weighted energy flux through the initial hypersurface $\Si_0$ of the interior region. Note that $H_{-\frac{R}{2}}=H_{0^*}$. The weighted energy flux bound through $H_{0^*}$ follows from the decay estimate \eqref{eq:Endecay:cur:ex} in the exterior region:
\begin{align*}
E[F](H_{0^*})+\int_{H_{0^*}}r^{3+\ga_0}|\a|^2dvd\om\les M_0.
\end{align*}
Here we note that on the boundary $H_{0^*}$ the charge part has bounded energy. Hence
take $p=1+\ga_0$, $\tau_1=0$ in the $r$-weighted energy estimate \eqref{eq:pWE:cur:in}. We derive that
\begin{equation*}
 \int_{H_{\tau_2^*}}r^{3+\ga_0}|\a|^2dvd\om+\int_{0}^{\tau_2}\int_{H_{\tau^*}}r^{\ga_0+2}(|\a|^2+\si^2+\rho^2)dvd\om\les M_0,\quad \forall \tau_2\geq 0.
\end{equation*}
We conclude that there exists a dyadic sequence $\{\tau_{n}\}$, $n\geq 3$ such that
\begin{equation*}
 \int_{H_{\tau_n^*}}r^{\ga_0+2}|\a|^2dvd\om \les (\tau_n)_+^{-1}M_0,\quad \la^{-1}\tau_n\leq \tau_{n+1}\leq \la \tau_n.
\end{equation*}
for some constant $\la$ depending only on $\ga_0$, $\ep$, $R$. Interpolation implies that
\[
 \int_{H_{\tau_n^*}}r^{1+2}|\a|^2dvd\om \les (\tau_n)_+^{-\ga_0}M_0.
\]
To bound $|J_{\Lb}||\rho|$ on the right hand side of the energy estimate \eqref{eq:ILE:cur:in}, we interpolate $|\rho|$ between the integrated local energy estimate and the above $r$-weighted energy estimate:
\begin{align*}
\iint_{\mathcal{D}_{\tau_1}^{\tau_2}}|J_{\Lb}||\rho|dxdt &\les \iint_{\mathcal{D}_{\tau_1}^{\tau_2}}\ep_1|\rho|^2 (r_+^{-\ep-1}+r_+^{\ga_0}\tau_+^{-1-\ga_0})+\ep_1^{-1}\tau_+^{2\ep}r_+^{1-\ep}|J_{\Lb}|^2 dxdt\\
&\les \ep_1 I^{-1-\ep}_0[F](\mathcal{D}_{\tau_1}^{\tau_2})+\ep_1^{-1}M_0(\tau_1)_+^{-1-\ga_0},\quad \forall 1>\ep_1>0.
\end{align*}
Here we have used the bound
\[
r_+^{\ep-1}\tau_+^{-2\ep}\leq r_+^{-1-\ep}+\tau_+^{-1-\ga_0}r_+^{\ga_0}.
\]
Take $\ep_1$ to be sufficiently small. From the energy estimate \eqref{eq:ILE:cur:in}, we then obtain
\begin{align*}
 I^{-1-\ep}_0[F](\mathcal{D}_{\tau_1}^{\tau_2})+E[F](\Si_{\tau_2})\les E[F](\Si_{\tau_1})+(\tau_1)_+^{-1-\ga_0}M_0
\end{align*}
for all $0\leq \tau_1\leq \tau_2$, $0<\ep_1<1$. In particular, we have
\[
 \int_{\tau_n}^{\tau_2}\int_{\{r\leq R\}\cap \{t=\tau\}}|F|^2 dx d\tau\les E[F](\Si_{\tau_1})+(\tau_1)_+^{-1-\ga_0}M_0.
\]
Then combine this integrated local energy estimate with the $r$-weighted energy estimate \eqref{eq:pWE:cur:in} with $p=1$. For all $\tau_n\leq \tau_2$, we derive that
\begin{align*}
 \int_{\tau_n}^{\tau_2}E[F](\Si_{\tau})d\tau&\les \int_{\tau_n}^{\tau_2}\int_{\{r\leq R\}\cap \{t=\tau\}}|F|^2 dx d\tau+\int_{\tau_n}^{\tau_2}\int_{H_{\tau^*}}
 (|\a|^2+\si^2+\rho^2)r^2dvd\om\\
 &\les \int_{H_{\tau_n^*}}r^{1+2}|\a|^2dvd\om+E[F](\Si_{\tau_n})+(\tau_n)_+^{-\ga_0}M_0\\
 &\les E[F](\Si_{\tau_n})+(\tau_n)_+^{-\ga_0}M_0.
\end{align*}
On the other hand, for all $\tau\leq \tau_2$, we have
\[
E[F](\Si_{\tau_2})\leq E[F](\tau)+(\tau)_+^{-1-\ga_0}M_0\les E[F](\Si_0)+M_0\les M_0.
\]
Then from the previous estimate, we can show that
\begin{align*}
(\tau_2-\tau_n)E[F](\Si_{\tau_2})\les E[F](\Si_{\tau_n})+(\tau_n)_+^{-\ga_0}M_0\les M_0.
\end{align*}
The above estimate holds for all $\tau_2\geq \tau_n$. In particular, we obtain the coarse bound
\[
E[F](\Si_{\tau})\les \tau_+^{-1}M_0,\quad \forall \tau\geq 0.
\]
Based on this coarse bound, we can take $\tau_2=\tau_{n+1}$ in the previous estimate. We then can show that
\[
(\tau_{n+1}-\tau_n)E[F](\Si_{\tau_{n+1}})\les (\tau_n)^{-\ga_0}M_0.
\]
As $\{\tau_n\}$ is dyadic, we conclude that
\[
E[F](\Si_{\tau_n})\les (\tau_n)^{-1-\ga_0}M_0,\quad \forall n\geq 3.
\]
Then using the energy estimate, we can show that for $\tau\in[\tau_n, \tau_{n+1}]$ we have
\[
E[F](\Si_{\tau})\les E[F](\tau_{n})+(\tau_n)_+^{-1-\ga_0}M_0\les (\tau_n)_+^{-1-\ga_0}M_0\les \tau_+^{-1-\ga_0}M_0.
\]
Having this energy flux decay, the integrated local energy decay \eqref{eq:Endecay:cur:in} follows from the integrated local energy estimate \eqref{eq:ILE:cur:in}.
\end{proof}
Since the Lie derivative $\mathcal{L}_{Z}$ commutes with the linear Maxwell equation from the commutator Lemma \ref{lem:commutator}, as a corollary of the above energy decay proposition, we also have the
energy decay estimates for the higher order derivatives of the Maxwell field.
\begin{cor}
\label{cor:Endecay:cur:hide}
 We have the following energy flux decay for the $k$-th derivative of the Maxwell field:
  \begin{equation}
  \label{eq:Endecay:cur:hide}
E[\mathcal{L}_Z^k \bar F](\Si_{\tau})\les (\tau)_+^{-1-\ga_0}M_k,\quad \forall \tau\in\mathbb{R}.
\end{equation}
\end{cor}
This decay estimate then leads to the integrated local energy and $r$-weighted energy estimates for the Maxwell field.
\begin{remark}
By using the finite speed of propagation, the estimates in the above proposition and corollary in the exterior region depend only on the data and $J$
in the exterior region $\{t+R\leq r\}$ instead of the whole spacetime. Therefore the quantity $M_k$ can be replaced by the corresponding one defined in the exterior region. However, the estimates in the interior region rely on the data in the whole space.
\end{remark}
\subsection{Pointwise bound for the Maxwell field}
\label{sec:pbdMaxwell}
The energy decay estimates derived in the previous section are sufficient to obtain pointwise bound for the Maxwell field $F$ after commuting the equation with $Z=\{\pa_t, \Om\}$ for sufficiently many times, e.g., in \cite{yang2}, four derivatives were used to show the pointwise bound for the solution. The aim of this section is to derive the pointwise bound for the Maxwell field $F$ merely assuming $M_2$ is finite, that is, we commute the equation with $Z$ for only twice. The difficulty is that we are lack of Klainerman-Sobolev embedding which leads to the decay of the solution directly, see e.g. \cite{LindbladMKG}. Our idea is that in the inner region $r\leq R$ we rely on elliptic estimates. In the outer region $r\geq R$, write the solutions as functions of $(u, v, \om)$. The angular momentum $\Om$ can be viewed as derivative on $\om$. The pointwise bound then follows by using a trace theorem on the null hypersurfaces and a Sobolev embedding on the sphere. Since we do not commute the equation with $L$ nor $\Lb$, those necessary energy estimates heavily rely on the null equations given in Lemma \ref{lem:nullMKG}.

Let's first consider the pointwise bound for the Maxwell field in the inner region $\{r\leq R\}$. To derive the pointwise bound, we use the vector fields $\pa_t$ and the angular momentum $\Om$ as commutators. Note that the angular momentum vanishes at $r=0$. In particular we are not able to get the robust estimates for the solution in the bounded region
$r\leq R$ merely from the angular momentum. We thus rely on the killing vector field $\pa_t$ andl elliptic estimates. The following proposition gives the estimates for the Maxwell field $F$ on the bounded region $\{r\leq R\}$.
\begin{prop}
 \label{prop:Est4F:in:R}
  For all $0\leq \tau$, $0\leq \tau_1<\tau_2$, we have
 \begin{align}
  \label{eq:Est4F:in:R}
  \int_{\tau_1}^{\tau_2}\sup_{|x|\leq R}|F|^2(\tau, x)d\tau&\les  \int_{\tau_1}^{\tau_2}\int_{r\leq R}|\pa^2 F|^2 dxdt \les M_2(\tau_1)_+^{-1-\ga_0},\\
 \label{eq:Est4F:in:R:p}
  |F|^2(\tau, x)&\les M_2 \tau_+^{-1-\ga_0},\quad \forall |x|\leq R.
  \end{align}
\end{prop}
\begin{remark}
Estimate \eqref{eq:Est4F:in:R:p} gives the pointwise bound for $F$ in the inner region $\{r\leq R\}$ but it is weaker than the integral version \eqref{eq:Est4F:in:R} in the sense of decay rate. It is this integrated decay estimate that allows us to control the nonlinearities in the inner region. In other words, it is not necessary to show the improved decay of the solution in the inner region by using our approach, see e.g. \cite{improvLuk}. However this does not mean that our method is not able to obtain the improved decay in the inner region. The improved decay can be derived by commuting the equation with the vector field $L$. For details about this, we refer to \cite{Volker:LSchw}.
\end{remark}
\begin{proof}
 We use elliptic estimates to prove this proposition. At fixed time $t$, let $E$, $H$ be the electric and magnetic part of the Maxwell field $F$. Let $B_r$ be the ball with radius $r$, that is, $B_r=\{t, |x|\leq r\}$. The Maxwell equation can be written as follows:
 \begin{align*}
  div(E)=J_0,\quad \pa_t H +cur(E)=0,\\
  div(H)=0,\quad \pa_t E-cur(H)=\bar J,
 \end{align*}
where $\bar J=(J_1, J_2, J_3)$ is the spatial part of $J$. Therefore by using elliptic theory we derive that
\begin{align*}
 \sum\limits_{k\leq 1}\|\pa_t^k F\|_{H^1_x(B_{ \frac{3R}{2}})}^2\leq \sum\limits_{k\leq 1}\|\pa_t^k H\|_{H^1_x(B_{ \frac{3R}{2}})}^2+\|\pa_t^k E\|_{H^1_x(B_{ \frac{3R}{2}})}^2\les \sum\limits_{k\leq 1}
 \|\pa_t^k J\|_{L^2_x(B_{ 2R})}^2+\|\pa_t^{k+1} F\|_{L^2_x(B_{ 2R})}^2.
\end{align*}
Make use of the above estimates with $k=1$. Differentiate the linear Maxwell equation with the spatial derivative $\nabla$. Using elliptic estimates again, we then obtain
\begin{align*}
 \|\pa F\|_{H^1_x(B_{ R})}^2\les \|\pa J\|_{L^2_x(B_{ 2R})}^2+\|\pa^2_t F\|_{L^2_x(B_{ 2R})}^2.
\end{align*}
Here we omitted the lower order terms. Integrate the above inequality from time $\tau_1$ to $\tau_2$. We derive
\begin{align*}
 \int_{\tau_1}^{\tau_2}\int_{r\leq R}|\pa^2 F|^2 dxdt &\les \int_{\tau_1}^{\tau_2}\int_{r\leq 2R}|\pa^2_t F|^2+|\pa J|^2 dxdt\\
 &\les I_0^{-1-\ep}[\pa_t^2 F](\mathcal{D}_{\tau_1^+}^{\tau_2})+I_0^{-1-\ep}[\pa_t J](\mathcal{D}_{\tau_1^+}^{\tau_2})+I_0^0[\nabla J](\mathcal{D}_{\tau_1}^{\tau_2}\cap\{r\leq 2R\})\\
 &\les M_2 (\tau_1)_+^{-1-\ga_0}.
\end{align*}
Here $\tau_1^+=\max\{\tau_1-R, 0\}$. The estimate \eqref{eq:Est4F:in:R} then follows by using Sobolev embedding.

For the pointwise bound \eqref{eq:Est4F:in:R:p}, first we note that
\begin{align*}
\int_{r\leq 2R}|\nabla J|^2 dx\les \sum\limits_{k\leq 1}\int_{\tau}^{\tau+1}|\nabla \mathcal{L}_Z^kJ|^2dxdt\les M_2\tau_+^{-1-\ga_0}.
\end{align*}
Consider energy estimate on the region $\mathcal{D}_1$ bounded by $\Si_{\tau^+}$, $\tau^+=\max\{\tau-R, 0\}$ and $t=\tau$, $\tau\geq 0$. From the energy estimate \eqref{eq:ILE:cur:in}, we conclude that
\[
\int_{r\leq 2R}|\mathcal{L}_Z^2 F|^2 dx=E[\mathcal{L}_Z^2 F](r\leq 2 R)\les E[\mathcal{L}_Z^2 F](\Si_{\tau^+})+I^{1+\ep}_0[\mathcal{L}_Z^2J](\mathcal{D}_1)\les M_2\tau_+^{-1-\ga_0}.
\]
Thus the pointwise bound \eqref{eq:Est4F:in:R:p} holds.
\end{proof}
To show the decay of the solution via the energy flux through the null hypersurface, we rely on the following trace theorem.
\begin{lem}[Trace theorem]
\label{lem:trace}
Let $f(r, \om)$ be a smooth function defined on $[a, b]\times \mathbb{S}^2$. Then
\begin{equation}
\label{eq:trace}
\left(\int_{\om}|f|^4(r_0, \om)d\om\right)^{\f12}\leq C\int_{a}^{b}\int_{\om}|f|^2+|\pa_rf|^2+|\pa_{\om}f|^2d\om dr,\quad \forall r_0\in[a, b]
\end{equation}
for some constant $C$ independent of $r_0$.
\end{lem}
\begin{proof}
The condition implies that $f\in H^1_{r, \om}$.
By using trace theorem,
\[
\|f(r_0, \cdot)\|_{H^{\f12}_{\om}}\leq C\|f\|_{H^1_{r, \om}},\quad \forall r_0\in[a, b].
\]
The lemma then follows by using Sobolev embedding on the sphere.
\end{proof}
Using this lemma, we are now able to show the pointwise bound for the Maxwell field when l$\{r\geq R\}$.
\begin{prop}
\label{prop:supF}
Let $\bar{\mathcal{D}}_{\tau_1}=\mathcal{D}_{\tau_1}\cap\{r\geq R\}$. Then we have
\begin{align}
\label{eq:supab:I}
\|r \mathcal{L}_{Z}^k\ab\|_{L_u^2L_v^\infty L_{\om}^2(\bar{\mathcal{D}}_{\tau_1})}^2&\les M_2(\tau_1)_+^{-1-\ga_0+2\ep},\quad k=0, 1,\\
\label{eq:supab:p}
|r\ab|^2(\tau, v, \om)&\les
 M_2\tau_+^{-1-\ga_0},\\
\label{eq:suparhosi:p}
 r^p(|r\a|^2+|r\si|^2)(\tau, v, \om)&\les
 M_2\tau_+^{p-1-\ga_0},\quad 0\leq p\leq 1+\ga_0,\\
 \label{eq:supbarrho:p}
 r^{p}|r\bar\rho|^2(\tau, v, \om)&\les M_2\tau_+^{p-1-\ga_0},\quad 0\leq p\leq 1-\ep,\\
\label{eq:suprhosi:I}
\|r\mathcal{L}_Z^k\si\|_{L_v^2L_u^\infty L_\om^2(\bar{\mathcal{D}}_{\tau})}^2 &\les M_2\tau_+^{-1-\ga_0+\ep},\quad k\leq 1.
\end{align}
\end{prop}
\begin{remark}
In terms of decay rate, the integral version \eqref{eq:supab:I} is stronger than the pointwise bound \eqref{eq:supab:p}. We are not able to improve the $u$ decay of the Maxwell field is due to the weak decay rate of the initial data as initially the pointwise bound \eqref{eq:supab:p} is the best we can get. However the integral version improves one order of decay in $u$. This is the key point that we are able to construct the global solution with the weak decay rate of the initial data.
\end{remark}
\begin{proof}
For the integral estimate \eqref{eq:supab:I}, we rely on the transport equation \eqref{eq:eq4ab} for $\ab$. For the case in the exterior region, one can choose the
initial hypersurface $\{t=0\}$. In the interior region, for all $0\leq \tau_1\leq \tau_2$, we can choose the incoming null hypersurface $\Hb_{\frac{\tau_2+R}{2}}^{\tau_1, \tau_2}$. Let's only consider the case in the interior region. From the equation \eqref{eq:eq4ab} for $\ab$ under the null frame, for $k=0$ or $1$, we can show that
\begin{align*}
\|r \mathcal{L}_{Z}^k\ab\|_{L_u^2L_v^\infty L_{\om}^2(\bar{\mathcal{D}}_{\tau_1}^{\tau_2})}^2&\les E[\mathcal{L}_Z^k\ab](\Hb_{\frac{\tau_2+R}{2}}^{\tau_1, \tau_2})+I^{-1-\ep}_0[\mathcal{L}_Z^k\ab](\bar{\mathcal{D}}_{\tau_1}^{\tau_2})
+\|r^{\frac{1+\ep}{2}}L\mathcal{L}_Z^k(r\ab)\|_{L_u^2L_v^2L_{\om}^2(\bar{\mathcal{D}}_{\tau_1}^{\tau_2})}^2\\
& \les M_2(\tau_1)_+^{-1-\ga_0}+\|r^{\frac{1+\ep}{2}}(|\mathcal{L}_Z^{k+1}\rho|+|\mathcal{L}_Z^{k+1}\si|)\|_{L_u^2L_v^2L_{\om}^2(\bar{\mathcal{D}}_{\tau_1}^{\tau_2})}^2
+I^{1+\ep}_0[\mathcal{L}_{Z}^k \J](\bar{\mathcal{D}}_{\tau_1}^{\tau_2})\\
&\les M_2(\tau_1)_+^{-1-\ga_0+2\ep}.
\end{align*}
Here we used interpolation to bound $\rho$ and $\si$. Indeed, the integrated local energy estimate implies that
\[
\iint r^{-\ep+1}(|\mathcal{L}_Z^{k+1}\rho|^2+|\mathcal{L}_Z^{k+1}\si|^2)dudvd\om\les M_2 (\tau_1)_+^{-1-\ga_0}.
\]
On the other hand the $r$-weighted energy estimate shows that
\[
\iint r^{2+\ga_0}(|\mathcal{L}_Z^{k+1}\rho|^2+|\mathcal{L}_Z^{k+1}\si|^2)dudvd\om\les M_2.
\]
Interpolation then implies the estimate for $\rho$ and $\si$. Thus estimate \eqref{eq:supab:I} holds.

For the pointwise bound \eqref{eq:supab:p} for $\ab$, we rely on the energy flux on the incoming null hypersurface together with Lemma \ref{lem:trace}. Consider the point $(\tau, v, \om)$. In the exterior region when $\tau<0$, let $\Hb_{\tau}$ be the incoming null hypersurface $\Hb_{v}^{\tau^*, -v}$, which is the incoming null hypersurface extending to the initial hypersurface $t=0$. In the interior region when $\tau\geq 0$, let $\Hb_{\tau}$ be $\Hb_{v}^{\tau, 2v-R}$ which is the incoming null hypersurface truncated by $r=R$. From the energy estimate \eqref{eq:ILE:cur:ex}, \eqref{eq:ILE:cur:in}, we conclude that
\[
\int_{\Hb_{\tau}}|r\mathcal{L}_Z^k\ab|^2dud\om\les E[\mathcal{L}_Z^k F](\Hb_{\tau})\les M_2\tau_+^{-1-\ga_0},\quad \forall k\leq 2.
\]
As $Z$ consists of $\pa_t$ and the angular momentum $\Om$, to apply Lemma \ref{lem:trace}, we need the energy flux of the tangential derivative $\Lb(\ab)$. We make use of the structure equation \eqref{eq:eq4ab} which implies that
\begin{align*}
\int_{\Hb_{\tau}}|\Lb\mathcal{L}_Z^k(r\ab)|^2dud\om&\les \int_{\Hb_{\tau}}|L\mathcal{L}_Z^k(r\ab)|^2+|\mathcal{L}_{\pa_t}\mathcal{L}_Z^k(r\ab)|^2dud\om\\
&\les \int_{\Hb_{\tau}}|\mathcal{L}_Z^{k+1}\rho|^2+|\mathcal{L}_Z^{k+1}\si|^2+|\mathcal{L}_Z^k(r\J)|^2+|\mathcal{L}_Z^{k+1}(r\ab)|^2dud\om\\
&\les E[\mathcal{L}_Z^{k+1}F](\Hb_{\tau})+I^{0}_{0}[\mathcal{L}_Z^{k+1}\J](\mathcal{D}_{\tau})\\
&\les M_2 \tau_+^{-1-\ga_0},\quad k\leq 1.
\end{align*}
Here note that $\Om=(r e_1, re_2)$. Then by Lemma \ref{lem:trace}, for all $v$ and fixed $\tau$,
\[
\left(\int_{\om}|r\mathcal{L}_Z^k\ab|^4(\tau, v, \om)d\om\right)^{\f12}\les M_2\tau_+^{-1-\ga_0},\quad k\leq 1.
\]
Estimate \eqref{eq:supab:p} then follows by using Sobolev embedding on the sphere.

\bigskip

For the pointwise bound \eqref{eq:suparhosi:p}, \eqref{eq:supbarrho:p} for $\a$, $\si$, $\bar\rho$, the proof for $\a$ is slightly different to that of $\si$ and $\bar\rho$. However the idea is the same. Let's consider $\a$ first. Consider $H_{\tau^*}$, $\tau\in \mathbb{R}$. The $r$-weighted energy estimates \eqref{eq:pWE:cur:ex}, \eqref{eq:pWE:cur:in} imply that
\[
\int_{H_{\tau^*}}r^{p}|r\mathcal{L}_Z^k\a|^2dvd\om\les M_2\tau_+^{p-1-\ga_0},\quad \forall 0\leq p\leq 1+\ga_0,\quad k\leq 2.
\]
To apply Lemma \ref{lem:trace}, we need the energy flux of the tangential derivative $L(r\a)$. Similar to the case of $\a$, we make use of the equation \eqref{eq:eq4a} and $\pa_t$ derivative:
\begin{align*}
\int_{H_{\tau^*}}r^{p}|L(r\mathcal{L}_Z^k\a)|^2dvd\om &\les \int_{H_{\tau^*}}r^p(|\Lb(r\mathcal{L}_Z^k\a)|^2+|r\pa_t\mathcal{L}_Z^k\a|^2)dvd\om\\
&\les  \int_{H_{\tau^*}}r^p(|\mathcal{L}_Z^{k+1}\rho|^2+|\mathcal{L}_Z^{k+1}\si|^2+|\mathcal{L}_Z^k(r\J)|^2+|\mathcal{L}_Z^{k+1}(r\a)|^2)dvd\om\\
&\les  M_2\tau_+^{p-1-\ga_0}+\int_{H_{\tau^*}}r^2(|\mathcal{L}_Z^{k+1}\rho|^2+|\mathcal{L}_Z^{k+1}\si|^2)+r^p|\mathcal{L}_Z^k(r\J)|^2dvd\om\\
&\les M_2\tau_+^{p-1-\ga_0}+E[\mathcal{L}_Z^k F](H_{\tau^*})+I^{p}_0[\mathcal{L}_Z^{k+1}\J](\mathcal{D}_{\tau})\\
&\les M_2\tau_+^{p-1-\ga_0}
\end{align*}
for $k\leq 1$.  Estimate for $\a$ then follows by Lemma \ref{lem:trace} together with Sobolev embedding on the sphere.

For $\bar \rho$, $\si$, we make use of the $r$-weighted energy estimates \eqref{eq:pWE:cur:ex}, \eqref{eq:pWE:cur:in} through the incoming null hypersurface $\Hb_{\tau}$ defined as above. First we have
\[
\int_{\Hb_{\tau}}r^{p-2}(|\mathcal{L}_Z^k(r^2\bar\rho)|^2+|\mathcal{L}_Z^k(r^2\si)|^2)dud\om\les M_2\tau_+^{p-1-\ga_0},\quad k\leq 2.
\]
To derive the tangential derivative $\Lb(r^2\bar\rho)$, $\Lb(r^2\si)$, we use the equations \eqref{eq:eq4rhoCu}, \eqref{eq:eq4si}. We can show that
\begin{align*}
\int_{\Hb_{\tau}}r^{p-2}(|\Lb(r^2\mathcal{L}_Z^k\bar\rho)|^2dud\om &\les \int_{\Hb_{\tau}}r^{p-2}(|r\mathcal{L}_Z^{k+1}\ab|^2+|r^2\mathcal{L}_Z^k J_{\Lb}|^2dud\om\\
&\les E[\mathcal{L}_Z^{k+1}F](\Hb_{\tau})+I^{p}_{0}[\mathcal{L}_Z^{k+1} J_{\Lb}](\mathcal{D}_{\tau})\\
&\les M_2\tau_+^{p-1-\ga_0},\quad k=0, \quad 1
\end{align*}
for all $0\leq p\leq 1-\ep$. We are not able to extend $p$ to the full range of $[0, 1+\ga_0]$ is due to the assumption on $J_{\Lb}$. The equation \eqref{eq:eq4si} for $\si$ does not involve $J_{\Lb}$. We hence have the full range $0\leq p\leq 1+\ga_0$ for $\si$.
Lemma \ref{lem:trace} and Sobolev embedding on the sphere then lead to the pointwise bound for $\bar\rho$ and $\si$. We thus have shown \eqref{eq:supbarrho:p}, \eqref{eq:suparhosi:p}.

Finally for the integrated decay estimates \eqref{eq:suprhosi:I}, we show it by integrating along the incoming null hypersurface. In the interior region case we integrate from from $r=R$ while in the exterior region we integrate from the initial hypersurface $t=0$. Let's only prove \eqref{eq:suprhosi:I} for the interior region case. In particular take $\bar{\mathcal{D}}_{\tau}$ to be $\bar{\mathcal{D}}_{\tau_1}^{\tau_2}$ for $0\leq \tau_1< \tau_2$. First by using the decay estimate \eqref{eq:Est4F:in:R} for $F$ when $r\leq R$, we can show that on the boundary $r=R$
\begin{align*}
\int_{\tau_1}^{\tau_2}|\mathcal{L}_Z^k F|^2(\tau, R,\om)d\om d\tau\les \int_{\tau_1}^{\tau_2}\int_{r\leq R}|\pa \mathcal{L}_Z^k F|^2dxd\tau\les M_2( \tau_1)_+^{-1-\ga_0}.
\end{align*}
Then from the transport equation \eqref{eq:eq4rhoCu}, \eqref{eq:eq4si}, we can show that
\begin{align*}
\|r\mathcal{L}_Z^k\si\|_{L_v^2L_u^\infty L_\om^2(\bar{\mathcal{D}}_{\tau_1}^{\tau_2})}^2 
&\les \int_{\tau_1}^{\tau_2}|\mathcal{L}_Z^k F|^2(\tau, R, \om)d\om d\tau+\iint_{\bar{\mathcal{D}}_{\tau_1}^{\tau_2}}r|\mathcal{L}_Z^k\si|^2 +|(\mathcal{L}_Z^k\si)\cdot \Lb(r^2\mathcal{L}_Z^k\si)|dudvd\om\\
&\les M_2(\tau_1)_+^{-1-\ga_0}++\iint_{\bar{\mathcal{D}}_{\tau_1}^{\tau_2}}r^{1+\ep}|\mathcal{L}_Z^k\si|^2+r^{1-\ep}|\mathcal{L}_Z^{k+1}\ab|^2dudvd\om\\
&\les M_2(\tau_1)_+^{-1-\ga_0}+M_2(\tau_1)_+^{-1-\ga_0+\ep}\les M_2(\tau_1)_+^{-1-\ga_0+\ep}.
\end{align*}
Here we have used the $r$-weighted energy estimates for $\si$ with $p=\ep$ and the integrated local energy estimates to bound $\ab$. This proves \eqref{eq:suprhosi:I}.
\end{proof}

\subsection{Energy decay for the scalar field}
\label{sec:EnEst4sca}
In this section, we study the energy decay for the complex scalar field $\phi$ satisfying the linear covariant wave equation $\Box_A \phi=0$. When the connection field $A$ is trivial, the energy decay has
been well studied by using the new approach, see e.g. \cite{yang1}. For general connection field $A$, presumably not small, new difficulty arises as there are interaction terms between the curvature $dA$ and the
scalar field. In the previous subsection, we derived the energy flux decay for the Maxwell field $F=dA$ with appropriate bound on $J$. The purpose of this section is to derive energy flux decay for the complex scalar field.

In addition to the assumption that $M_k$ is finite, for the complex scalar field, we assume the inhomogeneous term $\Box_A\phi$ and the initial data are bounded in the following norm:
\begin{equation}
\label{eq:def4Nkphi}
\begin{split}
 \mathcal{E}_k[\phi]=\sum\limits_{l\leq k}E^k_0[\phi]&+I^{1+\ga_0}_{1+\ep}[D_Z^l\Box_A\phi](\{t\geq 0\})+I^{1+\ep}_{1+\ga_0}[D_Z^l\Box_A \phi](\{t\geq 0\}).
\end{split}
\end{equation}
For solutions of \eqref{EQMKG}, $\mathcal{E}_{k}[\phi]$ denotes the weighted Sobolev norm of the initial data for the complex scalar field.
As the estimates in the interior region requires the information on the boundary $\Si_{0}$ of which $H_{0^*}$ is the boundary of the exterior region. Thus we need to obtain the energy decay estimates,
at least the boundedness of the energy flux in the exterior region. The main difficulty in the presence of nontrivial connection field is to control the interaction term $(dA)_{\cdot \nu}J^\nu[\phi]$ under very
weak estimates on the curvature $dA$. In the integrated local energy estimate \eqref{eq:ILE:Sca:ex} for the scalar field, it is not possible to control or absorb those terms as there is no smallness assumption on $dA$.
The idea is to make use of the null structure of $J^\nu[\phi]$ together with the $r$-weighted energy estimate \eqref{eq:pWE:sca:ex}. More precisely, we first control those terms in the $r$-weighted energy estimate through Gronwall's
inequality. Then we estimate those terms in the integrated local energy estimates. Once we have control on those interaction terms, the decay of the energy flux follows from the standard argument of the new approach, similar to that
of the energy decay for the Maxwell field in the previous section.

\bigskip

We need a lemma to control the scalar field $\phi$ by using the $r$-weighted energy.
\begin{lem}
 \label{lem:Est4phipWE}
 Assume $\phi$ vanishes at null infinity. In the exterior region on $H_{u}$, we have
 \begin{equation}
  \label{eq:Est4phipWE:ex}
  \int_{\om}|r\phi|^2(u, v, \om)d\om\les \int_{\om}|r\phi|^2(u, -u,\om)d\om+\b^{-1}u_+^{-\b}\int_{-u}^{v}\int_{\om}r^{1+\b}|D_L(r\phi)|^2 dvd\om,\quad \forall \b>0.
 \end{equation}
In the interior region on $\Si_{\tau}$, for $1\leq p\leq 2$, we have
\begin{equation}
\label{eq:Est4phipWE:in:p}
 \int_{\om}r^p|\phi|^2d\om\les \left(E[\phi](\Si_{\tau})\right)^{\delta_p}\left(I^{1+\ga_0}_0[r^{-1}D_L\psi](H_{\tau^*})\right)^{1-\delta_p},\quad \delta_p=\frac{2+\ga_0-p}{1+\ga_0}.
 \end{equation}
 Moreover on $\Si_{\tau}$, $\tau\in\mathbb{R}$, we have
 \begin{equation}
 \label{eq:Est4phiE:small}
 r\int_{\om}|\phi|^2d\om\les \ep_1^{-1}\int_{\Si}|\phi|^2d\tilde{v}d\om+\ep_1 E[\phi](\Si_{\tau})
 \end{equation}
 for all $1\geq \ep_1> 0$. Here $(\tilde{v}, \om)=(v, \om)$ when $r\geq R$ and $(r,\om)$ when $r\leq R$.
\end{lem}
\begin{proof}
 Estimate \eqref{eq:Est4phipWE:ex} follows from the inequality
 \[
  |r\phi|(u, v,\om)\leq |r\phi|(u, -u, \om)+\int_{-u}^{v}|D_L(r\phi)|dv
 \]
followed by the Cauchy-Schwarz's inequality.

In the interior region, for $\psi=r\phi$, the problem is that we can not integrate from the initial hypersurface or the boundary $H_{0^*}$ or the null infinity as the behaviour of $r\phi$ at null infinity is unknown (generically not zero).
However the scalar field $\phi$ vanishes at null infinity. We thus can bound $r|\phi|^2$ by the energy flux through $\Si$ with $\Si=\Si_{\tau}$ or $H_u$. More
precisely, on $\Si$ we can show that
\begin{align*}
r\int_{\om}|\phi|^2d\om&\les \int_{\Si}|\phi|^2d\tilde{v}d\om +\int_{\Si}r|D_{\tilde{v}}\phi||\phi|d\tilde{v}d\om\\
&\les \ep_1\int_{\Si}|D_{\tilde{v}}\phi|^2 r^2 d\tilde{v}d\om +(\ep_1^{-1}+1)\int_{\Si}|\phi|^2d\tilde{v}d\om\\
&\les \ep_1 E[\phi](\Si)+\ep_1^{-1}\int_{\Si}|\phi|^2d\tilde{v}d\om.
\end{align*}
This gives estimate \eqref{eq:Est4phiE:small}. In particular for $\ep_1=1$, from Hardy's inequality \eqref{eq:Hardy}, we conclude that estimate \eqref{eq:Est4phipWE:in:p} holds for $p=1$. To prove it for all $1\leq p\leq 2$, it suffices to show the estimate with $p=2$.
Consider the sphere with radius $r=\frac{\tau^*+v}{2}$ on
$H_{\tau^*}\subset \Si_{\tau}$. Choose the sphere with radius $r_1=\frac{\tau^*+v_1}{2}$ such that
\[
 r_1^{1+\ga_0}=E[\phi](\Si_{\tau})^{-1}\int_{H_{\tau^*}}r^{1+\ga_0}|D_L(r\phi)|^2dvd\om.l
\]
If $r\leq r_1$, then \eqref{eq:Est4phipWE:in:p} with $p=2$ follows from \eqref{eq:Est4phipWE:in:p} with $p=1$. Otherwise we have $r_1< r$.
Then
\begin{align*}
 \int_{\om}|r\phi|^2(\tau^*, v,\om)d\om&\les \int_{\om}|r\phi|^2(\tau^*, v_1, \om )+r_1^{-\ga_0}\int_{H_{\tau^*}}r^{1+\ga_0}|D_L(r\phi)|^2dvd\om\\
 &\les r_1 E[\phi](\Si_{\tau})+r_1^{-\ga_0}I^{1+\ga_0}_0[r^{-1}D_L\psi](H_{\tau^*})\\
 &\les \left(E[\phi](\Si_{\tau})\right)^{\frac{\ga_0}{1+\ga_0}}\left(I^{1+\ga_0}_0[r^{-1}D_L\psi](H_{\tau^*})\right)^{\frac{1}{1+\ga_0}}.
\end{align*}
Here we recall the notation $I$ defined in Section \ref{sec:notation}.
\end{proof}
The following lemma is quite simple but it turns out to be very useful.
\begin{lem}
 \label{lem:simplint}
Suppose $f(\tau)$ is smooth. Then for any $\b\neq 0$, we have the identity
\begin{equation*}
\int_{\tau_1}^{\tau_2}s^\b
f(s)ds=\b\int_{\tau_1}^{\tau_2}\tau^{\b-1}\int_{\tau}^{\tau_2}f(s)ds
d\tau+\tau_1^{\b}\int_{\tau_1}^{\tau_2}f(s)ds.
\end{equation*}
\end{lem}

\subsubsection{Energy decay in the exterior region}
In the exterior region, as $r\geq \frac{1}{3}u_+$, it suffices to consider the $r$-weighted energy estimate for the largest $p=1+\ga_0$. First we can show that
\begin{prop}
 \label{prop:Est4FJ:ex}
 In the exterior region for all $\tau_2\leq \tau_1\leq 0$, we have
 \begin{equation}
  \label{eq:Est4FJ:ex}
  \begin{split}
  \iint_{\mathcal{D}_{\tau_1}^{\tau_2}}r^{1+\ga_0}| F_{L\mu}J^{\mu}[\phi] |dxdt\les &  M_2 E_0^0[\phi]+M_2 \int_{u}u_+^{-1-\ep}\int_{v}r^{1+\ga_0}|D_L\psi|^2dvd\om du\\
  &+ |q_0|\iint_{\mathcal{D}_{\tau_1}^{\tau_2}}r^{\ga_0}(|D_L(r\phi)|^2 +|\D(r\phi)|^2)dvdud\om.
 \end{split}
 \end{equation}
\end{prop}
\begin{proof}
As $F=dA$ has different decay properties for different components, we estimate the integral according to the index $\mu$. Note that $ r^2 J[\phi]=J[r\phi]$.
For $\mu=\Lb$, we have
\begin{equation}
\label{eq:FLLbJ}
 |F_{L\Lb}J^{\Lb}[\phi]|\les r^{-2}|q_0| |D_L\psi||\psi|+|\bar \rho||D_L\psi||\psi|,\quad \psi=r\phi.
\end{equation}
The first term on the right hand side will be absorbed with the smallness assumption on the charge $q_0$ (as the data for the scalar field is small). In deed, by using Lemma \ref{lem:Hardyga}, we can show that
\begin{align*}
 2\iint r^{\ga_0-1}|D_L\psi||\psi|dudvd\om & \leq \iint r^{\ga_0} |D_L\psi|^2dvdud\om +\iint r^{\ga_0}|\phi|^2dvdud\om \\
 &\les \iint r^{\ga_0} |D_L\psi|^2dvdud\om+\int_{u}\int_{\om }(r^{1+\ga_0}|\phi|^2)(u, -u, \om )d\om du\\
 &\les \iint r^{\ga_0} |D_L\psi|^2dvdud\om+E_0^0[\phi].
 \end{align*}
For the second term on the right hand side of \eqref{eq:FLLbJ}, the idea is that we use Cauchy-Schwarz inequality and make use of the $r$-weighted energy estimate. First we can estimate that
\begin{align*}
2r^{1+\ga_0}|\bar \rho||D_L\psi||\psi| \leq r^{1+\ga_0}|D_L\psi|^2 u_+^{-1-\ep}+u_+^{1+\ep}r^2|\bar\rho|^2r^{1+\ga_0}|\phi|^2.
\end{align*}
The first term will be controlled through Gromwall's inequality. For the second term, we can first use Sobolev embedding on the unit sphere to bound $\bar \rho$ and then apply Lemma \ref{lem:Est4phipWE}:
\begin{align*}
 &\iint u_+^{1+\ep}r^2|\bar\rho|^2r^{1+\ga_0}|\phi|^2 dudvd\om \\
 &\les \int_{u}u_+^{1+\ep}\int_{v} \sum\limits_{j\leq 2}r^2\int_{\om}|\mathcal{L}_{\Om}^j\bar \rho|^2 d\om \cdot \int_{\om}r^{1+\ga_0}|\phi|^2d\om dvd u\\
 &\les \int_{u} u_+^{1+\ep -1} E^2[\bar F](H_{u}) (u_+^{\ga_0}\int_{\om}|r\phi|^2(u, -u, \om)d\om +\int_{v}\int_{\om}r^{1+\ga_0}|D_L\psi|^2 dvd\om) du\\
 &\les M_2\int_{|x|\geq R} r_+^{1+\ga_0-\ep-2}|\phi|^2(0, x)dx+M_2\int_{u}u_+^{-1-\ep}\int_{v}r^{1+\ga_0}|D_L\psi|^2dvd\om du.
\end{align*}
The first term is bounded by the weighted Sobolev norm of the initial data. The second term can be controlled by using Gronwall's inequality. Thus estimate \eqref{eq:Est4FJ:ex} holds for the case of $\mu=L$.

For $\mu=e_1$ or $e_2$, first we can bound
\begin{align*}
 r^{1+\ga_0}|F_{L e_j}||J^{e_j}[\psi]|\leq \ep_1 r^{\ga_0}|\nabb \psi|^2+\ep_1^{-1} r^{3+\ga_0}|\a|^2 r|\phi|^2, \quad \forall \ep_1>0.
\end{align*}
We choose sufficiently small $\ep_1$ so that the integral of the first term can be absorbed. For the second term, we first use Sobolev embedding on the unit sphere to bound $\a$ and then
Lemma \ref{lem:Est4phipWE} to control $\phi$:
\begin{align*}
 &\iint r^{3+\ep}|\a|^2 r^{1+\ga_0-\ep}|\phi|^2 dudvd\om\\
 & \les \int_{u} u_+^{\ga_0-\ep-1}\int_{v} r^{3+\ep}\sum\limits_{j\leq 2}\int_{\om }|\mathcal{L}_{\Om}^j \a|^2d\om \cdot \left(\int_{\om}|r\phi|^2(u, -u, \om )d\om +u_+^{-\ga_0}
  \int_{v}\int_{\om}r^{1+\ga_0}|D_L\psi|^2dvd\om \right)du\\
  &\les M_2\int_{|x|\geq R}r_+^{1+\ga_0-\ep -2}|\phi|^2(0, x)dx+M_2\int_{u}u_+^{-1-\ep}\int_{v}r^{1+\ga_0}|D_L\psi|^2dvd\om du\\
  &\les M_2E_0^0[\phi]+M_2\int_{u}u_+^{-1-\ep}\int_{v}r^{1+\ga_0}|D_L\psi|^2dvd\om du.
\end{align*}
 As the data for the scalar field is small, the charge is also small. In particular we can choose $\ep_1=|q_0|$ (if $q_0=0$, let $\ep_1$ small depending only on $\ep$, $\ga_0$ land $R$).
Therefore estimate \eqref{eq:Est4FJ:ex} holds for the case when $\mu=e_1$ or $e_2$. We thus finished the proof for the Proposition.
\end{proof}
As a corollary, we show the $r$-weighted energy flux decay of the scalar field in the exterior region.
\begin{cor}
 \label{cor:pWEdecay:sca:ex}
 Assume that the charge $q_0$ is sufficiently small, depending only on $\ep$, $R$, $\ga_0$. Then in the exterior region, we have the energy flux decay
 \begin{equation}
  \label{eq:pWEdecay:sca:ex}
  \begin{split}
   &\int_{H_{\tau_1^*}}r^p|D_L\psi|^2dvd\om+\iint_{\mathcal{D}_{\tau_1}}r^{p-1}(p|D_L\psi|^2+|\D\psi|^2)dvd\om du+\int_{\Hb_{-\tau_2^*}^{\tau_1^*}}r^p|\D\psi|^2dud\om \\
   &\les_{M_2} \mathcal{E}_0[\phi] (\tau_1)_+^{p-1-\ga_0},\quad \forall 0\leq p\leq 1+\ga_0,\quad \forall \tau_2\leq \tau_1\leq 0.
  \end{split}
 \end{equation}
\end{cor}
\begin{proof}
 It suffices to prove the corollary for $p=1+\ga_0$. For sufficiently small $q_0$ depending only on $\ep$, $\ga_0$ and $R$, from the $r$-weighted energy estimate \eqref{eq:pWE:sca:ex} and the estimate \eqref{eq:Est4FJ:ex} for the error term,
 the integral of $r^{\ga_0}|D_L(r\phi)|^2$ can be absorbed. Then estimate \eqref{eq:pWEdecay:sca:ex} follows from Gronwall's inequality.
\end{proof}
Next we make use of the $r$-weighted energy decay to show the energy flux decay and the integrated energy decay for the scalar field in the exterior region. In the sequel we assume the charge $q_0$ is sufficiently small so that estimate \eqref{eq:pWEdecay:sca:ex} of Corollary \ref{cor:pWEdecay:sca:ex} holds in the exterior region $\tau_2\leq \tau_1\leq 0$. From the integrated energy estimate \eqref{eq:ILE:Sca:ex}, it suffices to bound
the interaction term of the gauge field and the scalar field.

\begin{prop}
 \label{prop:Est4bFJ:ex}
  In the exterior region for all $\tau_2\leq \tau_1\leq 0$, we have
 \begin{equation}
  \label{eq:Est4bFJ:ex}
  \begin{split}
   &\iint_{\mathcal{D}_{\tau_1}^{ \tau_2}}|F_{L\nu}J^\nu[\phi]|+|F_{\Lb\nu}J^\nu(\phi)| dxdt\\
   &\les \ep_1 I^{-1-\ep}_0[D\phi](\mathcal{D}_{\tau_1}^{\tau_2})+C_{M_2, \ep_1}\left(\mathcal{E}_0[\phi](\tau_1)_+^{-1-\ga_0}+(\tau_1)_+^{\ep}\int_{-\tau_1^*}^{-\tau_2^*}v^{-1-\ep}E[\phi](\Hb_{v}^{-\tau_1^*})dv\right)
  \end{split}
 \end{equation}
 for all $\ep_1>0$ and some constant $C_{M_2, \ep}$ depending on $M_2$ and $\ep_1$.
 \end{prop}
\begin{proof}
The integral of $(dA)_{L\nu}J^\nu[\phi]$ has been controlled in the previous Proposition \ref{prop:Est4FJ:ex} as Corollary \ref{cor:pWEdecay:sca:ex} implies that the right hand
side of \eqref{eq:Est4FJ:ex} can be bounded by a constant depending on $M_2$, $\ep$, $\ga_0$ and $R$. Since in the exterior region $r\geq \frac{1}{3} u_+$, we easily obtain the desired bound:
\begin{equation*}
\iint_{\mathcal{D}_{\tau_1}^{ \tau_2}}|F_{L\nu}J^\nu[\phi]|dxdt\les_{M_2}(\tau_1)_+^{-1-\ga_0}\mathcal{E}_0[\phi].
\end{equation*}
It remains to estimate the integral of $F_{\Lb \nu}J^{\nu}[\phi]$. The $r$-weighted energy decay gives control for the "good" derivative of the scalar field. The problem is that we do not have any control for the "bad" derivative $D_{\Lb}\phi$. What worse is the existence of nonzero charge. Although the
charge can be small, we are not able to absort the charge part $q_0 r^{-2}J_{\Lb}[\phi]$ in the integrated local energy estimate \eqref{eq:ILE:Sca:ex} as there is a small $\ep$ loss of decay in $I^{-1-\ep}_0[\bar D\phi]$ on the
left hand side. The idea to treat this term is to make use of the energy flux on the incoming null hypersurface $\Hb_{-u_2}^{u_1}$ and then apply Gronwall's inequality. Let's first consider the easier terms in the integral of
$F_{\Lb \nu}J^{\nu}[\phi]$. For $\nu=e_1$ or $e_2$, we have
\[
 |F_{\Lb \nu}J^{\nu}[\phi]|\les |\ab||\D\phi||\phi|.
\]
Note that from estimate \eqref{eq:Est4phipWE:ex} of Lemma \ref{lem:Est4phipWE} and Corollary \ref{cor:pWEdecay:sca:ex}, we obtain
\[
 \int_{\om}|r\phi|^2(u, v, \om)d\om \les_{M_2}u_+\int_{\om}(-2u)|\phi(0, -2u, \om)|^2 d\om +\mathcal{E}_0[\phi] u_+^{-\ga_0}.
\]
Here we parametrize $\phi$ in $(t, r, \om )$ coordinates. We then use Sobolev embedding on the initial hypersurface $t=0$ to derive the decay of $\phi$:
\begin{equation}
 \label{eq:SPHDecay:sca:ex}
 \int_{\om}|r\phi|^2(u, v, \om)d\om \les_{M_2}\mathcal{E}_0[\phi] u_+^{-\ga_0}.
\end{equation}
From the $r$-weighted energy estimate \eqref{eq:pWEdecay:sca:ex}, we have the weighted angular derivative of the scalar field on the incoming null hypersurface
\[
 \int_{\Hb_{-\tau_2^*}^{\tau_1^*}}r^{1+\ga_0}|\D(r\phi)|^2dud\om\les_{M_2}\mathcal{E}_0[\phi].
\]
In the exterior region, note that $r\geq\frac{v}{2}$.
Therefore we can show that
\begin{align*}
 &\iint_{\mathcal{D}_{\tau_1}^{ \tau_2}}|F_{\Lb e_j}||J^{e_j}[\phi]|dxdt \\
 &\les \int_{-\tau_1^*}^{-\tau_2^*}\int_{-v}^{\tau_1^*}\int_{\om}r^2|\ab||\D\phi||\phi|d\om dudv\\
 &\les \int_{-\tau_1^*}^{-\tau_2^*}\int_{-v}^{\tau_1^*}r^{-\frac{3+\ga_0}{2}}\left(r^2\sum\limits_{j\leq 2}\int_{\om}|\mathcal{L}_{\Om}^j \ab|^2d\om\right)^{\f12}\left(r^{3+\ga_0}\int_{\om}|\D\phi|^2d\om\cdot \int_{\om }|r\phi|^2d\om\right)^{\f12} dudv\\
&\les_{M_2}\mathcal{E}_0[\phi]^{\f12}(\tau_1)_+^{-\frac{\ga_0}{2}}\int_{-\tau_1^*}^{-\tau_2^*}v^{-\frac{3+\ga_0}{2}}(E^2[dA](\Hb_{v}^{-\tau_1^*}))^{\f12}\mathcal{E}_0[\phi]^{\f12}dv\\
&\les_{M_2}\mathcal{E}_0[\phi](\tau_1)_+^{-\frac{\ga_0}{2}-\frac{1+\ga_0}{2}-\frac{1}{2}}\les_{M_2}\mathcal{E}_0[\phi] (\tau_1)_+^{-1-\ga_0}.
 \end{align*}
When $\nu=L$, first we have
\[
 |F_{\Lb L}||J^{L}[\phi]|\les |q_0|r^{-2}|D_{\Lb }\phi||\phi|+|\bar \rho||D_{\Lb}\phi||\phi|.
\]
The second term is easy to bound. We may use Cauchy-Schwarz to absorb $D_{\Lb}\phi$. In deed,
\[
 2|\bar \rho||D_{\Lb}\phi||\phi|\leq \ep_1 |D_{\Lb}\phi|^2 r^{-1-\ep}+\ep_1^{-1}|\bar \rho|^2|\phi|^2 r^{1+\ep},\quad \forall \ep_1>0.
\]
For sufficiently small $\ep_1$, the integral of the first term on the right hand side can be absorbed from the integrated energy estimate \eqref{eq:ILE:Sca:ex}. For the second term, we make use of estimate \eqref{eq:SPHDecay:sca:ex} to show
that
\begin{align*}
 \iint_{\mathcal{D}_{\tau_1}^{\tau_2}}|\bar \rho|^2r^{3+\ep}|\phi|^2 dvdud\om &\les \int_{\tau_1^*}^{\tau_2^*}\int_{-u}^{-\tau_2^*}\sum\limits_{j\leq 2}\int_{\om}r^2|\mathcal{L}_{\Om}^j \bar\rho|^2d\om\cdot r^{1+\ep}\int_{\om}|\phi|^2d\om dvdu\\
&\les_{M_2} \int_{\tau_1^*}^{\tau_2^*}E^2[\bar F](H_u^{-\tau_2^*})u_+^{-1+\ep-\ga_0}\mathcal{E}_0[\phi] du\\
& \les_{M_2} \mathcal{E}_0[\phi]\int_{u_1^*}^{u_2^*}u_+^{-2-\ga_0}du\les_{M_2}\mathcal{E}_0[\phi] (\tau_1)_+^{-1-\ga_0}.
\end{align*}
Finally we need to bound the charge part, namely the integral of $|q_0|r^{-2}|D_{\Lb }\phi||\phi|$. As we have explained previously, this term can not be absorbed even though the charge $q_0$ is small due to the loss of decay
in the integrated local energy $I^{-1-\ep}_0[\bar D\phi](\mathcal{D}_{\tau_1}^{\tau_2})$ in \eqref{eq:ILE:Sca:ex}. The idea is to make use of the energy flux in the incoming null hypersurface $\Hb_{-\tau_2^*}^{\tau_1^*}$ and then apply Gronwall's inequality.
From estimate \eqref{eq:SPHDecay:sca:ex} and noting that $r\geq \frac{1}{2}v$ in the exterior region, we can show that
\begin{align*}
\iint_{\mathcal{D}_{\tau_1}^{ \tau_2}}r^{-2}|D_{\Lb}\phi||\phi|dxdt &\les \int_{-\tau_1^*}^{-\tau_2^*}\int_{-v}^{\tau_1}\int_{\om}|D_{\Lb}\phi||\phi|d\om dudv\\
 &\les \int_{-\tau_1^*}^{-\tau_2^*}\int_{-v}^{\tau_1}r^{-2}\left(r^{2}\int_{\om}|D_{\Lb}\phi|^2d\om\cdot \int_{\om }|r\phi|^2d\om\right)^{\f12} dudv\\
 &\les_{M_2} \mathcal{E}_0[\phi]^{\f12} \int_{-\tau_1^*}^{-\tau_2^*}v^{-\frac{3+\ga_0-\ep}{2}}\int_{-v}^{\tau_1^*}r^{-\frac{1-\ga_0+\ep}{2}}\left(r^{2}\int_{\om}|D_{\Lb}\phi|^2d\om\right)^{\f12} u_+^{-\frac{\ga_0}{2}} dudv\\
&\les_{M_2} \mathcal{E}_0[\phi]^{\f12}\int_{-\tau_1^*}^{-\tau_2^*}v^{-\frac{3+\ga_0-\ep}{2}}(E[\phi](\Hb_{v}^{-\tau_1^*}))^{\f12}(\tau_1)_+^{-\frac{\ep}{2}}dv\\
&\les_{M_2}\mathcal{E}_0[\phi](\tau_1)_+^{-1-\ga_0}+(\tau_1)_+^{\ep}\int_{-\tau_1^*}^{-\tau_2^*}v^{-1-\ep}E[\phi](\Hb_{v}^{-\tau_1^*})dv.
\end{align*}
Combining all the previous estimates, we then have shown \eqref{eq:Est4bFJ:ex}.
\end{proof}
As a corollary we then can show the energy flux decay as well as the integrated local energy decay of the scalar field in the exterior region.
\begin{cor}
 \label{cor:ILEdecay:sca:ex}
For all $\tau_2<\tau_1\leq 0$, we have:
 \begin{equation}
  \label{eq:ILEdecay:sca:ex}
  \begin{split}
   & I^{-1-\ep}_0[\bar D\phi](\mathcal D_{\tau_1}^{\tau_2})+E[\phi](H_{\tau_1^*}^{-\tau_2^*})+E[\phi](\Hb_{-\tau_2^*}^{\tau_1^*})\les_{M_2} (\tau_1)_+^{-1-\ga_0}\mathcal{E}_0[\phi].
  \end{split}
 \end{equation}
\end{cor}
\begin{proof}
First choose $\ep_1$ in the estimate \eqref{eq:Est4bFJ:ex} to be sufficiently small, depending only on $\ep$, $\ga_0$ and $R$, so that after combing
 estimate \eqref{eq:Est4bFJ:ex} and the integrated energy estimate \eqref{eq:ILE:Sca:ex}, the term $\ep_1 I^{-1-\ep}_0[D\phi](\mathcal{D}_{\tau_1}^{\tau_2})$ on the right hand side of \eqref{eq:Est4bFJ:ex} can be absorbed by
$I^{-1-\ep}_0[\bar D\phi](\mathcal{D}_{\tau_1}^{\tau_2})$ on the left hand side of \eqref{eq:ILE:Sca:ex}. Then notice that we have the uniform bound:
 \[
  (\tau_1)_+^{\ep}\int_{-\tau_1^*}^{-\tau_2^*}v^{-1-\ep}dv\les 1,\quad \forall \tau_2<\tau_1\leq 0.
 \]
By using Gronwall's inequality (fix $\tau_1\leq 0$ and take $\tau_2\leq \tau_1$ as variable), we then obtain \eqref{eq:ILEdecay:sca:ex}.
\end{proof}

\subsubsection{Energy decay in the interior region}

Once we have the energy flux and the $r$-weighted energy decay estimates for the scalar field in the exterior region, we in particular have the energy flux bound for the scalar field on the boundary $H_{-\frac{R}{2}}$. This is necessary to
consider the energy flux decay in the interior region. Compared to the case in the exterior region, the charge is not a problem as the charge only effects the decay property of the Maxwell field in the exterior region. However, new difficulties
 arise in the interior region. First of all there is no lower bound for $\frac{r}{\tau_+}$. That means we may need estimates for general $p$ for the $r$-weighted energy estimate instead of simply the largest $p$. Secondly as we
 have explained before that we are not able to absorb the interaction term between the gauge field $A$ and the scalar field due to the fact that $dA$ is no longer small in our setting. Thus we need to rely on the $r$-weighted
 energy estimates and make use of the null structure of $J[\phi]$. In the exterior region, the idea is first to derive the
 $r$-weighted energy decay and then to obtain the integrated local energy and energy flux decay. In the interior region, we see from the $r$-weighted
 energy estimates \eqref{eq:pWE:sca:in} that the term $|F_{\Lb \mu}J^{\mu}[\phi]|$ also appears on the right hand side. This suggests that we have to consider the $r$-weighted energy estimate and the integrated local energy estimates
 simultaneously.

As the boundedness of the energy flux on the boundary $H_{-\frac{R}{2}}=H_{0^*}$ requires the smallness of the assumption that the charge $|q_0|$ is small. In the sequel when we
consider estimates in the interior, we always assume this without repeating it. Let's first estimate the interaction terms of $dA$ and the scalar field in the $r$-weighted energy estimate \eqref{eq:pWE:sca:in}.
 \begin{prop}
  \label{prop:Est4FJ:sca:in}
  In the interior region, for all $0\leq \tau_1\leq \tau_2$ and $1\leq p\leq 1+\ga_0$, we have
  \begin{align}
  \notag
   \iint_{\bar{\mathcal{D}}_{\tau_1}^{\tau_2}}&r^p|F_{L\mu}J^{\mu}[\phi]|^2 dxdt\les \ep_1\iint_{\bar{\mathcal{D}}_{\tau_1}^{\tau_2}}r^{p-1}|\D(r\phi)|^2dvd\om d\tau+
   I^p_{-1-\ep}[r^{-1}D_L(r\phi)](\bar{\mathcal{D}}_{\tau_1}^{\tau_2})\\
    \label{eq:Est4FJ:sca:in}
   &+M_2\ep_1^{-1}\left(\delta_p\int_{\tau_1}^{\tau_2}E[\phi](\Si_{\tau})\tau_+^{\delta_p^{-1}-1-\ep}d\tau+(1-\delta_p)I^{1+\ga_0}_{-1-\ep}[r^{-1}D_L\psi](\bar {\mathcal{D}}_{\tau_1}^{\tau_2})\right)
  \end{align}
  for all $\ep_1>0$. Here $\delta_p=\frac{2+\ga_0-p}{1+\ga_0}$ is given in Lemma \ref{lem:Est4phipWE} in line \eqref{eq:Est4phipWE:in:p}.
  \end{prop}
\begin{proof}
Denote $\psi=r\phi$ and $F=dA$. First we have
  \begin{align*}
 2r^p|F_{L\mu}J^{\mu}[\phi]|r^2\leq r^p|D_L\psi|^2\tau_+^{-1-\ep}+r^p |\rho|^2 |\psi|^2 \tau_+^{1+\ep}+\ep_1 r^{p-1}|\D\psi|^2+\ep_1^{-1}r^{p+3}|\a|^2|\phi|^2
 \end{align*}
 for all $\ep_1>0$. The first term can be absorbed by using Gronwall's inequality. The third term will be absorbed for sufficiently small $\ep_1$ depending only on $\ep$, $\ga_0$ and $R$. For the second term,
 we use the energy flux of $\rho$ on $H_{\tau^*}$ to bound $\rho$ and estimate \eqref{eq:Est4phipWE:in:p} of Lemma \ref{lem:Est4phipWE} to bound $\phi$. For the last term, we use the $r$-weighted energy estimate to bound $\a$.
Then similar to the proof of Proposition \ref{prop:Est4FJ:ex} we can show that
\begin{align*}
 &\int_{\tau_1}^{\tau_2}\int_{H_{\tau^*}}\tau_+^{1+\ep}r^p |\rho|^2|\psi|^2+r^{p+3}|\a|^2|\phi|^2dvd\om d\tau\\
 &\les \int_{\tau_1}^{\tau_2}\int_{2R+\tau^*}^{\infty}\sum\limits_{j\leq 2}\int_{\om}\tau_+^{1+\ep} r^2|\mathcal{L}_{\Om}^j \rho|^2+r^3|\mathcal{L}_{\Om}^j\a|^2d\om\cdot
\int_{\om} r^{p}|\phi|^2d\om dvdu\\
&\les M_2\int_{\tau_1}^{\tau_2}\tau_+^{-\ep} \left(E[\phi](\Si_{\tau})\right)^{\delta}\left(I^{1+\ga_0}_0[r^{-1}D_L\psi](H_{\tau^*})\right)^{1-\delta} d\tau\\
&\les M_2\left(\delta\int_{\tau_1}^{\tau_2}E[\phi](\Si_{\tau})\tau_+^{\delta^{-1}-1-\ep}d\tau+(1-\delta)I^{1+\ga_0}_{-1-\ep}[r^{-1}D_L\psi](\bar {\mathcal{D}}_{\tau_1}^{\tau_2})\right).
\end{align*}
The proposition then follows.
\end{proof}
Next we estimate the interaction terms in the energy estimate \eqref{eq:ILE:Sca:in}. We show the following:
\begin{prop}
 \label{prop:Est4FJ:sca:in:0}
 We have
 \begin{equation}
  \label{eq:Est4FJ:sca:in:0}
  \begin{split}
&\iint_{\mathcal{D}_{\tau_1}^{\tau_2}}|F_{L\nu}J^{\nu}[\phi]|+|F_{\Lb\nu}J^{\nu}[\phi]|dxdt\\
&\les \ep_1 I^{-1-\ep}_0[D\phi](\mathcal{D}_{\tau_1}^{\tau_2}) +\ep_1^{-1} \int_{\tau_1}^{\tau_2} g(\tau)E[\phi](\Si_\tau)d\tau+I_{-2-\ga_0}^{1+\ga_0}[r^{-1}D_L(r\phi)]
(\bar{\mathcal{D}}_{\tau_1}^{\tau_2})
  \end{split}
 \end{equation}
 for all $0<\ep_1<1$. Here
 \[
  g(\tau):=\sum\limits_{j\leq 2}I^{-1-\ep}_{1+2\ep}[\mathcal{L}_{\Om}^j F](\Si_{\tau})+\int_{H_{\tau^*}}r^{2+\ep}(|\mathcal{L}_{\Om}^j\a|^2+|\mathcal{L}_{\Om}^j\rho|^2)dvd\om+\sup\limits_{|x|\leq R}|F|^2(\tau, x).
 \]
\end{prop}
\begin{proof}
 For the integral on $\{r\geq R\}$, we use Sobolev embedding on the unit sphere to bound the curvature and the proof is quite similar to that of
 the previous proposition. On the finite region $r\leq R$, we make use of the $L^2_t L_x^\infty$ norm of the curvature given in Proposition \ref{prop:Est4F:in:R}. For the case when $r\geq R$, first we have
 \begin{align*}
  |F_{L\nu}J^{\nu}[\phi]|+|F_{\Lb\nu}J^{\nu}[\phi]|&\les (|\rho|+|\a|)|D\phi||\phi|+|\ab||\D\phi||\phi|\\
  &\les \ep_1 r_+^{-1-\ep}|D\phi|^2+\ep_1^{-1}(|\rho|^2+|\a|^2)r_+^{1+\ep}|\phi|^2+|\ab||\D\phi||\phi|.
 \end{align*}
The first term can be absorbed in the energy estimate \eqref{eq:ILE:Sca:in} for sufficiently small $\ep_1$.
For the second term, we can use estimate \eqref{eq:Est4phipWE:in:p} to bound $\phi$ by the energy flux through $H_{\tau^*}$ and the $r$-weighted energy to control the curvature terms. The last term is the most difficult term to control. The reason is that we do not have powerful estimates for $\ab$. The only estimates we have are the integrated local energy estimate and the energy flux through the incoming null hypersurface. Unlike the case in the
exterior region where we can make use of the energy flux through the incoming null hypersurface for $\ab$, this method fails in the interior region. The main reason is that the energy flux $E[F](\Hb_{v}^{\tau_1, \tau_2})$ decays in $\tau_1$ instead of $v$. A possible way to solve this issue is to assume pointwise bound for $\ab$. However the problem is that the pointwise decay for $\ab$ is too weak (due to the assumption on the initial data. We have explained this in the introduction) to be useful. We thus can only rely on the integrated local energy estimate for $\ab$. As there is a $r^{\ep}$ ldecay loss in the integrated local energy estimate for $\ab$, we are not able to bound $\phi$ simply by using the energy flux through $H_{\Si_{\tau^*}}$ but instead we need to make use of the $r$-weighted energy estimate. This means that we can not obtain uniform energy bound from the energy estimate \eqref{eq:ILE:Sca:in}. We need to combine it with the $r$-weighted energy estimate.

For the integral of $|\ab||\D\phi||\phi|$, from estimate \eqref{eq:Est4phipWE:in:p} with $p=1+\ep$, we can show that
\begin{align*}
&\int_{\tau_1}^{\tau_2}\int_{H_{\tau^*}}|\ab||\D\phi||\phi|r^2d\om dv d\tau\\
&\les \int_{\tau_1}^{\tau_2}\int_{2R+\tau^*}^\infty \left(\sum\limits_{j\leq 2}\int_{\om}r^{1-\ep}|\mathcal{L}_{\Om}^j\ab|^2d\om\right)^\f12\left(\int_{\om}r^2|\D\phi|^2d\om\cdot \int_{\om}r^{1+\ep}|\phi|^2d\om \right)^\f12 dv d\tau\\
&\les \sum\limits_{j\leq 2}\int_{\tau_1}^{\tau_2}\left(I^{-1-\ep}_0[\mathcal{L}_{\Om}^j \ab](\Si_{\tau})E[\phi](\Si_\tau)\right)^\f12 \left(E[\phi](\Si_\tau)\right)^{\f12\delta}\left(I_0^{1+\ga_0}[r^{-1}D_L(r\phi)](H_{\tau^*})\right)^{\f12-\f12\delta}d\tau\\
&\les \sum\limits_{j\leq 2}\int_{\tau_1}^{\tau_2} I^{-1-\ep}_{1+2\ep}[\mathcal{L}_{\Om}^j \ab](\Si_{\tau})E[\phi](\Si_\tau)d\tau+\int_{\tau_1}^{\tau_2}\tau_+^{-1-\ep}E[\phi](\Si_\tau)d\tau+I_{-2-\ga_0}^{1+\ga_0}[r^{-1}D_L(r\phi)]
(\bar{\mathcal{D}}_{\tau_1}^{\tau_2})
\end{align*}
Here $\delta=\frac{1+\ga_0-\ep}{1+\ga_0}$ and in the last step we have used Jensen's inequality as well as the relation
\[
 \f12+\ep-\f12 \delta(1+\ep)-(2+\ga_0)(\f12-\f12\delta)=\f12\frac{\ep}{1+\ga_0}>0.
\]
In the above estimate the first two terms will be estimated by using Gronwall's inequality. We keep the last term which involves the $r$-weighted energy estimates.
For the integral of $(|\rho|^2+|\a|^2)r_+^{1+\ep}|\phi|^2$, we use estimate \eqref{eq:Est4phipWE:in:p} to bound $\phi$. We have
\begin{align*}
 &\int_{\tau_1}^{\tau_2}\int_{H_{\tau^*}}(|\rho|^2+|\a|^2)r_+^{1+\ep}|\phi|^2r^2d\om dv d\tau\\
&\les \int_{\tau_1}^{\tau_2}\int_{2R+\tau^*}^\infty \sum\limits_{j\leq 2}\int_{\om}r^{2+\ep}(|\mathcal{L}_{\Om}^j\a|^2+|\mathcal{L}_{\Om}^j\rho|^2)d\om\cdot
\int_{\om}r|\phi|^2d\om  dv d\tau\\
&\les \sum\limits_{j\leq 2}\int_{\tau_1}^{\tau_2}\int_{H_{\tau^*}}r^{2+\ep}(|\mathcal{L}_{\Om}^j\a|^2+|\mathcal{L}_{\Om}^j\rho|^2)dvd\om \cdot E[\phi](\Si_{\tau})d\tau.
\end{align*}
This term will be controlled in the energy estimate \eqref{eq:ILE:Sca:in} by using Gronwall's inequality.

For the integral on the region $r\leq R$, we can show that
\begin{align*}
 \int_{\tau_1}^{\tau_2}\int_{r\leq R}|F_{L\nu}J^{\nu}[\phi]|+|F_{\Lb\nu}J^{\nu}[\phi]|dx d\tau &\les \ep_1\int_{\tau_1}^{\tau_2}\int_{r\leq R}|D\phi|^2 dxd \tau+\ep_1^{-1}\int_{\tau_1}^{\tau_2}\int_{r\leq R}|F|^2|\phi|^2dxd\tau\\
 &\les  \ep_1\int_{\tau_1}^{\tau_2}\int_{r\leq R}\frac{|D\phi|^2}{r_+^{1+\ep}} dxd \tau+\ep_1^{-1}\int_{\tau_1}^{\tau_2}\sup\limits_{|x|\leq R}|F|^2 \cdot E[\phi](\Si_{\tau}) d\tau
\end{align*}
for all $\ep_1>0$. The first term will be absorbed for small $\ep_1$. The second term
can be controlled by using Gronwall's inequality. Combining all these estimates above, we thus have shown estimate \eqref{eq:Est4FJ:sca:in:0} of the Proposition.l
\end{proof}
As a corollary the energy estimate \eqref{eq:ILE:Sca:in} lcan be reduced to the following:
\begin{cor}
 \label{cor:ILE:sca:in:sim0}
In the interior region, we have the following integrated local energy estimate
\begin{equation}
 \label{eq:ILE:sca:in:sim0}
\begin{split}
&I^{-1-\ep}_0[\bar D\phi](\mathcal{D}_{\tau_1}^{\tau_2})+E[\phi](\Si_{\tau_2})+\int_{\tau_1}^{\tau_2}\tau_+^{-1-\ep}E[\phi](\Si_{\tau})d\tau+
\iint_{\mathcal{D}_{\tau_1}^{\tau_2}}|F_{L\nu}J^{\nu}[\phi]|+|F_{\Lb\nu}J^{\nu}[\phi]|dxdt\\
&\les_{M_2} E[\phi](\Si_{\tau_1})+(\tau_1)_+^{-1-\ga_0}\mathcal{E}_0[\phi]+I_{-2-\ga_0}^{1+\ga_0}[r^{-1}D_L(r\phi)]
(\bar{\mathcal{D}}_{\tau_1}^{\tau_2}).
\end{split}
\end{equation}
\end{cor}
\begin{proof}
 First choose $\ep_1$ sufficiently small in the estimate \eqref{eq:Est4FJ:sca:in:0} so that combining the energy estimate \eqref{eq:ILE:Sca:in} with \eqref{eq:Est4FJ:sca:in:0} the integrated local energy
 $I^{-1-\ep}_0[D\phi](\mathcal{D}_{\tau_1}^{\tau_2})$ could be absorbed. By our notation, the smallness of $\ep_1$ depends only on $\ep$, $\ga_0$ and $R$. Then for the second term on the right hand side of \eqref{eq:Est4FJ:sca:in:0}, to
 apply Gronwall's inequality, we show that $g(\tau)$ (defined after line \eqref{eq:Est4FJ:sca:in:0}) is integrable. From the integrated local energy estimates \eqref{eq:Endecay:cur:in} and the $r$-weighted energy estimates \eqref{eq:pWE:cur:in} for the Maxwell field, we conclude from the previous section that
 \begin{align*}
 I^{-1-\ep}_0[\mathcal{L}_{Z}^k F](\mathcal{D}_{\tau_1}^{\tau_2})&\les M_k (\tau_1)_+^{-1-\ga_0},\quad \int_{\tau_1}^{\tau_2}\int_{H_{\tau^*}}r^{2+\ep}(|\mathcal{L}_{Z}^k \a|^2+|\mathcal{L}_Z^k \rho|^2)dvd\om d\tau\les M_k(\tau_1)_+^{-\ga_0+\ep}.
 \end{align*}
 Therefore by using Lemma \ref{lem:simplint} and Proposition \ref{prop:Est4F:in:R}, we can show that
 \begin{align*}
  \int_{\tau_1}^{\tau_2}g(\tau)d\tau&\les M_2(\tau_1)_+^{-\ga_0+\ep}+\sum\limits_{j\leq 2}I^{-1-\ep}_{1+2\ep}[\mathcal{L}_{\Om}^j F](\mathcal{D}_{\tau_1}^{\tau_2})\\
  &\les M_2(\tau_1)_+^{-\ga_0+\ep}+\sum\limits_{j\leq 2}\int_{\tau_1}^{\tau_2}\tau_+^{2\ep}I^{-1-\ep}_{0}[\mathcal{L}_{\Om}^j F](\mathcal{D}_{\tau}^{\tau_2})d\tau+(\tau_1)_+^{1+2\ep}I^{-1-\ep}_{0}[\mathcal{L}_{\Om}^j F](\mathcal{D}_{\tau_1}^{\tau_2})\\
  &\les M_2(\tau_1)_+^{-\ga_0+\ep}+M_2\int_{\tau_1}^{\tau_2}\tau_+^{-1-\ga_0+2\ep}d\tau+M_2(\tau_1)_+^{-\ga_0+2\ep}\les M_2(\tau_1)_+^{-\ga_0+2\ep}.
 \end{align*}
By using this uniform bound, the second term on the right hand side of \eqref{eq:Est4FJ:sca:in:0} can be absorbed by using Gronwall's inequality. The corollary then follows.
\end{proof}
We now can use Proposition \ref{prop:Est4FJ:sca:in} and the above corollary to obtain the necessary $r$-weighted energy estimates. To derive energy decay estimates, we at least need the $r$-weighted energy estimates with $p=1$ and $p=1+\ga_0$
(some $p$ bigger than one. However, the decay rate depends on this largest $p$). In any case, we first choose $\ep_1$ in estimate \eqref{eq:Est4FJ:sca:in} sufficiently small so that combining it with the $r$-weighted energy estimate
\eqref{eq:pWE:sca:in}, the first term on the right hand side of \eqref{eq:Est4FJ:sca:in} can be absorbed (note that $\ga_0<1$). The second term on the right hand side of \eqref{eq:Est4FJ:sca:in} can be controlled by using
Gronwall's inequality. Let's first combine the $r$-weighted energy estimate \eqref{eq:pWE:sca:in} for $p=1$ with the integrated local energy estimate \eqref{eq:ILE:sca:in:sim0} to derive the bound for the integral of
the energy flux.
\begin{prop}
 \label{prop:pWE:sca:in:1}
In the interior region for all $0\leq \tau_1<\tau_2$ we have
 \begin{equation}
  \label{eq:pWE:sca:in:1}
  \begin{split}
 \int_{\tau_1}^{\tau_2}E[\phi](\Si_\tau)d\tau\les_{M_2} &\int_{H_{\tau_1^*}}r|D_L\psi|^2dvd\om+E[\phi](\Si_{\tau_1})+(\tau_1)_+^{-\ga_0}\mathcal{E}_0[\phi]
 +I_{-2-\ga_0}^{1+\ga_0}[r^{-1}D_L(r\phi)](\bar{\mathcal{D}}_{\tau_1}^{\tau_2}).
\end{split}
 \end{equation}
\end{prop}
\begin{proof}
In $(t, r, \om )$ coordinate, using Sobolev embedding, we have
\[
 \int_{\om}|\phi|^2(\tau, R, \om)d\om\les \int_{r\leq R}|\phi|^2+|D\phi|^2dx.
\]
Then we can show that
\begin{align*}
 \int_{\tau_1}^{\tau_2}E[\phi](\Si_\tau)d\tau&\les \int_{\tau_1}^{\tau_2}\int_{r\leq R}|D\phi|^2dx d\tau +\int_{\tau_1}^{\tau_2}\int_{H_{\tau^*}}|D_L(r\phi)|^2+|\D(r\phi)|^2dvd\om d\tau\\
 &\qquad+\int_{\tau_1}^{\tau_2}\int_{\om}|\phi|^2(\tau, R, \om )d\om\\
 &\les I^{-1-\ep}_0[\bar D\phi](\mathcal{D}_{\tau_1}^{\tau_2})+\int_{\tau_1}^{\tau_2}\int_{H_{\tau^*}}|D_L(r\phi)|^2+|\D(r\phi)|^2dvd\om d\tau.
\end{align*}
Therefore take $p=1$ in the $r$-weighted energy esimate \eqref{eq:pWE:sca:in}. From the above argument, we obtain the following bound for the integral of the enrgy flux:
\begin{equation*}
  \begin{split}
 &\int_{\tau_1}^{\tau_2}E[\phi](\Si_\tau)d\tau\les \int_{H_{\tau_1^*}}r|D_L\psi|^2dvd\om+M_2\int_{\tau_1}^{\tau_2}E[\phi](\Si_{\tau})\tau_+^{-\ep}d\tau\\
&\quad+C_{M_2}\left(E[\phi](\Si_{\tau_1})+(\tau_1)_+^{-\ga_0}\mathcal{E}_0[\phi]+I_{-2-\ga_0}^{1+\ga_0}[r^{-1}D_L(r\phi)](\bar{\mathcal{D}}_{\tau_1}^{\tau_2})\right)
\end{split}
 \end{equation*}
 for some constant $C_{M_2}$ depending on $M_2$, $\ep$, $\ga_0$ and $R$. For the second term, we further can bound:
 \[
  \tau_+^{-\ep}=(\ep_1^{-\frac{1}{\ep}}\tau_+^{-1-\ep})^{\frac{\ep}{1+\ep}}\cdot (\ep_1)^{\frac{1}{1+\ep}}\leq \frac{\ep}{1+\ep} \ep_1^{-\frac{1}{\ep}}\tau_+^{-1-\ep}+\frac{\ep_1}{1+\ep},\quad \forall \ep_1>0.
  \]
Choose $\ep_1$ sufficiently small so that the second term can be absorbed. Then the first term can be bounded by using Corollary \ref{cor:ILE:sca:in:sim0}. Therefore the previous estimates amounts to
estimate \eqref{eq:pWE:sca:in:1}.
\end{proof}
As a corollary, we show that
\begin{cor}
 \label{cor:pWE:sca:in:1:Wga}
 We have
   \begin{equation}
  \label{eq:pWE:sca:in:1:Wga}
  \begin{split}
   \int_{\tau_1}^{\tau_2}\tau_+^{\ga_0-\ep}E[\phi](\Si_\tau)d\tau\les_{M_2}& \int_{H_{\tau_1^*}}r^{1+\ga_0}|D_L\psi|^2dvd\om+(\tau_1)_+^{1+\ga_0-\ep}E[\phi](\Si_{\tau_1})\\
&+\mathcal{E}_0[\phi]+I_{-2-\ep}^{1+\ga_0}[r^{-1}D_L(r\phi)](\bar{\mathcal{D}}_{\tau_1}^{\tau_2}).
\end{split}
 \end{equation}
\end{cor}
\begin{proof}
By using estimate \eqref{eq:Est4phipWE:in:p} of Lemma \ref{lem:Est4phipWE}, we have the bound
\[
 \int_{H_{\tau^*}}|D_L(r\phi)|^2dvd\om \leq \int_{H_{\tau^*}}|D_L\phi|^2r^2dvd\om +\lim\limits_{r\rightarrow\infty}\int_{\om}r|\phi|^2d\om\les E[\phi](\Si_{\tau}).
\]
For all $\ep_1>0$, we have the following inequality:
\[
 \tau_+^{\ga_0-1-\ep}r=(\ep_1^{-\ga_0}r^{1+\ga_0}\tau_+^{-1-\ep})^{\frac{1}{1+\ga_0}}(\ep_1\tau_+^{\ga_0-\ep})^{\frac{\ga_0}{1+\ga_0}}\leq
 \frac{\ep_1^{-\ga_0}r^{1+\ga_0}\tau_+^{-1-\ep}}{1+\ga_0}+\frac{\ga_0 \ep_1\tau_+^{\ga_0-\ep}}{1+\ga_0}.
\]
In particular the above inequality holds for $r=1$.
Moreover we also have
\[
 (\tau_1)_+^{\ga_0-\ep}r=(r^{1+\ga_0})^{\frac{1}{1+\ga_0}}\left((\tau_1)_+^{1+\ga_0-\frac{1+\ga_0}{\ga_0}\ep}\right)^{\frac{\ga_0}{1+\ga_0}}\leq r^{1+\ga_0}+(\tau_1)_+^{1+\ga_0-\ep}.l
\]
Therefore from estimate \eqref{eq:pWE:sca:in:1}, we can show that
\begin{align*}
 &\int_{\tau_1}^{\tau_2}\tau_+^{\ga_0-1-\ep}\left(\int_{H_{\tau^*}}r|D_L\psi|^2dvd\om+E[\phi](\Si_{\tau})+\tau_+^{-\ga_0}\mathcal{E}_0[\phi]+I_{-2-\ga_0}^{1+\ga_0}[r^{-1}D_L(r\phi)](\bar{\mathcal{D}}_{\tau}^{\tau_2})\right)d\tau\\
 &\les \ep_1^{-\ga_0}\int_{\tau_1}^{\tau_2}\tau_+^{-1-\ep}\int_{H_{\tau^*}}r^{1+\ga_0}|D_L\psi|^2dvd\om d\tau+\ep_1\int_{\tau_1}^{\tau_2}\tau_+^{\ga_0-\ep}E[\phi](\Si_\tau)d\tau\\
 &\quad+\ep_1^{-\ga_0}\int_{\tau_1}^{\tau_2}\tau_+^{-1-\ep}E[\phi](\Si_{\tau})d\tau+\mathcal{E}_0[\phi]+I_{-2-\ep}^{1+\ga_0}[r^{-1}D_L(r\phi)](\bar{\mathcal{D}}_{\tau_1}^{\tau_2}).
\end{align*}
On the right hand side of the above estimate, the first term can be grouped with the last term. The second term will be absorbed for small $\ep_1$. The third term can be bounded by using estimate \eqref{eq:ILE:sca:in:sim0}. Therefore
by using Lemma \ref{lem:simplint} and Proposition \ref{prop:pWE:sca:in:1}, we can show that
 \begin{align*}
   \int_{\tau_1}^{\tau_2}\tau_+^{\ga_0-\ep}E[\phi](\Si_\tau)d\tau\les_{M_2}& \ep_1\int_{\tau_1}^{\tau_2}\tau_+^{\ga_0-\ep}E[\phi](\Si_\tau)d\tau+
  \ep_1^{-\ga_0} \int_{H_{\tau_1^*}}r^{1+\ga_0}|D_L\psi|^2dvd\om+\ep_1^{-\ga_0}\mathcal{E}_0[\phi]\\
  &+\ep_1^{-\ga_0}(\tau_1)_+^{1+\ga_0-\ep}E[\phi](\Si_{\tau_1})+\ep_1^{-\ga_0}I_{-2-\ep}^{1+\ga_0}[r^{-1}D_L(r\phi)](\bar{\mathcal{D}}_{\tau_1}^{\tau_2}).
\end{align*}
Let $\ep_1$ be sufficiently small, depending on $M_2$, $\ep$, $\ga_0$ and $R$. We obtain estimate \eqref{eq:pWE:sca:in:1:Wga}.
\end{proof}
Estimate \eqref{eq:pWE:sca:in:1:Wga} will be used to derive the $r$-weighted energy estimate with $p=1+\ga_0$.
\begin{prop}
  \label{prop:pWE:sca:in:1ga}
 We have
 \begin{equation}
  \label{eq:pWE:sca:in:1ga}
  \begin{split}
 &\int_{H_{\tau_2^*}}r^{1+\ga_0}|D_L\psi|^2dvd\om+\int_{\tau_1}^{\tau_2}\int_{H_{\tau^*}}r^{\ga_0}(|D_L\psi|^2+|\D\psi|^2)dvd\om d\tau\\
&\les_{M_2} \int_{H_{\tau_1^*}}r^{1+\ga_0}|D_L\psi|^2dvd\om+\mathcal{E}_0[\phi]+(\tau_1)_+^{1+\ga_0-\ep}E[\phi](\Si_{\tau_1}).
\end{split}
 \end{equation}
\end{prop}
\begin{proof}
By taking $\ep_1$ in estimate \eqref{eq:Est4FJ:sca:in} to be sufficiently small and combining it with the $r$-weighted energy estimate \eqref{eq:pWE:sca:in} for $p=1+\ga_0$, from corollary \ref{cor:ILE:sca:in:sim0}, we obtain
  \begin{align*}
 &\int_{H_{\tau_2^*}}r^{1+\ga_0}|D_L\psi|^2dvd\om+\int_{\tau_1}^{\tau_2}\int_{H_{\tau^*}}r^{\ga_0}(|D_L\psi|^2+|\D\psi|^2)dvd\om d\tau\\
 &\les \int_{H_{\tau_1^*}}r^{1+\ga_0}|D_L\psi|^2dvd\om+\mathcal{E}_0[\phi]
 +M_2\left(\int_{\tau_1}^{\tau_2}E[\phi](\Si_{\tau})\tau_+^{\ga_0-\ep}d\tau+I^{1+\ga_0}_{-1-\ep}[r^{-1}D_L\psi](\bar {\mathcal{D}}_{\tau_1}^{\tau_2})\right).\\
&\qquad+C_{M_2}(E[\phi](\Si_{\tau_1})+(\tau_1)_+^{-1-\ga_0}\mathcal{E}_0[\phi]+I_{-2-\ga_0}^{1+\ga_0}[r^{-1}D_L\psi](\bar{\mathcal{D}}_{\tau_1}^{\tau_2}))
\end{align*}
for some constant $C_{M_2}$ depending on $M_2$. Estimate \eqref{eq:pWE:sca:in:1ga} then follows from estimate \eqref{eq:pWE:sca:in:1:Wga} together with Gronwall's inequality.
\end{proof}
If we take $\tau_1$ on the right hand side of \eqref{eq:pWE:sca:in:1ga} to be $0$, from the energy estimate \eqref{eq:ILEdecay:sca:ex} and the $r$-weighted energy estimate \eqref{eq:pWEdecay:sca:ex} in the exterior region, we conclude that the right hand side of \eqref{eq:pWE:sca:in:1ga} is bounded with $\tau_1=0$. In particular the $r$-weighted energy estimate in the interior region is bounded as we expected.
\begin{cor}
\label{cor:pWEdecay:sca:in:1ga}
Let $\psi=r\phi$. For all $0\leq \tau_1\leq \tau_2$, we have
 \begin{equation}
  \label{eq:pWEdecay:sca:in:1ga}
  \begin{split}
 \int_{H_{\tau_2^*}}r^{1+\ga_0}|D_L\psi|^2dvd\om+\int_{\tau_1}^{\tau_2}\int_{H_{\tau^*}}r^{\ga_0}(|D_L\psi|^2+|\D\psi|^2)dvd\om d\tau
 \les_{M_2} \mathcal{E}_0[\phi].
\end{split}
 \end{equation}
\end{cor}
\begin{proof}
 From the $r$-weighted energy decay estimate \eqref{eq:pWEdecay:sca:ex} for $p=1+\ga_0$, $\tau_1=0$ in the exterior region, let $u_2\rightarrow \infty$. We obtain
\[
\int_{H_{0^*}}r^{1+\ga_0}|D_L\psi|^2dvd\om=\int_{H_{-\frac{R}{2}}}r^{1+\ga_0}|D_L\psi|^2\les_{M_2}\mathcal{E}_0[\phi].
\]
From the energy estimate \eqref{eq:ILEdecay:sca:ex} in the exterior region, we conclude that
\[
E[\phi](\Si_{0})=E[\phi](\{t=0, r\leq R\})+E[\phi](H_{-\frac{R}{2}})\les \mathcal{E}_0[\phi].
\]
Then estimate \eqref{eq:pWEdecay:sca:in:1ga} follows from \eqref{eq:pWE:sca:in:1ga} by taking $\tau_1=0$ on the right hand side.
\end{proof}
This uniform bound for the $r$-weighted energy estimate in the interior region is crucial for the energy flux decay. It in particular implies that the terms involving the $r$-weighted energy flux on the right hand side of the energy estimate \eqref{eq:ILE:sca:in:sim0} and the integral of the energy flux estimate \eqref{eq:pWE:sca:in:1} have the right decay to show the energy flux decay.
\begin{prop}
\label{prop:Enerdecay:sca:in}
In the interior region, we have the energy flux decay:
\begin{equation}
\label{eq:Enerdecay:sca:in}
E[\phi](\Si_\tau)\les_{M_2}\mathcal{E}_0[\phi]\tau_+^{-1-\ga_0},\quad \forall \tau\geq 0.
\end{equation}
\end{prop}
\begin{proof}
Estimate \eqref{eq:pWEdecay:sca:in:1ga} implies that
\[
I_{-2-\ga_0}^{1+\ga_0}[r^{-1}D_L\psi](\bar{\mathcal{D}}_{\tau_1}^{\tau_2})\les_{M_2}(\tau_1)_+^{-1-\ga_0}\mathcal{E}_0[\phi],\quad \forall 0\leq \tau_1\leq \tau_2.
\]
Then using the pigeon hole argument similar to the proof of Proposition \ref{prop:Endecay:cur} for the energy flux decay of
the Maxwell field in the interior, the energy decay estimate \eqref{eq:Enerdecay:sca:in} for the scala field follows from the energy estimate \eqref{eq:ILE:sca:in:sim0}, the integral of the energy flux estimate \eqref{eq:pWE:sca:in:1} and the $r$-weighted energy estimate \eqref{eq:pWEdecay:sca:in:1ga}. For a detailed proof for this, we refer to Proposition 2 of \cite{yang2}.
\end{proof}

\subsubsection{Energy decay estimates for the first order derivative of the scalar field}
In this section, we derive the energy flux decay estimates for the derivative of the scalar field. The difficulty is that the covariant wave operator $\Box_A$ does not commute with $D_{Z}$. Commutators are quadratic in the Maxwell field and
the scalar field. In our setting, the Maxwell field is large. In particular those terms can not be absorbed. The idea is to exploit the null structure of the commutators and to use Gronwall's inequality adatped to our foliation $\Si_{\tau}$.

In the following, we always use $\psi$ to denote the weighted scalar field $r\phi$, that is, $\psi=r\phi$. The first order derivative of $\phi$ is shorted as $\phi_1$, similarly, for $\phi_2$.
More precisely, we denote $\phi_1=D_Z\phi$, $\phi_2=D_{Z}^2\phi$. Same notation for the weighted scalar field $\psi$, e.g., $\psi_1=r D_Z\phi$. For any function $f$, under the null coordinates $(u, v, \om)$, we denote
\[
\|f\|_{L_v^2 L_u^{\infty}L_{\om}^2(\mathcal{D})}^2:=\int_{v}\sup\limits_{u}\int_{\om}|f|^2d\om dv,
\]
where $(u, v,\om)$ are the null coordinates on the region $\mathcal{D}$. Similarly we have the notation of $\|f\|_{L_u^2 L_v^\infty L_{\om}^2(\mathcal{D})}$. We can also define $L_u^pL_v^q L_{\om}^r$ norms for general $p$, $q$, $r$.

To apply Corollary \ref{cor:ILEdecay:sca:ex} for the exterior region and Proposition \ref{prop:Enerdecay:sca:in} for the interior region, it suffices to control the commutator terms.
However, we are not able to bound the commutator terms directly by using the zero's order energy estimates. One has to make use of the energy flux of the first order derivative of the solution and then apply Gronwall's inequality.
However for the energy estimate for the first order derivative of the solution, the key is to understand the commutator $[\Box_A, D_Z]$ with $Z=\pa_t$ or the angular momentum. The cases of $\pa_t$ and the angular momentum are
quite different. The main reason is that the angular momentum contains weights in $r$ while $\pa_t$ does not. For the case when $Z=\pa_t$, it is easy to bound $[\Box_A, D_{\pa_t}]\phi$. The only place we need to be careful is the charge part. For the case of $Z=\Om$, the problem is that the commutator $[\Box_A, D_{\Om}]$ will produce a term of the form $Z^\nu F_{\mu\nu}D^{\mu}\phi$ which can not be written as a linear term of $D_Z\phi$. The estimate for the commutator terms heavily rely on the null structure. We first show the following lemma for the commutator terms.
\begin{lem}
\label{lem:Est4commu:1}
When $|x|\geq R$, we have
\begin{equation}
\label{eq:Est4commu:1}
|[\Box_A, D_Z]\phi|\les |\a||D_{\Lb}\psi|+(|\ab|+r^{-1}|\rho|)|D_{L}\psi|+|F||\D\phi|+(|J|+r|\J|+|\si|+r^{-1}|\rho|)|\phi|.
\end{equation}
When $r\leq R$, we have
\begin{equation}
\label{eq:Est4commu:1:in}
|[\Box_A, D_Z]\phi|\les |F||\bar{D} \phi|+|J||\phi| .
\end{equation}
\end{lem}
\begin{remark}
In this paper, all the quantity involving $Z$ should be interpreted as the sum of the quantity for all possible vector fields $Z$ unless we specify it.
\end{remark}

\begin{proof}
Let $\psi=r\phi$. First from Lemma \ref{lem:commutator}, we can write
\begin{equation}
\label{eq:com1:nullst}
[\Box_A, D_{Z}]\phi=2i r^{-1}Z^\nu F_{\mu\nu}D^{\mu}\psi+i \pa^\mu F_{\mu\nu}Z^\nu\phi+i\phi\left(-2 Z^\nu F_{\mu\nu}r^{-1}\pa^{\mu}r+\pa^{\mu}Z^\nu F_{\mu\nu}\right).
\end{equation}
We write the commutator terms as above is to exploit the null structure. The first term is the main term. Since we will rely on the $r$-weighted energy estimates, this suggests to write the main term of the commutator in terms of the weighted solution $r\phi$. The second term is easy as $\pa^\mu F_{\mu\nu}$ is nonlinear term of $\phi$ by the Maxwell equation. Let's first estimate the third term. When $Z=\Om$, note that $r^{-1}\Om$ is linear combination of $e_1$, $e_2$. We then can show that
\begin{align*}
|r^{-1}Z^\nu F_{\mu\nu}D^{\mu}(r\phi)|\les |\a||D_{\Lb}(r\phi)|+|\ab||D_{L}(r\phi)|.
\end{align*}
This is the null structure we need: the "bad" component $\ab$ of the curvature does not interact with the "bad" component $D_{\Lb}(r\phi)$ of the scalar field. Similarly, when $Z=\pa_t$, the "bad" term $r^{-1}\ab D_{\Lb}(r\phi)$ does not appear. More precisely, we have
\begin{align*}
|r^{-1}Z^\nu F_{\mu\nu}D^{\mu}(r\phi)|\les r^{-1}(|\a|+|\ab|)|\D(r\phi)|+r^{-1}|\rho||D_r(r\phi)|.
\end{align*}
For the second term on the right hand side of \eqref{eq:com1:nullst}, we note that $\pa^\mu F_{\mu\nu}$ is nonlinear term of $\phi$. We have
\begin{equation*}
|\pa^\mu F_{\mu\nu}Z^\nu\phi|\les (|J|+r|\J|)|\phi|.
\end{equation*}
For the third term on the right hand side of \eqref{eq:com1:nullst}, we show that
\begin{equation*}
|i\phi\left(-2 Z^\nu F_{\mu\nu}r^{-1}\pa^{\mu}r+\pa^{\mu} Z^\nu F_{\mu\nu}\right)|\les (|\si|+r^{-1}|\rho|)|\phi|.
\end{equation*}
The case when $Z=\pa_t$ is trivial. To check the above inequality for the case when $Z=\Om$, it suffices to prove it for the component $\Om_{jk}=x_j \pa_k-x_k\pa_j$. Then we can show that
\begin{align*}
-2\Om^\nu F_{\mu\nu}r^{-1}\pa^{\mu}r+\pa^{\mu}\Om^\nu F_{\mu\nu}&=2F_{jk}-2F(\pa_r, \Om_{ij})\\
&=2F(\om_{j}\pa_r+\pa_j-\om_j\pa_r, \om_k\pa_r+\pa_k-\om_k\pa_r)-2F(\pa_r, \Om_{jk})\\
&=2F(\pa_j-\om_j\pa_r, \pa_k-\om_k\pa_r).
\end{align*}
Here recall that $\om_j=r^{-1}x_j$. Since for all $j=1$, $2$, $3$, $\pa_j-\om_j\pa_r$ is orthogonal to $L$, $\Lb$, we conclude that $\pa_j-\om_j\pa_r$ is a linear combination of $e_1$ and $e_2$. The desired estimate then follows as the norm of the vector fields $\pa_j-\om_j\pa_r$ is less than 1.
\end{proof}

We begin a series of propositions in order to estimate the weighted spacetime norm of the commutators. The estimates in the bounded region $r\leq R$ are easy to obtain as the weights are finite.
We now concentrate on the region $r\geq R$. Let $\bar{\mathcal{D}}_{\tau}=\mathcal{D}_{\tau}\cap\{|x|\geq R\}$. We first consider
$|\a||D_{\Lb}(r\phi)|$.
\begin{prop}
\label{prop:supDLbphi}
For all $\ep_1>0$, we have
\begin{align}
\label{eq:supDLbphi}
\|D_{\Lb}(r\phi)\|_{L_u^2 L_v^\infty L_{\om}^2(\bar{\mathcal{D}}_{\tau})}& \les_{M_2}\mathcal{E}_0[\phi]\ep_1^{-1} \tau_+^{-1-\ga_0}+\ep_1 I^{1+\ep}_0[r^{-1}D_LD_{\Lb}\psi](\bar{\mathcal{D}}_{\tau}).
\end{align}
\end{prop}
\begin{proof}
The idea is to bound $\sup|D_{\Lb}(r\phi)|$ by the $L^2$ norm of $D_{L}D_{\Lb}(r\phi)$. In the exterior region when $\bar{\mathcal{D}}_{\tau}=\mathcal{D}_{\tau}^{-\infty}$, we can integrate from the initial hypersurface $t=0$.
In the interior region, choose the incoming null hypersurface $\Hb_{\frac{\tau_2+R}{2}}^{\tau_1, \tau_2}$ as the starting surface. Denote $\psi=r\phi$. We show estimate \eqref{eq:supDLbphi} for the interior region case, that is, when
$0\leq \tau_1<\tau_2$. On the outgoing null hypersurface $H_{\tau^*}$, for all $0\leq \tau_1\leq \tau\leq \tau_2$, we have
\begin{align*}
\sup\limits_{v\geq\frac{\tau+R}{2}}\int_{\om}|D_{\Lb}(r\phi)|^2(\tau^*, v, \om)d\om & \les \int_{\om}|D_{\Lb}(r\phi)|^2(\tau^*, \frac{\tau_2+R}{2}, \om)d\om+\int_{H_{\tau^*}}|D_L D_{\Lb}(r\phi)|\cdot |D_{\Lb}(r\phi)|dvd\om.
\end{align*}
Integrate the above estimate from $\tau_1$ to $\tau_2$ and apply Cauchy-Schwarz's inequality to the last term. From the integrated local energy estimate
 \eqref{eq:ILE:sca:in:sim0} and the energy decay estimate \eqref{eq:Enerdecay:sca:in}, we then derive
\begin{align*}
\int_{\tau_1}^{\tau_2}\sup\limits_{v\geq \frac{R+\tau}{2}}\int_{\om}|D_{\Lb}(r\phi)|^2 d\om  d\tau\les_{M_2}\mathcal{E}_0[\phi]\ep_1^{-1} (\tau_1)_+^{-1-\ga_0}+\ep_1 I^{1+\ep}_0[r^{-1}D_LD_{\Lb}\psi](\bar{\mathcal{D}}_{\tau_1}^{\tau_2})
\end{align*}
for all $\ep_1>0$. The case in the exterior region follows in a similar way.
\end{proof}
We also need the analogous estimates for $D_{L}(r\phi)$.
\begin{prop}
\label{prop:supDLphi}
For all $\ep_1>0$ and $0\leq p\leq 1+\ga_0$, we have
\begin{align}
\label{eq:supDlphi}
\|r^{\frac{p}{2}}D_{L}(r\phi)\|_{L_v^2 L_u^\infty L_{\om}^2(\bar{\mathcal{D}}_{\tau})}^2 &\les_{M_2}\ep_1^{-1} \mathcal{E}_0[\phi](\tau)_+^{-1-\ga_0}+\ep_1 I^{p_1}_{p_2}[r^{-1}D_{\Lb}D_{L}\psi](\bar{\mathcal{D}}_{\tau}).
\end{align}
Here $p_1=\max\{1+\ep, p\}$, $p_2=\min\{1+\f12\ep, p\}$.
\end{prop}
\begin{proof}
Similar to the proof of the previous proposition, we choose the starting surface for $D_{L}(r\phi)$ to be $H_{\tau_1^*}$ in the interior region and the initial hypersurface $t=0$ in the exterior region.
We only prove the proposition for the exterior region case. On $\Hb_{v}^{\tau^*}$, $v\geq -\tau^*$, we can show that
\begin{align*}
r^p\int_{\om}|D_{L}\psi|^2d\om\les \int_{\om}(r^p|D_L\psi|^2)(-v, v, \om)d\om+\int_{\Hb_v^{\tau^*}}r^{p-1}|D_L\psi|^2+r^p|D_L\psi||D_{\Lb}D_L\psi|dud\om.
\end{align*}
The integral of the first term can bounded by the assumption on the data. We control the second term by using the $r$-weighted energy estimate. We bound the last term as follows:
\begin{align*}
r^p|D_L\psi||D_{\Lb}D_L\psi|\les \ep_1 r^{p_1} u_+^{p_2}|D_{\Lb}D_L\psi|^2+\ep_1^{-1}r^{2p-p_1}u_+^{-p_2}|D_L\psi|^2,\quad \forall \ep_1>0.
\end{align*}
When $2p\geq p_1$, we can use the $r$-weighted energy estimate \eqref{eq:pWEdecay:sca:ex} to bound the weighted integral of $|D_L\psi|$. Otherwise one can use interpolation and the integrated local energy decay
estimate \eqref{eq:ILEdecay:sca:ex}. For any case, from the energy decay estimates \eqref{eq:pWEdecay:sca:ex}, \eqref{eq:ILEdecay:sca:ex}, \eqref{eq:Enerdecay:sca:in}, \eqref{eq:pWEdecay:sca:in:1ga} for $\phi$, one can always show that
\begin{align*}
\iint_{\mathcal{D}_{\tau}}r^{2p-p_1}u_+^{-p_2}|D_L\psi|^2 dudvd\om\les_{M_2}\mathcal{E}_0[\phi]\tau_+^{p-1-\ga_0}.
\end{align*}
Another way to understand the above estimate is to use interpolation. It suffices to show the above estimate with $p=0$ and $p=1+\ga_0$. The former case follows by using the integrated local energy estimates for $\phi$ while the later situation relies on the $r$-weighted energy estimate.
Estimate \eqref{eq:supDlphi} for the exterior region case then follows. The interior region case holds in a similar way.
\end{proof}

As we only commute the equation with $\pa_t$ or the angular momentum $\Om$, to estimate the weighted spacetime integral of $D_LD_{\Lb}(r\phi)$ in terms of $D_Z\phi$, we use the equation of $\phi$ under the null frame.
\begin{lem}
\label{lem:EQ4sca:null}
Under the null frame, we can write the covariant wave operator $\Box_A$ as follows:
\begin{equation}
\label{eq:EQ4sca:null}
r\Box_A \phi=r D^\mu D_\mu\phi=-D_L D_{\Lb}(r\phi)+\D^2(r\phi)-i\rho \cdot r\phi=-D_{\Lb} D_{L}(r\phi)+\D^2(r\phi)+i\rho \cdot r\phi
\end{equation}
for any complex scalar field $\phi$. Here $\D^2=\D^{e_1}\D_{e_1}+\D^{e_2}\D_{e_2}$ and $\rho=\f12 (dA)_{\Lb L}$.
\end{lem}
\begin{proof}
The lemma follows by direct computation.
\end{proof}
This lemma leads to the following estimates for $D_L D_{\Lb}(r\phi)$ and $D_{\Lb}D_{L}(r\phi)$.
\begin{prop}
\label{prop:Est4DLDLbphi}
For all $1+\ep\leq p\leq 1+\ga_0$, we have
\begin{align}
\label{eq:Est4DLDLbphi}
&I^{p}_{2+\ga_0-p-2\ep}[r^{-1}(|D_{L}D_{\Lb}\psi|+|D_{L}D_{\Lb}\psi|)](\bar{\mathcal{D}}_{\tau})\les_{M_2}\mathcal{E}_0[\phi]+I^{-1-\ep}_{1+\ga_0-\ep}[D\phi_1](\bar{\mathcal{D}}_{\tau})
+I^{\ga_0}_{0}[\D\psi_1](\bar{\mathcal{D}}_{\tau}).
\end{align}
Here $\phi_1=D_Z\phi$ and $\psi_1=D_Z(r\phi)$.
\end{prop}
\begin{proof}
Let's only consider estimate for $D_LD_{\Lb}(r\phi)$ in the interior region. The proof easily implies the estimates for $D_{\Lb}D_L(r\phi)$. The case in the exterior region is easier as in that region $r\geq\frac{1}{3}u_+$ and
it suffices to show the estimate for $p=1+\ga_0$ which is similar to the proof for the interior region case. Take $\bar{\mathcal{D}}_{\tau}$ to be $\bar{\mathcal{D}}_{\tau_1}^{\tau_2}$ for $0\leq \tau_1=\tau<\tau_2$. From the equation \eqref{eq:EQ4sca:null} for $\phi$ under the null frame, we derive
\[
r^p|D_L D_{\Lb}(r\phi)|^2 \les r^p|\Box_A \phi|^2 r^2+r^p|r\phi \rho|^2+r^p|r^{-1}\D D_{\Om}\psi|^2.
\]
Here we note that $|\D^2\psi|^2\les |r^{-1}\D D_{\Om}\psi|$. The integral of the first term on the right hand side can be bounded by $\mathcal{E}_0[\phi]$. For the second term, we control $\phi$ by using Lemma \ref{lem:Est4phipWE}.
The last term is favorable as it is a form of $\D D_Z\psi$. We will absorb those terms with the help of the small constant $\ep_1$ from Propositions \ref{prop:supDLbphi}, \ref{prop:supDLphi}. According to our notation in this section, let $\psi_1=D_{\Om}\psi$. For all $1+\ep\leq p\leq 1+\ga_0$, we have
\[
\tau_+^{2+\ga_0-p-2\ep}r^{p-2}\les r^{\ga_0}+\tau_+^{1+\ga_0-\ep}r^{-1-\ep}, \quad r\geq R.
\]
Since the energy flux for $\phi$ decays from Proposition \ref{prop:Enerdecay:sca:in}, by using Lemma \ref{lem:Est4phipWE}, we conclude that
\[
\int_{\om}r^p|\phi|^2d\om\les_{M_2}\mathcal{E}_0[\phi](\tau_1)_+^{p-2-\ga_0}.
\]
Therefore for all $1+\ep\leq p\leq 1+\ga_0$ we can show that
\begin{align*}
&\iint_{\bar{\mathcal{D}}_{\tau_1}^{\tau_2}}\tau_+^{2+\ga_0-p-2\ep}r^{p}|D_L D_{\Lb}(r\phi)|^2 dvdud\om\\
&\les I^{p}_{2+\ga_0-p- 2\ep}[\Box_A\phi](\bar{\mathcal{D}}_{\tau_1}^{\tau_2})+I^{-1-\ep}_{1+\ga_0-\ep}[D\phi_1](\bar{\mathcal{D}}_{\tau_1}^{\tau_2})
+I^{\ga_0}_{0}[\D\psi_1](\bar{\mathcal{D}}_{\tau_1}^{\tau_2})\\
&\qquad+\int_{\tau_1}^{\tau_2}\tau_+^{2+\ga_0-p-2\ep}\int_{\frac{\tau+R}{2}}^{\infty} \int_{\om}r^{p}|\phi|^2d\om\cdot \sum\limits_{j\leq 2}\int_{\om}r^2|\mathcal{L}_{\Om}^2 \bar \rho|^2d\om dvdu\\
&\les_{M_2}\mathcal{E}_0[\phi]+I^{-1-\ep}_{1+\ga_0-\ep}[D\phi_1](\bar{\mathcal{D}}_{\tau_1}^{\tau_2})
+I^{\ga_0}_{0}[\D\psi_1](\bar{\mathcal{D}}_{\tau_1}^{\tau_2})
\end{align*}
This finished the proof.
\end{proof}

Next we estimate the weighted spacetime norm of $|\a||D_{\Lb}(r\phi)|$.
\begin{prop}
\label{prop:com1:aDLbphi}
For all $1+\ep\leq p\leq 1+\ga_0$, $\ep_1>0$, we have
\begin{align}
\label{eq:com1:aDLbphi}
\iint_{\bar{\mathcal{D}}_{\tau}}u_+^{2+\ga_0+\ep-p}r^p|\a|^2 |D_{\Lb}(r\phi)|^2 dxdt &\les_{M_2}\mathcal{E}_0[\phi] \ep_1^{-1} \tau_+^{-\ga_0+\ep}+\ep_1 I^{1+\ep}_{1+\ep}[r^{-1}D_LD_{\Lb}\psi](\bar{\mathcal{D}}_{\tau} ).
\end{align}
\end{prop}
\begin{proof}
Make use of Proposition \ref{prop:supDLbphi}. For all $1+\ep\leq p\leq 1+\ga_0$, we can show that
\begin{align*}
\iint_{\bar{\mathcal{D}}_{\tau}} r^p|\a|^2 |D_{\Lb}(r\phi)|^2 dxdt &\les \|D_{\Lb}\psi\|_{L_u^2 L_v^\infty L_{\om}^2(\bar{\mathcal{D}}_{\tau})}^2\cdot
\|r^{\frac{p}{2}+1}\a\|_{L_u^{\infty}L_v^2 L_{\om}^{\infty}(\bar{\mathcal{D}}_{\tau})}^2\\
&\les \|D_{\Lb}\psi\|_{L_u^2 L_v^\infty L_{\om}^2(\bar{\mathcal{D}}_{\tau})}^2\cdot
\sum\limits_{j\leq 2}\|r^{\frac{p}{2}+1}\mathcal{L}_{\Om}^j\a\|_{L_u^{\infty}L_v^2 L_{\om}^{2}(\bar{\mathcal{D}}_{\tau})}^2\\
&\les_{M_2}\mathcal{E}_0[\phi] \ep_1^{-1} \tau_+^{p-2-2\ga_0}+\ep_1 \tau_+^{p-1-\ga_0}I^{1+\ep}_0[r^{-1}D_LD_{\Lb}\psi](\bar{\mathcal{D}}_{\tau} )
\end{align*}
for all $\ep_1>0$. As the above estimate holds for all $\tau\in\mathbb{R}$, from Lemma \ref{lem:simplint}, we conclude that
\begin{align*}
\iint_{\bar{\mathcal{D}}_{\tau}}  u_+^{2+\ga_0+\ep-p}r^p|\a|^2 |D_{\Lb}(r\phi)|^2 dxdt
&\les_{M_2}\mathcal{E}_0[\phi] \ep_1^{-1} \tau_+^{-\ga_0+\ep}+\ep_1 I^{1+\ep}_{1+\ep}[r^{-1}D_LD_{\Lb}\psi](\bar{\mathcal{D}}_{\tau} ).
\end{align*}
This finished the proof for estimate \eqref{eq:com1:aDLbphi}.
\end{proof}
Next we estimate the weighted spacetime integral of $(|\ab|+r^{-1}|\rho|)|D_L(r\phi)|$. One possible way to bound this term, in particular $\ab$, is to make use of the energy flux through the incoming null hypersurface.
It turns out that we loss a little bit of decay in $u$ and we are not able to close the bootstrap argument later. An alternative way is to use $\sup\limits_{v}\int_{\om}|\ab|^2d\om$ which has to exploit the equation for $F$.
For $\tau\in\mathbb{R}$, let
\begin{equation}
\label{eq:defofgu}
g(\tau)=\sum\limits_{k\leq 1}\|\mathcal{L}_{\Om}^k(r\ab)\|_{L_v^\infty L_{\om}^2(H_{\tau^*})}^2+\sum\limits_{k\leq 2}\int_{\Si_{\tau}}\frac{|\mathcal{L}_{Z}^k \bar{F}|^2}{r_+^{1+\ep}}r^{2}d\tilde{v}d\om.
\end{equation}
Here $(\tilde{v}, \om)$ are coordinates of $\Si_{\tau}$. This notation should not be confused with the one in Proposition \ref{prop:Est4FJ:sca:in:0} in the previous section and this function will only be used in this section.
We can not show that $g(\tau)$ decays in $\tau$. However as a corollary we can show that
\begin{cor}
\label{cor:L2gu}
The function $g(\tau)$ is integrable in $\tau$:
\begin{equation}
\label{eq:L2gu}
\int_{\tau_1}^{\tau_2}\tau_+^{1+\ep}g(\tau)d\tau\les M_2\tau_+^{-\ga_0+3\ep},\quad \int_{\tilde{\tau}\leq \tau}\tilde{\tau}_+^{1+\ep}g(\tilde{\tau})d\tilde{\tau}\les M_2\tau_+^{-\ga_0+3\ep}
\end{equation}
for all $0\leq \tau_1\leq \tau_2$, $\tau\leq 0$.
\end{cor}
\begin{proof}
By using Lemma \ref{lem:simplint}, the corollary follows from estimate \eqref{eq:supab:I} and the integrated local energy estimates \eqref{eq:ILE:cur:ex}, \eqref{eq:ILE:cur:in} for the curvature.
\end{proof}
We now can estimate the weighted spacetime integral of $|\ab||D_{L}\psi|$.
\begin{prop}
\label{prop:com1:abDLphi}
For all $1+\ep\leq p\leq 1+\ga_0$, $\ep_1>0$, we have
\begin{align}
\notag
&\iint_{\bar{\mathcal{D}}_{\tau}}u_+^{2+\ga_0+\ep-p}|\ab|^2 |D_{L}(r\phi)|^2 r^{p} dxdt \\
\label{eq:com1:abDLphi}
&\les_{M_2}\ep_1 \iint_{\bar{\mathcal{D}}_{\tau}} \tilde{\tau}_+^{2+\ga_0+\ep-p}g(\tilde{\tau}) r^{p}|D_L\psi_1|^2 dvd\om d\tilde{\tau}+\ep_1^{-1}\mathcal{E}_0[\phi]\tau_+^{-\ga_0+3\ep}.
\end{align}
\end{prop}
\begin{proof}
We first use Sobolev embedding on the unit sphere to bound that
\[
\||r\ab||D_L\psi|\|_{L_{\om}^2}^2\les (\|r\ab\|_{L_{\om}^2}^2+\|r\mathcal{L}_{\Om}\ab\|_{L_{\om}^2}^2)\cdot(\ep_{1}^{-1}\|D_L\psi\|_{L_{\om}^2}^2+\ep_1\|D_{\Om}D_{L}\psi\|_{L_{\om}^2}^2),\quad \ep_1>0
\]
The proof for this estimate for all connection $A$ follows from the case when $A$ is trivial as the norm is gauge invariant. We in particular can choose a gauge so that the function is real.
Then make use of estimate \eqref{eq:supab:I} of Proposition \ref{prop:supF}. We therefore can show that
\begin{align*}
&\|r^{\frac{p}{2}}u_+^{\frac{1}{2}(2+\ga_0+\ep-p)}|r\ab||D_L\psi|\|_{L_u^2L_v^2L_{\om}^2(\bar{\mathcal{D}}_\tau)}^2\\
&\les\|r^{\frac{p}{2}}\sum\limits_{k\leq 1}\|\mathcal{L}_{\Om}^k(r\ab)\|_{L_\om^2}\cdot u_+^{\frac{1}{2}(2+\ga_0+\ep-p)}(\ep_1^{\f12}\|D_{\Om}D_L\psi\|_{L_{\om}^2}+\ep_1^{-\f12}\|D_L\psi\|_{L_{\om}^2})\|_{L_u^2L_v^2}^2\\
&\les \| \tilde{\tau}_+^{2+\ga_0+\ep-p}g(\tilde{\tau})(\ep_1\|r^{\frac{p}{2}}D_LD_{\Om}\psi\|_{L_v^2L_{\om}^2}^2+\ep_1\|r^{\frac{p}{2}}r\a \psi\|_{L_v^2L_{\om}^2}^2+\ep_1^{-1}\|r^{\frac{p}{2}}D_L\psi\|_{L_v^2L_{\om}^2}^2)\|_{L_u^1}\\
&\les \ep_1\int_{\tilde{\tau}} \tilde{\tau}_+^{2+\ga_0+\ep-p}g(\tilde{\tau})\int_{H_{\tilde{\tau}^*}}r^{p}|D_L\psi_1|^2 dvd\om d\tilde{\tau}+\ep_1^{-1}\|\tilde{u}_+^{1+\ep}g(\tilde{u})\|_{L_u^1}\|u_+^{\frac{1+\ga_0-p}{2}}r^{\frac{p}{2}}D_L\psi\|_{L_u^\infty L_v^2 L_\om^{2}}^2 \\
&\qquad+\ep_1\|\tilde{u}_+^{1+\ep-\ga_0}g(\tilde{u})\|_{L_u^1}\|u_+^{\frac{1+\ga_0-p}{2}}r^{\frac{p}{2}}r\a\|_{L_u^\infty L_v^2 L_\om^{\infty}}^2\|u_+^{\frac{\ga_0}{2}}\psi\|_{L_u^\infty L_v^\infty L_{\om}^2}^2\\
&\les_{M_2}\ep_1 \iint_{\bar{\mathcal{D}}_{\tau}} \tilde{\tau}_+^{2+\ga_0+\ep-p}g(\tilde{\tau}) r^{p}|D_L\psi_1|^2 dvd\om du+\ep_1^{-1}\mathcal{E}_0[\phi]\tau_+^{-\ga_0+3\ep}.
\end{align*}
Here we have used the $r$-weighted energy estimates \eqref{eq:pWE:sca:in}, \eqref{eq:pWE:sca:ex} and estimate \eqref{eq:Est4phipWE:in:p} to bound $\phi$.
\end{proof}
For $|r^{-1}\rho||D_{L}\psi|$, we have extra decay in $r$ which allows us to use Proposition \ref{prop:supDLphi}.
\begin{prop}
\label{prop:com1:rhophi}
For all $\ep_1>0$, we have
\begin{align}
\label{eq:com1:rhophi}
\iint_{\bar{\mathcal{D}}_{\tau}}u_+^{1+\ga_0}|r^{-1}\rho|^2 |D_{L}(r\phi)|^2 r^{1+\ga_0} dxdt \les_{M_2}\ep_1 I^{1+\ep}_{\ep}[r^{-1}D_{\Lb}D_L\psi](\mathcal{D}_{\tau_1}^{\tau_2})+\ep_1^{-1}\mathcal{E}_0[\phi].
\end{align}
\end{prop}
\begin{proof}
The idea is that we bound $\rho$ by using the energy flux through the incoming null hypersurface and $D_{L}\psi$ by using Proposition \ref{prop:supDLphi}. In the exterior region, we need to specially consider
the effect of the nonzero charge. Other than that, the proof is the same for the interior region case. We thus take $\bar{\mathcal{D}}_{\tau}$ to be $\mathcal{D}_{\tau}$ with $\tau\leq 0$. For all $1+\ep\leq p\leq 1+\ga_0$, we then can show that
\begin{align*}
\iint_{\mathcal{D}_{\tau}}|r^{-1}\rho|^2 |D_{L}\psi|^2 r^{p} dxdt&\les \iint_{\mathcal{D}_{\tau}}|\bar{\rho}|^2 |D_{L}\psi|^2 r^{p} dudvd\om +\iint_{\mathcal{D}_{\tau}}|q_0|^2|D_{L}\psi|^2 r^{p-4} dudvd\om\\
&\les_{M_2} \|D_L\psi\|_{L_v^2 L_u^\infty L_{\om}^2(\mathcal{D}_{\tau})}^2\|r\bar\rho\|_{L_v ^\infty L_u^2 L_{\om}^\infty(\mathcal{D}_{\tau})}^2+\mathcal{E}_0[\phi]\tau_+^{-1-2\ga_0}\\
&\les_{M_2} \ep_1\tau_+^{-1-\ga_0} I^{1+\ep}_0[r^{-1}D_{\Lb}D_L\psi](\mathcal{D}_{\tau})+\ep_1^{-1}\mathcal{E}_0[\phi]\tau_+^{-1-2\ga_0}.
\end{align*}
The above estimate also holds for the interior region case when $\bar{\mathcal{D}}_{\tau}=\mathcal{D}_{\tau_1}^{\tau_2}$ for all $0\leq \tau_1<\tau_2$. As the estimates holds for all $\tau$, from Lemma \ref{lem:simplint}, we then can show that (take the interior region for example)
\begin{align*}
&\iint_{\mathcal{D}_{\tau_1}^{\tau_2}}\tau_+^{1+\ga_0}|r^{-1}\rho|^2 |D_{L}\psi|^2 r^{p} dxdt\\
&\les_{M_2} \ep_1 I^{1+\ep}_0[r^{-1}D_{\Lb}D_L\psi](\mathcal{D}_{\tau_1}^{\tau_2})+\ep_1\int_{\tau_1}^{\tau_2}\tau_+^{-1}I^{1+\ep}_0[r^{-1}D_{\Lb}D_L\psi]
(\mathcal{D}_{\tau}^{\tau_2})d\tau+\ep_1^{-1}\mathcal{E}_0[\phi]\\
&\les_{M_2}\ep_1 I^{1+\ep}_{\ep}[r^{-1}D_{\Lb}D_L\psi](\mathcal{D}_{\tau_1}^{\tau_2})+\ep_1^{-1}\mathcal{E}_0[\phi].
\end{align*}
Here we note that $\ln\tau_+\les \tau_+^{\ep}$.
\end{proof}
Next we estimate $r^{-1}|F||\D(r\phi)|$.
\begin{prop}
\label{prop:com1:FDZphi}
For all $\ep_1>0$, we have
\begin{align}
\label{eq:com1:FDZphi}
\iint_{\bar{\mathcal{D}}_{\tau}}u_+^{1+\ga_0}|r^{-1}F|^2|\D(r\phi)|^2r^{1+\ga_0} dxdt &\les_{M_2}\ep_1^{-1}\mathcal{E}_0[\phi]+\ep_1\int_{\tilde{\tau}}\tilde{\tau}_+^{1+\ga_0}g(\tilde{\tau})E[D_Z\phi](H_{\tilde{\tau}^*})d\tilde{\tau}.
\end{align}
 \end{prop}
\begin{proof}
The idea is that we use the energy flux through the outgoing null hypersurface to bound $\D(r\phi)=D_{\Om}\phi$ and the integrated local energy to control $F$. We only show the estimate in the exterior region.
Take $\bar{\mathcal{D}_{\tau}}$ to be $\mathcal{D}_{\tau}$ for any $\tau\leq 0$. In the exterior region we have the relation $r\geq \frac{1}{3}u_+$. Therefore from estimate \eqref{eq:Est4phiE:small} and recalling the
definition \eqref{eq:defofgu} of $g(u)$, we can show that
\begin{align*}
&\iint_{\mathcal{D}_{\tau}}u_+^{1+\ga_0}|F|^2|\D(r\phi)|^2r^{1+\ga_0} dudvd\om\\
 &\les \int_{u}u_+^{1+\ga_0}\int_{v}\sum\limits_{k\leq 2}r^{1-\ep}\int_{\om}|\mathcal{L}_{\Om}^k\bar{F}|^2+|q_0r^{-2}|^2 d\om \cdot \int_{\om}r|D_{\Om}\phi|^2d\om dudv\\
 &\les |q_0|^2 \iint_{\mathcal{D}_{\tau}}\frac{|\D\phi|^2}{r^{1+\ep}}dxdt+\int_{\tilde{\tau}}(\tilde{\tau})_+^{1+\ga_0}g(\tilde{\tau})(\ep_1^{-1}\int_{H_{\tilde{\tau}^*}}|\D\phi|^2 r^2dvd\om+\ep_1E[D_{Z}\phi](H_{\tilde{\tau}^*}))d\tilde{\tau}\\
 &\les_{M_2}\mathcal{E}_0[\phi]\tau_+^{-1-\ga_0}+\int_{\tilde{\tau}}(\tilde{\tau})_+^{1+\ga_0}g(\tilde{\tau})\ep_1^{-1}E[\phi](H_{\tilde{\tau}^*})
 d\tilde{\tau}+\ep_1\int_{u}g(u)E[D_Z\phi](H_u)du\\
 &\les_{M_2}\ep_1^{-1}\mathcal{E}_0[\phi]\tau_+^{-1-\ga_0+2\ep}+\ep_1\int_{\tilde{\tau}}\tilde{\tau}_+^{1+\ga_0}g(\tilde{\tau})E[D_Z\phi](H_{\tilde{\tau}^*})d\tilde{\tau}.
\end{align*}
Here we assumed that $\ga_0<1$ and $\ep$ is sufficiently small. For the case $\ga_0=1$, the above estimate also holds but in a different form where we have to rely on the $r$-weighted energy estimate. For the sake of simplicity, we would not discuss in details when $\ga_0\geq 1$.
\end{proof}
Finally we estimate the weighted spacetime norm of $(|J|+|r\J|+|\si|+|r^{-1}\rho|)|\phi|$. We show that
\begin{prop}
\label{prop:com1:Jphi}
For all $1+\ep\leq p\leq 1+\ga_0$, we have
\begin{align}
\label{eq:com1:Jphi}
\iint_{\bar{\mathcal{D}}_{\tau}}(|J|^2+|r\J|^2+|\si|^2+|r^{-1}\rho|^2)|\phi|^2r^{p}u_+^{2+\ga_0+\ep-p} dxdt &\les_{M_2}\mathcal{E}_0[\phi]\tau_+^{-\ga_0+\ep}.
\end{align}
\end{prop}
\begin{proof}
Let's first consider $(|J|^2+|\si|^2+|r^{-1}\rho|^2)|\phi|^2$. The idea is that we bound $\phi$ by the energy flux. Note that the nonzero charge only affects the estimates
in the exterior region where $r\geq \frac{1}{3}u_+$. From the embedding \eqref{eq:Est4phipWE:in:p} and the energy decay estimates \eqref{eq:ILEdecay:sca:ex}, \eqref{eq:Enerdecay:sca:in}, we can show that
\begin{align*}
&\iint_{\bar{\mathcal{D}}_{\tau}}(|J|^2+|\si|^2+|r^{-1}\rho|^2)|\phi|^2r^{p+2}u_+^{2+\ga_0+\ep-p} dudvd\om\\
&\les \int_{u}u_+^{2+\ga_0+\ep-p}\int_{v}\sum\limits_{k\leq 2}r^{p+1}\int_{\om}|\mathcal{L}_{\Om}^k J|^2+|\mathcal{L}_{\Om}^k\si|^2+
|r^{-1}\mathcal{L}_{\Om}^k\bar \rho|^2+|q_0r^{-3}|^2 d\om \cdot \int_{\om}r|\phi|^2d\om dvdu\\
&\les_{M_2} \mathcal{E}_0[\phi]\int_{u}u_+^{1+\ep-p}\int_{v}\sum\limits_{k\leq 2}r^{p+1}\int_{\om}|\mathcal{L}_{\Om}^k J|^2+|\mathcal{L}_{\Om}^k\si|^2+
|r^{-1}\mathcal{L}_{\Om}^k\bar \rho|^2+|q_0r^{-3}|^2 d\om dvdu\\
&\les_{M_2}\tau_+^{-\ga_0+\ep}.
\end{align*}
Here we used the $r$-weighted energy estimates \eqref{eq:pWE:cur:ex}, \eqref{eq:pWE:cur:in} to bound the curvature components and the definition for $M_2$ to control $J$.
For $|r\J|^2|\phi|^2$, the only difference is that we need to put more $r$ weights on $\phi$. By using the embedding inequality \eqref{eq:Est4phipWE:in:p} and the energy decay estimates \eqref{eq:Enerdecay:sca:in}, \eqref{eq:ILEdecay:sca:ex},
we conclude that
\[
  \int_{\om}r^{1+p-\ga_0}|\phi|^2d\om \les_{M_2}\tau_+^{p-1-2\ga_0}.
\]
Therefore we have
\begin{align*}
\iint_{\bar{\mathcal{D}}_{\tau}}u_+^{2+\ga_0+\ep-p}|r\J|^2|\phi|^2r^{p+2}dudvd\om&\les \int_{u}u_+^{2+\ga_0+\ep-p}\int_{v}\sum\limits_{k\leq 2}r^{3+\ga_0}\int_{\om}
|\mathcal{L}_{\Om}^k \J|^2d\om\cdot \int_{\om}r^{1+p-\ga_0}|\phi|^2d\om dudv\\
&\les_{M_2} \mathcal{E}_{0}[\phi]\iint_{\bar{\mathcal{D}}_{\tau}}u_+^{1-\ga_0+\ep}r^{3+\ga_0}\int_{\om}
|\mathcal{L}_{\Om}^k \J|^2d\om dvdu\\
&\les_{M_2}\mathcal{E}_{0}[\phi]\tau_+^{-\ga_0}.
\end{align*}
Estimate \eqref{eq:com1:Jphi} then follows.
\end{proof}
Now we are remained to consider the spacetime norm on the bounded region $r\leq R$.
\begin{prop}
 \label{prop:com1:R}
 On the bounded region $r\leq R$, for all $0\leq \tau_1<\tau_2$ we have
 \begin{equation}
  \label{eq:com1:R}
  \int_{\tau_1}^{\tau_2}\tau_+^{1+\ga_0}\int_{r\leq R}|[\Box_A, D_Z]\phi|^2dxdt\les_{M_2}\mathcal{E}_0[\phi](\tau_1)_+^{-1-\ga_0}.
 \end{equation}
\end{prop}
\begin{proof}
First we conclude from the energy estimate \eqref{eq:Enerdecay:sca:in} that the energy flux of the scalar field decays:
\[
 E[\phi](\Si_{\tau})\les_{M_2}\tau_+^{-1-\ga_0},\quad \forall \tau\geq 0.
\]
From the commutator estimates \eqref{eq:Est4commu:1:in}, we have
\[
 |[\Box_A, D_Z]\phi|\les |F||\bar D\phi|+|J||\phi|.
\]
For the first term, we make use of estimate \eqref{eq:Est4F:in:R}:
\begin{align*}
  \int_{\tau_1}^{\tau_2}\tau_+^{1+\ga_0}\int_{r\leq R}|F|^2|\bar D\phi|dxdt &\les \int_{\tau_1}^{\tau_2}\sup\limits_{|x|\leq R}|F|^2(\tau, x) E[\phi](\Si_{\tau})\tau_+^{1+\ga_0}d\tau\\
  &\les_{M_2}\mathcal{E}_0[\phi]\int_{\tau_1}^{\tau_2}\sup\limits_{|x|\leq R}|F|^2(\tau, x)d\tau\\
  &\les_{M_2}\mathcal{E}_0[\phi](\tau_1)_+^{-1-\ga_0}.
 \end{align*}
For $|J||\phi|$, we use Sobolev embedding on the ball $B_R$ with radius $R$ at fixed time $\tau$:
 \begin{align*}
  \int_{\tau_1}^{\tau_2}\tau_+^{1+\ga_0}\int_{r\leq R}|J|^2|\phi|^2dxdt &\les \int_{\tau_1}^{\tau_2}\tau_+^{1+\ga_0}\|J\|_{H^1_x(B_R)}^2\cdot \|\phi\|_{H^1_x(B_R)}^2d\tau\\
  &\les_{M_2}\int_{\tau_1}^{\tau_2}\mathcal{E}_0[\phi]\int_{r\leq R}|\nabla J|^2+|J|^2 dx d\tau\\
  &\les_{M_2}\mathcal{E}_0[\phi](\tau_1)_+^{-1-\ga_0}.
 \end{align*}
 Thus estimate \eqref{eq:com1:R} holds.
\end{proof}
Now from Lemma \ref{lem:Est4commu:1}, combine estimates \eqref{eq:com1:aDLbphi}, \eqref{eq:com1:abDLphi}-\eqref{eq:com1:R}. We can bound the first order commutator.
\begin{cor}
\label{cor:Est4:com1}
For all positive constant $\ep_1<1$, we have
\begin{equation}
\label{eq:Est4:com1}
\begin{split}
&I^{1+\ga_0}_{1+\ep}[[\Box_A, D_Z]\phi](\{t\geq 0\})+ I^{1+\ga_0}_{1+\ep}[[\Box_A, D_Z]\phi](\{t\geq 0\}\\
&\les_{M_2}\ep_1 I^{-1-\ep}_{1+\ga_0-\ep}[D\phi_1](\{t\geq 0\})+\ep_1 I^{\ga_0}_{0}[\D\psi_1](\{t\geq 0\}\cap\{r\geq R\})+\mathcal{E}_0[\phi] \ep_1^{-1}\\
&+\ep_1\int_{\mathbb{R}}\tau_+^{1+\ga_0}g(\tau)E[D_Z\phi](\Si_{\tau})d\tau+\ep_1 \int_{\mathbb{R}}\int_{H_{\tau^*}} \tau_+^{2+\ga_0+\ep-p}g(\tau) r^{p}|D_L\psi_1|^2 dvd\om d\tau
\end{split}
\end{equation}

\end{cor}
\begin{proof}
From Lemma \ref{lem:Est4commu:1}, estimate \eqref{eq:Est4:com1} is a consequence of estimates \eqref{eq:com1:aDLbphi}, \eqref{eq:com1:abDLphi}, \eqref{eq:com1:FDZphi}, \eqref{eq:com1:rhophi}, \eqref{eq:com1:Jphi}, \eqref{eq:com1:R}. The term $I^{1+\ep}_{1+\ep}[r^{-1}D_{\Lb}D_L\psi](D)$ can further be controlled by using Proposition \ref{prop:Est4DLDLbphi} with $p=1+\ep$.
\end{proof}
Now we are able to derive the energy decay estimates for the first order derivative of the scalar field. Based on the result for the decay estimates for $\phi$ in the previous subsection, it suffices to bound $\mathcal{E}_0[D_Z\phi]$.
\begin{prop}
\label{prop:bd4EDZphi}
We have the following bound:
\begin{equation}
\label{eq:bd4EDZphi}
\mathcal{E}_0[D_Z\phi]\les_{M_2}\mathcal{E}_1[\phi].
\end{equation}
\end{prop}
\begin{proof}
First by definition,
\begin{align*}
\mathcal{E}_0[D_Z\phi]&\les \mathcal{E}_1[\phi]+I^{1+\ga_0}_{1+\ep}[[\Box_A, D_Z]\phi](\{t\geq 0\})+ I^{1+\ep}_{1+\ga_0}[[\Box_A, D_Z]\phi](\{t\geq 0\}.
\end{align*}
Then from previous estimate \eqref{eq:Est4:com1}, the above inequality leads to
\begin{align*}
\mathcal{E}_0[D_Z\phi]&\les_{M_2} \ep_1 I^{-1-\ep}_{1+\ga_0-\ep}[DD_Z\phi](\{t\geq 0\})+\ep_1 I^{\ga_0}_{0}[\D\psi_1](\{t\geq 0\}\cap\{r\geq R\})+\mathcal{E}_1[\phi] \ep_1^{-1}\\
&+\ep_1\int_{\mathbb{R}}\tau_+^{1+\ga_0}g(\tau)E[D_Z\phi](\Si_{\tau})d\tau+\ep_1 \int_{\mathbb{R}}\int_{H_{\tau^*}} \tau_+^{2+\ga_0+\ep-p}g(\tau) r^{p}|D_L\psi_1|^2 dvd\om d\tau
\end{align*}
for all $0<\ep_1<1$. By our notation, the implicit constant is independent of $\ep_1$ and $\phi_1=D_Z\phi$.

Now from the integrated local energy estimates \eqref{eq:ILEdecay:sca:ex}, \eqref{eq:ILE:sca:in:sim0} combined with Lemma \ref{lem:simplint}, we can show that
\begin{align*}
I^{-1-\ep}_{1+\ga_0-\ep}[DD_Z\phi](\{t\geq 0\})\les_{M_2} \mathcal{E}_0[D_Z\phi].
\end{align*}
By the energy decay estimate \eqref{eq:ILE:Sca:ex}, \eqref{eq:Enerdecay:sca:in}, we have the energy decay for $D_Z\phi$:
\[
E[D_Z\phi](\Si_{\tau})\les_{M_2}\mathcal{E}_0[D_Z\phi]\tau_+^{-1-\ga_0},\quad \forall \tau\in \mathbb{R}.
\]
Moreover, the $r$-weighted energy estimates \eqref{eq:pWEdecay:sca:ex}, \eqref{eq:pWEdecay:sca:in:1ga} imply that
\begin{align*}
\tau_+^{1+\ga_0-p}\int_{H_{\tau^*}}r^p|D_{L}(rD_Z\phi)|^2dvd\om+\int_{\mathbb{R}}\int_{H_{\tau^*}}r^{\ga_0}|\D(rD_{Z}\phi)|^2 dvd\om d\tau\les_{M_2}\mathcal{E}_0[D_Z\phi].
\end{align*}
Recall the definition for $g(\tau)$ in line \eqref{eq:defofgu}. By Lemma \ref{cor:L2gu}, we then can demonstrate that
\begin{align*}
&\int_{\mathbb{R}}\tau_+^{1+\ga_0}g(\tau)E[D_Z\phi](\Si_{\tau})d\tau\les_{M_2} \mathcal{E}_0[D_Z\phi] \int_{\mathbb{R}}g(\tau)d\tau\les_{M_2}\mathcal{E}_0[D_Z\phi],\\
&\int_{\mathbb{R}}\int_{H_{\tau^*}} \tau_+^{2+\ga_0+\ep-p}g(\tau) r^{p}|D_L\psi_1|^2 dvd\om d\tau\les_{M_2}\mathcal{E}_0[D_Z\phi] \int_{\mathbb{R}}\tau_+^{1+\ep}g(\tau)d\tau\les_{M_2} \mathcal{E}_0[D_Z\phi].
\end{align*}
We therefore derive that
\[
\mathcal{E}_0[D_Z\phi]\les_{M_2} \ep_1 \mathcal{E}_0[D_Z\phi]+\ep_1^{-1}\mathcal{E}_1[\phi],\quad \forall 0<\ep_1<1.
\]
Take $\ep_1$ to be sufficiently small, depending only on $M_2$, $\ga_0$, $R$ and $\ep$. We then obtain estimate \eqref{eq:bd4EDZphi}.
\end{proof}
The argument then implies all the desired energy decay estimates for the first order derivative of the scalar field in terms of $\mathcal{E}_1[\phi]$. Moreover estimate \eqref{eq:Est4:com1} can be improved to be the following:
\begin{cor}
\label{cor:Est4:com1:b}
For all positive constant $\ep_1<1$, we have
\begin{equation}
\label{eq:Est4:com1:b}
\begin{split}
I^{1+\ga_0}_{1+\ep}[[\Box_A, D_Z]\phi](\{t\geq 0\})+ I^{1+\ga_0}_{1+\ep}[[\Box_A, D_Z]\phi](\{t\geq 0\}\les_{M_2}\ep_1 \mathcal{E}_1[\phi]+\mathcal{E}_0[\phi] \ep_1^{-1}.
\end{split}
\end{equation}
\end{cor}

\subsubsection{Energy decay estimates for the second order derivatives of the scalar field}
\label{sec:2ndenergy}
In this subsection, we establish the energy decay estimates for the second order derivative of the scalar field. Note that the definition of $M_2$ records the size and regularity of the connection field $A$ which is
independent of the scalar field. In particular Proposition \ref{prop:bd4EDZphi} and Corollary \ref{cor:Est4:com1:b} apply to $\phi_1=D_Z\phi$:
\begin{align*}
&\mathcal{E}_0[D_Z\phi_1]\les_{M_2}\mathcal{E}_1[\phi_1],\\
I^{1+\ga_0}_{1+\ep}[[\Box_A, D_Z]\phi_1](\{t\geq 0\})+& I^{1+\ga_0}_{1+\ep}[[\Box_A, D_Z]\phi_1](\{t\geq 0\}\les_{M_2}\ep_1 \mathcal{E}_1[\phi_1]+\mathcal{E}_1[\phi] \ep_1^{-1}
\end{align*}
for all $0<\ep_1<1$. Here by Proposition \ref{prop:bd4EDZphi}, $\mathcal{E}_0[\phi_1]\les_{M_2} \mathcal{E}_1[\phi]$. To derive the energy decay estimates for the second order derivative of the solution, it suffices to bound $\mathcal{E}_1[\phi_1]$.
As $\phi_1=D_Z\phi$, by definition
\begin{align*}
 \mathcal{E}_1[\phi_1]&=\mathcal{E}_0[\phi_1]+E^1_0[\phi_1]+I^{1+\ga_0}_{1+\ep}[D_Z\Box_A\phi_1](\{t\geq 0\})+I^{1+\ep}_{1+\ga_0}[D_Z\Box_A \phi_1](\{t\geq0\})\\
 &\les \mathcal{E}_2[\phi]+I^{1+\ga_0}_{1+\ep}[D_Z[\Box_A,D_Z]\phi](\{t\geq 0\})+I^{1+\ep}_{1+\ga_0}[D_Z[\Box_A,D_Z] \phi](\{t\geq0\})\\
 &\les \mathcal{E}_2[\phi]+I^{1+\ga_0}_{1+\ep}[[D_Z, [\Box_A,D_Z]]\phi](\{t\geq 0\})+I^{1+\ep}_{1+\ga_0}[[D_Z,[\Box_A,D_Z]] \phi](\{t\geq0\})\\
 &\qquad+I^{1+\ga_0}_{1+\ep}[[\Box_A,D_Z]D_Z\phi](\{t\geq 0\})+I^{1+\ep}_{1+\ga_0}[[\Box_A,D_Z] D_Z\phi](\{t\geq0\})\\
 &\les_{M_2} \mathcal{E}_2[\phi]+I^{1+\ga_0}_{1+\ep}[[D_Z, [\Box_A,D_Z]]\phi](\{t\geq 0\})+I^{1+\ep}_{1+\ga_0}[[D_Z,[\Box_A,D_Z]] \phi](\{t\geq0\})\\
 &\qquad+\ep_1 \mathcal{E}_1[\phi_1]+\mathcal{E}_1[\phi] \ep_1^{-1}
\end{align*}
for $0<\ep_1<1$.
Let $\ep_1$ to be sufficiently small. We then conclude that
\begin{align*}
\mathcal{E}_1[\phi_1]&\les_{M_2} \mathcal{E}_2[\phi]+I^{1+\ga_0}_{1+\ep}[[D_Z, [\Box_A,D_Z]]\phi](\{t\geq 0\})+I^{1+\ep}_{1+\ga_0}[[D_Z,[\Box_A,D_Z]] \phi](\{t\geq0\}).
\end{align*}
Therefore to bound $\mathcal{E}_1[\phi_1]$, it is reduced to control the second order commutator $[D_Z, [\Box_A,D_Z]]\phi$.

First we have the following analogue of Lemma \ref{lem:Est4commu:1}:
\begin{lem}
\label{lem:Est4commu:2}
For all $X, \, Y\in Z$, when $r\geq R$, we have
\begin{equation}
\label{eq:Est4commu:2}
\begin{split}
|[D_X,[\Box_A, D_Y]]\phi|\les & |[\Box_{\mathcal{L}_Z A}, D_Z]\phi|+(|F|^2+|r\a||r\ab|+|r\si|^2+|r\rho|(|\ab|+|\a|))|\phi|.
\end{split}
\end{equation}
When $r\leq R$, we have
\begin{equation}
\label{eq:Est4commu:2:in}
|[D_X, [\Box_A, D_Y]]\phi|\les |[\Box_{\mathcal{L}_Z A}, D_Z]\phi|+|[\Box_{A}, D_Z]\phi|+|F|^2|\phi|.
\end{equation}
Here we note that $\mathcal{L}_Z F=\mathcal{L}_Z d A=d\mathcal{L}_Z A$.
\end{lem}
\begin{proof}
First from Lemma \ref{lem:commutator}, we can write
\begin{equation*}
[\Box_A, D_{X}]\phi=2i X^\nu F_{\mu\nu}D^{\mu}\phi+i \pa^\mu( F_{\mu\nu}X^\nu)\phi.
\end{equation*}
We need to compute the double commutator $[D_Y, [\Box_A, D_X]]\phi$ for $X$, $Y\in Z$. We can compute that
\begin{align*}
[D_Y, [\Box_A, D_{X}]]\phi&=\mathcal{L}_Y\left(2i X^\nu F_{\mu\nu}D^{\mu}+i \pa^\mu( F_{\mu\nu}X^\nu)\right)\phi\\
&=2i (\mathcal{L}_Y F)(D\phi, X)+2i F([D_Y, D\phi], X)+2iF(D\phi, [Y, X])+iY(\pa^\mu( F_{\mu\nu}X^\nu))\phi.
\end{align*}
Here
\begin{align*}
[D_Y, D\phi]=D^{\mu}\phi [Y, \pa_{\mu}]+[D_Y, D^{\mu}]\phi \pa_{\mu}=i F_{Y\pa_{\mu}}\phi\pa^{\mu}-(\pa^{\mu}Y^{\nu}+\pa^{\nu}Y^{\mu})D_{\mu}\phi\pa_{\nu}
\end{align*}
As $X$, $Y\in Z$, $Z=\{\pa_t, \Om_{ij}\}$, we conclude that $X$, $Y$ are killing:
\[
\pa^\mu X^\nu+\pa^\nu X^\mu=0,\quad \pa^\mu Y^\nu+\pa^\nu Y^\mu=0.
\]
This implies that the following term can be simplified:
\begin{align*}
Y(\pa^\mu( F_{\mu\nu}X^\nu))&=[Y, \pa^\mu]F(\pa_{\mu}, X)+\pa^{\mu} (\mathcal{L}_Y F)(\pa_{\mu}, X)+\pa^{\mu} F(\mathcal{L}_Y \pa_{\mu}, X)+\pa^{\mu} F( \pa_{\mu}, \mathcal{L}_Y X)\\
&=\pa^{\mu} (\mathcal{L}_Y F)(\pa_{\mu}, X)+\pa^{\mu} F( \pa_{\mu}, [Y,X])
\end{align*}
Therefore we can write the double commutator as follows:
\begin{align*}
[D_Y, [\Box_A, D_{X}]]\phi&=2i (\mathcal{L}_Y F)(D\phi, X)+i\pa^{\mu} (\mathcal{L}_Y F)(\pa_{\mu}, X)\\
&+2iF(D\phi, [Y, X])+i\pa^{\mu} F( \pa_{\mu}, [Y,X])\phi-2 F^{\mu}_{\ X}F_{Y\mu}\phi.
\end{align*}
Note that $[X, Y]\in Z$ for $X$, $Y\in Z=\{\pa_t, \Om_{ij}\}$. We thus can write $2iF(D\phi, [Y, X])+i\pa^{\mu} F( \pa_{\mu}, [Y,X])\phi$ as $[\Box_A, D_{[Y, X]}]\phi$ and can be bounded by
using Lemma \ref{lem:Est4commu:1}. The term $2i (\mathcal{L}_Y F)(D\phi, X)+i\pa^{\mu} (\mathcal{L}_Y F)(\pa_{\mu}, X)$ has the same form with $[\Box_A, D_{X}]\phi$ if we replace $F$ with $\mathcal{L}_Y F$. In particular the bound follows from Lemma \ref{lem:Est4commu:1}. Therefore to show this lemma, it remains to control $F^{\mu}_{\ X}F_{Y\mu}\phi$ for $X$, $Y\in Z$. This term has crucial null structure we need to exploit when $r\geq R$. The main difficulty is that the angular momentum $\Om$ contains weights in $r$. If both $X, \, Y\in \Om$, then
\[
|F^{\mu}_{\ X}F_{Y\mu}|\les |r\a||r\ab|+|r\si|^2.
\]
If $X=Y=\pa_t$, then
\[
|F^{\mu}_{\ X}F_{Y\mu}|\les |F|^2.
\]
If one and only one of $X,\, Y$ is $\pa_t$, then the null structure is as follows:
\begin{align*}
|F^{\mu}_{\ X}F_{Y\mu}|&\les r|F^{\mu}_{\ L}F_{e_i\mu}|+r|F^{\mu}_{\ \Lb}F_{e_i\mu}|\les r(|\rho|+|\si|)(|\ab|+|\a|).
\end{align*}
We see that the "bad" term $r|\ab|^2$ does not appear on the right hand side. Hence
 \[
|F^{\mu}_{\ X}F_{Y\mu}|\les |F|^2+|r\a||r\ab|+|r\si|^2+|r\rho|^2,\quad \forall X,\, Y\in Z.
\]
Therefore estimate \eqref{eq:Est4commu:2} holds. On the bounded region $\{r\leq R\}$, null structure is not necessary and estimate \eqref{eq:Est4commu:2:in} follows trivially.
\end{proof}
The above lemma shows that the double commutator $[D_Z, [\Box_A, D_Z]]\phi$ consists of quadratic part $[\Box_{\mathcal{L}_Z^k A}, D_Z]\phi$ which can be bounded similar to $[\Box_A, D_Z]\phi$ as
we can put one more derivative $D_Z$ on the scalar field $\phi$ when we do Sobolev embedding. It thus suffices to control those cubic terms in \eqref{eq:Est4commu:2}.
\begin{prop}
\label{prop:Est4:com2:quad}
We have
\begin{equation}
\label{eq:Est4:com2:quad}
\begin{split}
&\sum\limits_{k\leq 1}I_{1+\ga_0}^{1+\ep}[[\Box_{\mathcal{L}_Z^k A}, D_Z]\phi](\{t\geq 0\})+ I^{1+\ga_0}_{1+\ep}[[\Box_{\mathcal{L}_Z^k A}, D_Z]\phi](\{t\geq 0\}\\
&\les_{M_2}\ep_1 I^{-1-\ep}_{1+\ga_0-\ep}[D\phi_2](\{t\geq 0\})+\ep_1 I^{\ga_0}_{0}[\D\psi_2](\{t\geq 0\}\cap\{r\geq R\})+\mathcal{E}_1[\phi] \ep_1^{-1}\\
&+\ep_1\int_{\mathbb{R}}\tau_+^{1+\ga_0}g(\tau)E[\phi_2](\Si_{\tau})d\tau+\ep_1 \int_{\mathbb{R}}\int_{H_{\tau^*}} \tau_+^{2+\ga_0+\ep-p}g(\tau) r^{p}|D_L\psi_2|^2 dvd\om d\tau
\end{split}
\end{equation}
for all positive constant $\ep_1$. Here $\phi_2=D_Z^2\phi$, $\psi_2=rD_Z^2\phi$. The function $g(\tau)$ is defined in line \eqref{eq:defofgu}.
\end{prop}
\begin{proof}
From Corollary \ref{cor:Est4:com1} and the decay estimates for the first order derivative of the scalar field, it suffices to consider estimate \eqref{eq:Est4:com2:quad} with $k=1$. The difference between estimate \eqref{eq:Est4:com1} and estimate \eqref{eq:Est4:com2:quad} is that $F$ is replaced with $\mathcal{L}_Z F$ in \eqref{eq:Est4:com2:quad}. However we are allowed to put one more derivative on the scalar field ($\phi_1=D_Z\phi$ is replaced with $D_Z^2\phi$). Note that for the proof of estimate \eqref{eq:Est4:com1} the higher order derivative comes in when we use Sobolev embedding on the sphere to bound $\|F\cdot D\phi\|_{L_{\om}^2}$:
\[
\|F\cdot D\phi\|_{L_{\om}^2}\les \sum\limits_{k\leq 2}\|\mathcal{L}_Z^k F\|_{L_\om^2}\cdot \|D\phi\|_{L_{\om}^2} \textnormal{ or }
\sum\limits_{k\leq 1}\|\mathcal{L}_Z^k F\|_{L_\om^2}\cdot \|DD_{Z}^k\phi\|_{L_{\om}^2}.
\]
For estimate \eqref{eq:Est4:com2:quad}, the corresponding term $\mathcal{L}_Z F\cdot D\phi$ can be bounded as follows:
\[
\|\mathcal{L}_Z F\cdot D\phi\|_{L_{\om}^2}\les \sum\limits_{k\leq 1}\|\mathcal{L}_Z^k \mathcal{L}_Z F\|_{L_\om^2}\cdot \|D D_Z^k \phi\|_{L_{\om}^2} \textnormal{ or }
\|\mathcal{L}_Z F\|_{L_\om^2}\cdot \sum\limits_{k\leq 2}\|DD_{Z}^k\phi\|_{L_{\om}^2}.
\]
This is how we can transfer one derivative on $F$ to the scalar field $\phi$. In particular estimate \eqref{eq:Est4:com2:quad} holds.
\end{proof}
From Lemma \ref{lem:Est4commu:2}, to bound the double commutator, it suffices to control the cubic terms in \eqref{eq:Est4commu:2}, \eqref{eq:Est4commu:2:in}. We rely on the pointwise bound for the Maxwell field summarized in Propositions \ref{prop:Est4F:in:R}, \ref{prop:supF}.
\begin{prop}
\label{prop:Est4:com2:cubic}
For all $1+\ep\leq p\leq 1+\ga_0$, we have
\begin{equation}
\label{eq:Est4:com2:cubic}
\begin{split}
&I^p_{2+\ga_0+\ep-p}[(|F|^2+|r\a||r\ab|+|r\si|^2+|r\rho|(|\ab|+|\a|))|\phi|](\{t\geq 0, \, r\geq R\})\\
&+I^p_{2+\ga_0+\ep-p}[|F|^2|\phi|](\{t\geq 0, \, r\leq R\})\les_{M_2}\mathcal{E}_1[\phi].
\end{split}
\end{equation}
\end{prop}
\begin{proof}
On the finite region $r\leq R$, the weights $r^p$ have an upper bound. The Maxwell field $F$ can be bounded by using the pointwise estimate \eqref{eq:Est4F:in:R:p}. We then can estimate the scalar field by using the integrated local energy estimates. In deed for all $0\leq \tau_1<\tau_2$, we can show that
\begin{align*}
\int_{\tau_1}^{\tau_2}\int_{r\leq R}\tau_+^{1+\ga_0}|F|^4|\phi|^2dxd\tau\les\int_{\tau_1}^{\tau_2}\tau_+^{1+\ga_0}\sup|F|^4|\phi|^2dx d\tau\les_{M_2}(\tau_1)_+^{-2-2\ga_0} \mathcal{E}_1[\phi].
\end{align*}
For the cubic terms on the region $r\geq R$, let's first consider $|r\a||r\ab||\phi|$. We use the $r$-weighted energy estimates \eqref{eq:pWE:cur:ex}, \eqref{eq:pWE:cur:in} for the Maxwell field to control $\a$ and the
integrated decay estimate \eqref{eq:supab:I} of Proposition \ref{prop:supF} to bound $\ab$. The reason that we can not use the pointwise bound \eqref{eq:supab:p} is the weak decay rate. The scalar field $\phi$ can be bounded by using Lemma \ref{lem:Est4phipWE}. In deed, for $1+\ep\leq p\leq 1+\ga_0$, we can show that
\begin{align*}
&I^{p}_{2+\ga_0+\ep-p}[|r\a||r\ab|| \phi|](\{t\geq 0\}\cap\{r\geq R\})\\
&\les\iint u_+^{2+\ga_0+\ep-p}r^{p+2}|r\a|^2|r\ab|^2|\phi|^2dudvd\om\\
& \les\sum\limits_{k\leq 1} \int_{u}\int_{v}u_+^{2+\ga_0+\ep-p}r^{1+\ga_0}\int_{\om}|r\mathcal{L}_Z^k \a|^2d\om \cdot \int_{\om}|r\mathcal{L}_Z^k \ab|^2d\om \cdot
\int_{\om}r^{p+1-\ga_0}|\mathcal{L}_Z^k \phi|^2d\om dvdu\\
&\les_{M_2}\mathcal{E}_1[\phi]\sum\limits_{k\leq 1} \int_{u}\int_{v}u_+^{1+\ep-\ga_0}r^{1+\ga_0}\int_{\om}|r\mathcal{L}_Z^k \a|^2d\om \cdot \int_{\om}|r\mathcal{L}_Z^k \ab|^2d\om dvdu\\
&\les_{M_2}\mathcal{E}_1[\phi]\sum\limits_{k\leq 1} \int_{u}u_+^{1+\ep-\ga_0}\int_{v}r^{1+\ga_0}\int_{\om}|r\mathcal{L}_Z^k \a|^2d\om dv \cdot\sup\limits_{v}  \int_{\om}|r\mathcal{L}_Z^k \ab|^2d\om \, du\\
&\les_{M_2}\mathcal{E}_1[\phi]\int_{\mathbb{R}}\tau_+^{1+\ep-\ga_0}g(\tau)d\tau\les_{M_2}\mathcal{E}_1[\phi].
\end{align*}
Here recall the definition of $g(\tau)$ in line \eqref{eq:defofgu} and the last step follows from Corollary \ref{cor:L2gu}.

For $|F|^2|\phi|$, we use the pointwise estimates \eqref{eq:supab:p}, \eqref{eq:suparhosi:p} of Proposition \ref{prop:supF} to bound the Maxwell field $F$. The scalar field $\phi$ can be bounded by using Lemma \ref{lem:Est4phipWE} as
above. In the exterior region where the the Maxwell field contains the charge part $q_0 r^{-2}dt\wedge dr$, we have the relation $r_+\geq \frac{1}{2}u_+$. We can show that
\begin{align*}
&I^{p}_{2+\ga_0+\ep-p}[|F|^2 \cdot \phi](\{t\geq 0\}\cap\{r\geq R\})\\
&\les\iint u_+^{2+\ga_0+\ep-p}r^{p+2}|\bar F|^4|\phi|^2dudvd\om+|q_0|^2\iint_{t+R\leq r} u_+^{2+\ga_0+\ep-p}r^{p+2-8}|\phi|^2dudvd\om\\
& \les_{M_2}\int_{u}\int_{v}u_+^{2+\ga_0+\ep-p-2-2\ga_0}r^{p+2-4-1}\int_{\om}r|\phi|^2d\om dvdu+\mathcal{E}_0[\phi]\\
&\les_{M_2}\mathcal{E}_1[\phi] \int_{u}\int_{v}u_+^{\ep-p-1-2\ga_0}r^{p-3}dvdu\les_{M_2}\mathcal{E}_1[\phi].
\end{align*}
For $|r\si|^2|\phi|$, same reason as the case $|r\a||r\ab||\phi|$, we are not allowed to use the pointwise bound \eqref{eq:suparhosi:p} to control $\si$ due to the strong $r$ weights here. Instead we use the $r$-weighted energy estimate for $\si$ on the incoming null hypersurface together with the integrated decay estimate \eqref{eq:suprhosi:I}.
Then we can show that
\begin{align*}
\iint_{\bar{\mathcal{D}}_{\tau_1}}r^p|r\si|^4|\phi|^2dxdt
& \les\sum\limits_{k\leq 1} \int_{u}\int_{v}r^{1+\ga_0}\int_{\om}|r\mathcal{L}_Z^k \si|^2d\om \cdot \int_{\om}|r\mathcal{L}_Z^k \si|^2d\om \cdot
\int_{\om}r^{p+1-\ga_0}|\mathcal{L}_Z^k \phi|^2d\om dvdu\\
&\les_{M_2}\mathcal{E}_1[\phi](\tau_1)_+^{-1+p-2\ga_0}\sum\limits_{k\leq 1} \int_{v}\int_{u}r^{1+\ga_0}\int_{\om}|r\mathcal{L}_Z^k \si|^2d\om du\cdot \sup\limits_{u}\int_{\om}|r\mathcal{L}_Z^k \si|^2d\om dv\\
&\les_{M_2}\mathcal{E}_1[\phi](\tau_1)_+^{-1+p-2\ga_0}\sum\limits_{k\leq 1} \|r\mathcal{L}_Z^k \si\|_{L_v^2L_u^\infty L_{\om}^2(\bar{\mathcal{D}}_{\tau_1})}^2\\
&\les_{M_2}\mathcal{E}_1[\phi](\tau_1)_+^{-2+p+\ep-3\ga_0}.
\end{align*}
This holds for all $\tau_1\in \mathbb{R}$. Since
\[
2+3\ga_0-\ep-p>2+\ga_0+\ep-p,\quad 0\leq p\leq 1+\ga_0,
\]
from Lemma \ref{lem:simplint}, we obtain
\[
I^{p}_{2+\ga_0+\ep-p}[|r\si|^2 \phi](\{t\geq 0\}\cap\{r\geq R\})\les_{M_2}\mathcal{E}_1[\phi].
\]
Finally for $|r\rho|(|\ab|+|\a|)|\phi|$, we need to take into consideration of the charge effect in the exterior region. Except this charge, the proof for the interior region case is the same. Let's merely estimate this cubic term in the exterior region. In particular take $\bar{\mathcal{D}}_{\tau_1}$ to be $\mathcal{D}_{\tau_1}$ for some $\tau_1<0$. By using the $r$-weighted energy estimate for $\bar \rho$ and the pointwise bound \eqref{eq:supab:p}, \eqref{eq:suparhosi:p} for $F$, for $0\leq p\leq 1+\ga_0$ we then can show that
\begin{align*}
&\iint_{\mathcal{D}_{\tau_1}} r^{p+2}|r\rho|^2(|\a|^2+|\ab|^2)|\phi|^2dudvd\om\\
&\les \iint_{\mathcal{D}_{\tau_1}} |q_0|r^{p}(|\a|^2+|\ab|^2)|\phi|^2dudvd\om+\iint_{\mathcal{D}_{\tau_1}} r^{p+2}|r\bar\rho|^2(|\a|^2+|\ab|^2)|\phi|^2dudvd\om\\
& \les_{M_2}\mathcal{E}_1[\phi](\tau_1)_+^{-1-2\ga_0}+\sum\limits_{k\leq 1} \int_{u}\int_{v}r^{p-1}\int_{\om}|r\mathcal{L}_Z^k \bar{\rho}|^2 d\om \cdot \sup (|r\ab|^2+|r\a|^2) \cdot
\int_{\om}r|\mathcal{L}_Z^k \phi|^2d\om dvdu\\
&\les_{M_2}\mathcal{E}_1[\phi](\tau_1)_+^{-1-2\ga_0}+\mathcal{E}_1[\phi](\tau_1)_+^{-2-2\ga_0}\sum\limits_{k\leq 1} \iint_{\mathcal{D}_{\tau_1}}r^{p-1}|r\mathcal{L}_Z^k \bar{\rho}|^2 du dv d\om \\
&\les_{M_2}\mathcal{E}_1[\phi](\tau_1)_+^{-1-2\ga_0}.
\end{align*}
Here the last term is bounded by using the $r$-weighted energy estimates for $\bar \rho$. As $\tau_1$ is arbitrary, from Lemma \ref{lem:simplint}, we derive that
\[
I^{p}_{2+\ga_0+\ep-p}[r\rho\cdot (|\ab|+|\a| )\cdot \phi](\{t\geq 0\}\cap\{r\geq R\})\les_{M_2}\mathcal{E}_1[\phi], \quad 1+\ep\leq p\leq 1+\ga_0.
\]
To summarize, we have shown \eqref{eq:Est4:com2:cubic}.
\end{proof}
The above two propositions \ref{prop:Est4:com2:quad}, \ref{prop:Est4:com2:cubic} together with Lemma \ref{lem:Est4commu:2} lead to the desired estimates for the double commutator $[D_X, [\Box_A, D_Y]]$ for $X, \, Y\in Z$.
Then by the argument at the beginning of this section, we have control of $\mathcal{E}_0[D_X D_Y \phi]$. By using the same argument of Proposition \ref{prop:bd4EDZphi}, we then can bound $\mathcal{E}_0[D_X D_Y \phi]$ by $\mathcal{E}_2[\phi]$ which then implies the decay of the second order derivative of the scalar field.
\begin{prop}
\label{prop:bd4DZZphi}
For any $X, \, Y\in Z$, we have the following bound:
\begin{equation}
\label{eq:bd4DZZphi}
\mathcal{E}_0[D_XD_Y\phi]\les_{M_2}\mathcal{E}_2[\phi].
\end{equation}
\end{prop}
\begin{proof}
From the argument at the beginning of this section (before Lemma \ref{lem:Est4commu:2}), we derive that
\begin{align*}
\mathcal{E}_0[D_X D_Y\phi]&\les_{M_2} \mathcal{E}_2[\phi]+I^{1+\ga_0}_{1+\ep}[[D_X, [\Box_A,D_Y]]\phi](\{t\geq 0\})+I^{1+\ep}_{1+\ga_0}[[D_X,[\Box_A,D_Y]] \phi](\{t\geq0\}).
\end{align*}
Then by Lemma \ref{lem:Est4commu:2} and Proposition \ref{eq:Est4:com2:quad}, for all $0<\ep_1<1$, $X$, $Y\in Z$, we conclude that
\begin{align*}
&\mathcal{E}_0[D_X D_Y\phi]\les_{M_2} \ep_1 I^{-1-\ep}_{1+\ga_0-\ep}[D\phi_2](\{t\geq 0\})+\ep_1 I^{\ga_0}_{0}[\D\psi_2](\{t\geq 0\}\cap\{r\geq R\})\\
&\quad+\mathcal{E}_1[\phi] \ep_1^{-1}+\ep_1\int_{\mathbb{R}}\tau_+^{1+\ga_0}g(\tau)E[\phi_2](\Si_{\tau})d\tau+\ep_1 \int_{\mathbb{R}}\int_{H_{\tau^*}} \tau_+^{2+\ga_0+\ep-p}g(\tau) r^{p}|D_L\psi_2|^2 dvd\om d\tau,
\end{align*}
where $\phi_2=D_XD_Y\phi$, $\psi_2=r\phi_2$. The proposition follows by the same argument of Proposition \ref{prop:bd4EDZphi}.
\end{proof}
\subsection{Pointwise bound for the scalar field}
Once we have the bound \eqref{eq:bd4DZZphi}, from Proposition \ref{prop:Enerdecay:sca:in} and Corollary \ref{cor:ILEdecay:sca:ex}, we obtain the energy flux decay estimates as well as the $r$-weighted energy estimates for
the second order derivatives of the scalar field. In another word, simply assuming $M_2$ is finite (see definition of $M_2$ in line \eqref{eq:def4NkF}), we then can derive the energy decay estimates for
the second order derivatives of the scalar field. For the MKG equations, $J=\delta F$ is quadratic in $\phi$. To construct global solutions, we need to bound these nonlinear terms.
In this section, we show the pointwise bound for the scalar field with the assumption that $M_2$ is finite.

We start with an analogue of Proposition \ref{prop:Est4F:in:R} regarding the pointwise bound of the scalar field in the finite region $r\leq R$. Similar to the pointwise bound of the Maxwell field, we use elliptic estimates. However as the connection field $A$ is general, we are not able to apply the elliptic estimates for the flat case directly. We therefore establish an elliptic lemma for the operator $\Delta_A=\sum\limits_{i=1}^3 D_iD_i$ first. Let $B_{R_1}$ be the ball with radius $R_1$ in $\mathbb{R}^3$. Define
\[
\|\phi\|_{H^k(B_{R_1})}=\sum\limits_{1\leq j_{l}\leq 3}\|D_{j_1}D_{j_2}\ldots D_{j_{k}}\phi\|_{L^2(B_{R_1})}+\|\phi\|_{H^{k-1}(B_{R_1})},\quad k\geq 1.
\]
Then we have
\begin{lem}
\label{lem:elliptic:A}
We have the following elliptic estimates:
\begin{equation}
\label{eq:elliptic:A}
\|\phi\|_{H^2(B_{R_1})}\les_{M_2, R_1, R_2}\|\Delta_A\phi\|_{L^2(B_{R_2})}+(1+\|F\|_{L^{\infty}(B_{R_2})}+\|J\|_{H^1(B_{R_2})}) \|\phi\|_{H^1(B_{R_2})}
\end{equation}
for all $R_1<R_2$. Here the constant $M_2$ is defined in line \eqref{eq:def4NkF} and $J=\delta(dA)$ or $J_{j}=\pa^{i}(dA)_{ij}$.
\end{lem}
\begin{proof}
The proof is similar to the case when the connection field $A$ is trivial. For the case when the scalar field $\phi$ is compactly supported in some ball $B_{R_1}$, using integration by parts, we can show that
\begin{align*}
\int_{B_{R_1}}D_i D_j\phi\cdot \overline{D_iD_j\phi}dx&=-\int_{B_{R_1}} D_iD_iD_j\phi \cdot \overline{D_j\phi}dx\\
&=-\int_{B_{R_1}} D_jD_iD_i\phi \cdot \overline{D_j\phi}dx-\int_{B_{R_1}} [D_iD_i, D_j]\phi \cdot \overline{D_j\phi}dx\\
&=\int_{B_{R_1}}|\Delta_A\phi|^2dx-\int_{B_{R_1}} \sqrt{-1}(2F_{ij}D_i\phi+\pa_{i}F_{ij}\phi)\cdot \overline{D_j\phi}dx.
\end{align*}
Estimate \eqref{eq:elliptic:A} then follows.

For general complex function $\phi$, we can choose a real cut-off function $\chi$ which is supported on the ball $B_{R_2}$ and equal to $1$ on the smaller ball $B_{R_1}$. By direct computation, we can show that:
\begin{align*}
\|\Delta_A(\chi\phi)\|_{L^2(B_{R_2})}&=\|\chi \Delta_A\phi+2\pa_i\chi \cdot D_i\phi+\Delta \chi \cdot \phi\|_{L^2(B_{R_2})}\\
&\les \|\Delta_A \phi\|_{L^2(B_{R_2})}+\|\phi\|_{H^1(B_{R_2})}.
\end{align*}
The lemma then follows from the above argument for the compactly supported case.
\end{proof}
We assume $\Box_A\phi$ verifies the following extra bound
\[
\int_{\tau_1}^{\tau_2}\int_{r\leq 2R}|D\Box_A\phi|^2+|D_Z D\Box_A\phi|^2dxd\tau\leq C \mathcal{E}_2[\phi](\tau_1)_+^{-1-
\ga_0},\quad 0\leq \tau_1<\tau_2
\]
for some constant $C$ depending only on $R$.
The above elliptic estimate adapted to the connection field $A$ implies the following pointwise bound for the scalar field $\phi$ on the compact region $r\leq R$.
\begin{prop}
\label{prop:Est4phi:in:R}
For all $0\leq \tau$, $0\leq \tau_1<\tau_2$, we have
 \begin{align}
  \label{eq:Est4phi:in:R}
  \int_{\tau_1}^{\tau_2}\sup_{|x|\leq R}(|D\phi|^2+|\phi|^2)(\tau, x)d\tau&\les  \int_{\tau_1}^{\tau_2}\int_{r\leq R}|D^2 D\phi|^2+|\phi|^2 dxdt \les_{ M_2}\mathcal{E}_2[\phi](\tau_1)_+^{-1-\ga_0},\\
 \label{eq:Est4phi:in:R:p}
  |D\phi|^2(\tau, x)+|\phi|^2(\tau, x)&\les_{ M_2}\mathcal{E}_2[\phi] \tau_+^{-1-\ga_0},\quad \forall |x|\leq R.
  \end{align}
\end{prop}
\begin{proof}
At the fixed time $\tau\geq 0$, consider the elliptic equation for the scalar field $\phi_k=D_Z^k\phi$:
\[
\Delta_A \phi_k=D_tD_t\phi_k+D_k\Box_A\phi+[\Box_A, D_Z^k]\phi.
\]
Proposition \ref{prop:Est4F:in:R} and \ref{prop:supF} indicate that the Maxwell field $F$ is bounded. The definition of $M_2$ shows that
\[
\|J\|_{H^1(B_{2R})}^2\les \int_{\tau}^{\tau+1}|\nabla J|^2+|\pa_t \nabla J|^2+|J|^2 +|\pa_t J|^2dx dt\les M_2.
\]
Here $B_{R_1}$ denotes the ball with radius $R_1$ at time $\tau$. Then by the previous Lemma \ref{lem:elliptic:A}, we conclude that
\begin{align*}
\|\phi_k\|_{H^2(B_{\frac{3}{2}R})}^2\les_{M_2} \|D_tD_t\phi_k\|_{L^2(B_{2R})}^2+\|D_Z^k\Box_A\phi\|_{L^2(B_{2R})}^2+\|[\Box_A, D_Z^k]\phi\|_{L^2(B_{2R})}^2+\|\phi_k\|_{H^1(B_{2R})}^2.
\end{align*}
This gives the $H^2$ estimates for $D_t\phi$, $\phi$. To obtain estimates for $D_j\phi$, commute the equation with $D_j$:
\[
\Delta_A D_j\phi=D_jD_tD_t\phi+D_j\Box_A\phi+[\Delta_A, D_j]\phi=D_jD_tD_t\phi+D_j\Box_A\phi+\sqrt{-1}(2F_{ij}D_i\phi+\pa_{i}F_{ij}\phi).
\]
Then by using Lemma \ref{lem:elliptic:A} again, we obtain
\begin{align*}
\|D_j\phi\|_{H^2(B_{R})}^2&\les_{M_2} \|D_j\phi\|_{H^1(B_{1.5R})}^2+\|\Delta_A D_j\phi\|_{L^2(B_{1.5R})}^2\\
&\les_{M_2} \|\phi\|_{H^2(B_{1.5R})}^2+\|D_j D_t^2\phi\|_{L^2(B_{1.5R})}^2+\| D_j\Box_A\phi\|_{L^2(B_{1.5R})}^2.
\end{align*}
Here we have used the fact $|F|^2\les M_2$, $\|J\|_{H^1(B_{2R})}^2\les M_2$. Then for the pointwise bound \eqref{eq:Est4phi:in:R:p}, we need to show the energy flux decay through $B_{2R}$ at time $\tau$. This can be fulfilled by considering the energy estimate obtained by using the vector field $\pa_t$ as multiplier on the region bounded by $t=\tau$ and $\Si_{\tau-R}$ (recall that $\Si_{\tau}=H_{\tau^*}$ for negative $\tau<0$). Corollary \ref{cor:ILE:sca:in:sim0} together with Proposition \ref{prop:Enerdecay:sca:in}, \ref{prop:bd4DZZphi} imply that
\begin{align*}
E[D_Z^k \phi](B_{2R})\les E[D_Z^k\phi](\Si_{\tau-R})+(\tau-R)_+^{-1-\ga_0}\mathcal{E}_0[D_Z^k\phi]\les_{M_2}\mathcal{E}_k[\phi]\tau_+^{-1-\ga_0},\quad k\leq 2.
\end{align*}
For the flux of the inhomogeneous term $D\Box_A\phi$ and the commutator term $[D_Z, l\Box_A]\phi$, we can make use of the integrated local energy estimates. More precisely combine the above $H^2$ estimates for $\phi_k=D_Z^k\phi$, $k=0$, $1$ and $D_j\phi$. We can show that
\begin{align*}
&\|D_j\phi\|_{H^2(B_{R})}^2+\sum\limits_{k\leq 1}\|\phi_k\|_{H^2(B_{R})}^2\\
&\les_{M_2} \sum\limits_{l\leq 2}E[D_Z^l\phi](B_{2R})+\| D\Box_A\phi\|_{L^2(B_{2R})}^2+\| [\Box_A, D_Z]\phi\|_{L^2(B_{2R})}^2\\
&\les_{M_2} \mathcal{E}_2[\phi]\tau_+^{-1-\ga_0}+\int_{\tau}^{\tau+1}\int_{r\leq 2R}|D\Box_A\phi|^2+|D_tD\Box_A\phi|^2dxd\tau+I^0_0[D_Z [\Box_A, D_Z]\phi](D_{\tau-R}^{\tau})\\
&\les_{M_2}\mathcal{E}_2[\phi]\tau_+^{-1-\ga_0}.
\end{align*}
Here we have used the following bound:
\[
I^{1+\ep}_{1+\ga_0}[D_Z^k[\Box_A, D_Z]\phi](\{t\geq 0\})\les_{M_2} \mathcal{E}_{k+1}[\phi],\quad k=0, 1.
\]
which is a consequence of the proof in the previous section (see the argument in the beginning of Section \ref{sec:2ndenergy}). Then Sobolev embedding implies the pointwise bound \eqref{eq:Est4phi:in:R:p} for $\phi$.

For the integrated decay estimates \eqref{eq:Est4phi:in:R}, we integrate the $H^2$ norm of $D_j\phi$ from time $\tau_1$ to $\tau_2$:
\begin{align*}
\int_{\tau_1}^{\tau_2}\|D\phi\|_{H^2(B_{R})}^2 d\tau&\les_{M_2}\int_{\tau_1}^{\tau_2}\sum\limits_{l\leq 2}\|D_Z^l\phi\|_{L^2(B_{2R})}^2+\| D\Box_A\phi\|_{L^2(B_{2R})}^2+\| [\Box_A, D_Z]\phi\|_{L^2(B_{2R})}^2 d\tau\\
&\les_{M_2}\sum\limits_{l\leq 2}I^{-1-\ep}_0[D_Z^l\phi](D_{\tau_1-R}^{\tau_2})+\int_{\tau_1}^{\tau_2}\int_{r\leq 2R}|D\Box_A\phi|^2dxd\tau+I^{0}_0[[D_Z, \Box_A]\phi](D_{\tau_1-R}^{\tau_2})\\
&\les_{M_2}\mathcal{E}_2[\phi](\tau_1)_+^{-1-\ga_0}.
\end{align*}
Here we have used the integrated local energy estimates for the second order derivative of the scalar field. Then Sobolev embedding implies the integrated decay estimate \eqref{eq:Est4phi:in:R}.
\end{proof}
\begin{remark}
\label{remark:Sobolev}
For the Sobolev embedding adapted to the connection $A$, it suffices to establish the $L^p$ embedding in terms of the $H^1$ norm. As the norm is gauge invariant we can choose a particular gauge so that the function is real. For real function $f$ we have the trivial bound $\|D_A f\|_{L^2}\geq \|\pa f\|_{L^2}$. This explains the Sobolev embedding we have used in this paper adapted to the general connection field $A$.
\end{remark}
Next we consider the pointwise bound for the scalar field outside the cylinder $r\leq R$. The decay estimate for $\phi$ easily follows from Lemma \ref{lem:Est4phipWE} as we have energy decay estimates for second order derivatives of $\phi$. However this does not apply to the derivative of $\phi$ due to the limited regularity (only two derivatives). Like the Maxwell field in Proposition \ref{prop:supF}, we rely on Lemma \ref{lem:trace}.
\begin{prop}
\label{prop:supphi}
We have the following pointwise bound:
\begin{align}
\label{eq:supDLbphi:I}
\|D_{\Lb}(rD_Z^k\phi)\|_{L_u^2 L_v^\infty L_{\om}^2(\bar{\mathcal{D}}_{\tau})}& \les_{M_2}\mathcal{E}_{k+1}[\phi] (\tau_1)_+^{-1-\ga_0+3\ep},\quad k=0, 1,\\
\label{eq:supDLphi:I}
\|r^{\frac{p}{2}}D_{L}(rD_Z^k\phi)\|_{L_v^2 L_u^\infty L_{\om}^2(\bar{\mathcal{D}}_{\tau})}^2 &\les_{M_2} \mathcal{E}_{k+1}[\phi](\tau_1)_+^{p+4\ep-1-\ga_0},\quad 0\leq p\leq 1+\ga_0-4\ep, \quad k=0, 1,\\
\label{eq:supDLpsi:p}
r^p(|D_L(r\phi)|^2+|\D(r\phi)|^2)(\tau, v, \om)&\les_{M_2}\mathcal{E}_2[\phi]\tau_+^{p-1-\ga_0},\quad 0\leq p\leq 1+\ga_0,\\
\label{eq:supDLbpsi:p}
|D_{\Lb}(r\phi)|^2(\tau, v, \om)&\les_{M_2}\mathcal{E}_2[\phi]\tau_+^{-1-\ga_0},\\
\label{eq:supphi:p}
r^p|\phi|^2(\tau, v,\om)&\les_{M_2}\mathcal{E}_2[\phi]\tau_+^{p-2-\ga_0},\quad 1\leq p\leq 2.
\end{align}
\end{prop}
\begin{remark}
If we have one more derivative (assume $M_3$), then we have better estimate for $\D(r\phi)$ as we can write it as $D_{Z}\phi$.
\end{remark}
\begin{proof}
Estimate \eqref{eq:supDLbphi:I} follows from \eqref{eq:supDLbphi}, \eqref{eq:Est4DLDLbphi} together with the $r$-weighted energy and integrated local energy estimates for the scalar field $D_Z^k\phi$, $k\leq 2$. Estimate \eqref{eq:supDLphi:I} is a consequence of \eqref{eq:supDlphi}, \eqref{eq:Est4DLDLbphi}.

For the pointwise bound for the scalar field, let $\phi_k=D_Z^k\phi$, $\psi_k=r\phi_k$, $k\leq 2$. First the $r$-weighted energy estimates \eqref{eq:pWEdecay:sca:ex}, \eqref{eq:pWEdecay:sca:in:1ga} imply that
\[
\int_{H_{\tau^*}}r^p|D_L\psi_k|^2dvd\om\les_{M_2}\mathcal{E}_k[\phi]\tau_+^{p-1-\ga_0},\quad k\leq 2, \quad 0\leq p\leq 1+\ga_0.
\]
From the $r$-weighted energy estimate for $F$ and Lemma \ref{lem:Est4phipWE}, we can bound the commutator:
\begin{align*}
\int_{H_{\tau^*}}r^p|[D_Z^2, D_L]\psi|^2dvd\om&\les \int_{H_{\tau^*}}r^p(|F_{ZL}D_Z\psi|^2+|\mathcal{L}_Z F_{ZL}\psi|^2)dvd\om\\
&\les \sum\limits_{l\leq 1}\int_{H_{\tau^*}}r^p(|\mathcal{L}_Z^l\rho \psi_{1-l}|^2+|r\mathcal{L}_Z^l\a||\psi_{l-1}|)dvd\om\\
&\les_{M_2}\mathcal{E}_2[\phi]|q_0|\tau_+^{p-3-\ga_0}+\mathcal{E}_2[\phi]\tau_+^{p-1-2\ga_0}\\
&\les_{M_2}\mathcal{E}_2[\phi]\tau_+^{p-1-\ga_0}.
\end{align*}
Here the charge part only appears when $\tau<0$. The previous two estimates lead to:
\begin{align*}
\int_{H_{\tau^*}}r^p|D_Z^kD_L D_Z^l\psi|^2dvd\om\les_{M_2}\mathcal{E}_2[\phi]\tau_+^{p-1-\ga_0},\quad k+l\leq 2, \quad 0\leq p\leq 1+\ga_0.
\end{align*}
To apply Lemma \ref{lem:trace}, we need the energy flux for $D_LD_L\psi$. From the null equation \eqref{eq:EQ4sca:null} for the scalar field, on the outgoing null hypersurface $H_{\tau^*}$, for $k=0,\, 1$, we can show that
\begin{align*}
\int_{H_{\tau^*}}r^p|D_{\Lb}D_{L}\psi_k|^2dvd\om&\les\int_{H_{\tau^*}}r^p(|\rho\cdot r\phi_k|^2+|r^{-1}\D D_{\Om}\psi_k|^2+|r\Box_A\phi_k|^2)dvd\om\\
&\les_{M_2}\mathcal{E}_{k+1}[\phi]\tau_+^{p-1-\ga_0}.
\end{align*}
Here the first term $\rho\cdot r\psi_k$ has been bounded in the above commutator estimate for $[D_Z^2, D_L]\psi $. The second term $|r^{-1}\D D_{\om}\psi_k|^2$ can be bounded by the energy flux of $\mathcal{L}_Z^2 F$ through $H_{\tau^*}$ as $p\leq 2$. The bound for $\Box_A\phi_k$ follows from the argument in the previous section \ref{sec:2ndenergy} where we have shown that $\mathcal{E}_1[\phi_k]\les_{M_2}\mathcal{E}_{2}[\phi_{k-1}]$, $k=0$, $1$. Now commute $D_L$ with $\psi_k=D_Z^k\psi$. First we can show that
\[
|D_L[D_L, D_Z]\psi|\les |L F_{LZ}||\psi|+|F_{ZL}||D_L\psi|\les (|\Lb(r\a)|+|r\mathcal{L}_Z\a|+|L\rho|)|\psi|+(|\rho|+|r\a|)|D_L\psi|.
\]
On the right hand side of the above inequality, the second term is easy to bound as we can control the Maxwell field $\rho$, $r\a$ by the $L^\infty$ norm shown in Proposition \ref{prop:supF} and the scalar field $\psi$ by the $r$-weighted energy estimates. For the first term, we have to use the null structure equations of Lemma \ref{lem:nullMKG} to control $L(r\a)$, $L\rho$. We can show that
\begin{align*}
\int_{H_{\tau^*}}r^p|D_L[D_L, D_Z]\psi|^2dvd\om &\les_{M_2}\int_{H_{\tau^*}}r^p |D_Z^2 D_L\psi|^2+r^p(|\Lb(r\a)|^2+|r\mathcal{L}_Z\a|^2+|L\rho|^2)\mathcal{E}_2[\phi]dvd\om\\
&\les_{M_2} \mathcal{E}_{2}[\phi](\tau_+^{p-1-\ga_0}+\int_{H_{\tau^*}}r^p(|\mathcal{L}_{\Om}(\rho, \si, \a)|^2+|rJ|^2+|\rho|^2)dvd\om)\\
&\les_{M_2} \mathcal{E}_{2}[\phi]\tau_+^{p-1-\ga_0}.
\end{align*}
Here we can bound $\rho$, $\a$, $\si$ by the energy flux as $p\leq 2$. For the inhomogeneous term $J$ we can use one more derivative $\mathcal{L}_{\pa_t}$. In particular we have shown that
\begin{align*}
&\sum\limits_{k\leq 1}\int_{H_{\tau^*}}r^p(|D_{L}D_Z^k D_{L}\psi|^2+|D_{\Om}D_Z^k D_L\psi|^2+|D_Z^kD_L\psi|^2)dvd\om\\
&\les \sum\limits_{l\leq 2}\int_{H_{\tau^*}}r^p(|D_Z^l D_L\psi|^2+|D_{L}[D_Z, D_{L}]\psi|^2+|D_{\Lb}D_LD_Z \psi|^2+|D_{\pa_t}D_LD_Z \psi|^2)dvd\om\\
&\les_{M_2}\mathcal{E}_{k+1}[\phi]\tau_+^{p-1-\ga_0}.
\end{align*}
Then by using Lemma \ref{lem:trace} and Sobolev embedding we derive the pointwise estimate for $D_{L}\psi$ (see Remark \ref{remark:Sobolev} for the Sobolev embedding adapted to the connection $A$). This proves the first part of \eqref{eq:supDLpsi:p}.

For $D_{\Lb}\psi$ and $\D(r\phi)$, we make use of the energy flux through the incoming null hypersurface $\Hb_{\tau}$ which is defined as $\Hb_{v}^{\tau^*, -v}$ when $\tau<0$ or $\Hb_{v}^{\tau, 2v-R}$ when $\tau\geq 0$. From the energy estimates \eqref{eq:pWEdecay:sca:ex}, \eqref{eq:ILEdecay:sca:ex}, \eqref{eq:Enerdecay:sca:in}, \eqref{eq:pWEdecay:sca:in:1ga}, we obtain the energy flux decay
\begin{align*}
\int_{\Hb_{\tau}}|D_{\Lb}D_Z^k\psi|^2+|\D D_Z^k\psi|^2 +\tau_+^{-p} r^p|D_{\Om} D_Z^k\phi|^2+r^2|D_{\Lb}D_Z^k\phi|^2dud\om\les_{M_2}\mathcal{E}_2[\phi]\tau_+^{-1-\ga_0}
\end{align*}
for $k\leq 2,\quad 0\leq p\leq 1+\ga_0$. For $\D(r\phi)=D_{\Om}\phi$, the above estimates together with Lemma \ref{lem:trace} indicate that
\[
r^{p}|D_{\Om}\phi|^2\les_{M_2}\mathcal{E}_2[\phi]\tau_+^{p-1-\ga_0},\quad 0\leq p\leq 1+\ga_0.
\]
Thus the second part of \eqref{eq:supDLpsi:p} holds.

For $D_{\Lb}\psi$, we need to pass the $D_{\Lb}$ derivative to $\psi$. We can compute the commutator:
\begin{align*}
|[D_Z^2, D_{\Lb}]\psi|\les (|r\mathcal{L}_Z\ab|+|\mathcal{L}_Z\rho|)|\psi|+(|r\ab|+|\rho|)|D_Z\psi|.
\end{align*}
We can bound $\psi$ by using Lemma \ref{lem:Est4phipWE} and $\rho$, $\ab$ by using the energy flux through $\Hb_{\tau}$. Then the previous energy estimates imply that
\begin{align}
\label{eq:DZDLbDZpsi}
\int_{\Hb_{\tau}}|D_Z^k D_{\Lb}D_Z^l\psi|^2+|r^{-1}D_Z^{k+1}\psi|^2 dud\om \les_{M_2}\mathcal{E}_2[\phi]\tau_+^{-1-\ga_0},\quad k+l\leq 2.
\end{align}
To apply Lemma \ref{lem:trace}, we also need the energy flux of $D_{\Lb}D_{\Lb}\psi$. We use the null equation \eqref{eq:EQ4sca:null} to show that
\begin{align*}
\int_{\Hb_{\tau}}|D_L D_{\Lb}\psi_k|^2dud\om\les_{M_2}\mathcal{E}_2[\phi]\tau_+^{-1-\ga_0},\quad k\leq 1.
\end{align*}
The proof of this estimate is similar to that through the outgoing null hypersurface we have done above. To pass the $D_{\Lb}$ derivative to $\psi$, we commute $D_{\Lb}$ with $\psi_1=D_Z\psi$:
\[
|D_{\Lb}[D_{\Lb}, D_Z]\psi|\les |D_{\Lb}\psi|(|r\ab|+|\rho|)+|\psi|(|\Lb\rho|+|L(r\ab)|+|\pa_t(r\ab)|).
\]
Again we can bound $D_{\Lb}\psi$ by using the energy flux and $r\ab$, $\rho$ by the $L^\infty $ norm. For the second term, $\psi$ can be bounded by using Lemma \ref{lem:Est4phipWE} and the curvature components $\Lb\rho$, $L(r\ab)$ are controlled by using the null structure equations \eqref{eq:eq4ab}, \eqref{eq:eq4rhoCu}. More precisely, we can show that
\begin{align*}
\sum\limits_{k\leq 1}\int_{\Hb_{\tau}}|D_{\Lb}D_Z^k D_{\Lb}\psi|^2dud\om&l\les\int_{\Hb_{\tau}}|D_{\Lb} [D_Z, D_{\Lb}]\psi|^2+|D_{L} D_{\Lb}D_Z^k\psi|^2+|D_{\pa_t}D_{\Lb}D_Z^k\psi|^2dud\om\\
&\les_{M_2}\mathcal{E}_2[\phi]\tau_+^{-1-\ga_0}.
\end{align*}
This estimate and \eqref{eq:DZDLbDZpsi} combined with Lemma \ref{lem:trace} imply the pointwise bound \eqref{eq:supDLbpsi:p} for $D_{\Lb}\psi$.

The pointwise bound \eqref{eq:supphi:p} for $\phi$ follows from Lemma \ref{lem:Est4phipWE}:
\[
\int_{\om}r^p|D_Z^k\phi|^2(\tau, v, \om)d\om\les_{M_2}\mathcal{E}_k[\phi]\tau_+^{p-2-\ga_0}, \quad k\leq 2, \quad 1\leq p\leq 2
\]
together with Sobolev embedding on the sphere.
\end{proof}

\section{Bootstrap argument}
\label{sec:btstp}
We use bootstrap argument to prove the ourl theorem. In the exterior region, we decompose the full Maxwell field $F$ into the chargeless part and the charge part:
\begin{equation*}
F=\bar F+q_0\chi_{\{r\geq t+R\}}  r^{-2}dt\wedge dr.
\end{equation*}
We make the following bootstrap assumptions on the nonlinearity $J_{\mu}=\pa^\nu F_{\nu\mu}=\Im(\phi\cdot\overline{D_\mu\phi})$:
\begin{equation}
\label{eq:btstrap:assum}
m_2\leq 2\mathcal{E}.
\end{equation}
Here recall the definition of $m_2$ in \eqref{eq:def4NkF}. Since the nonlinearity $J$ is quadratic in $\phi$, $m_2$ has size $\mathcal{E}^2$. By assuming small $\mathcal{E}$, we then can improve the above bootstrap assumption and hence conclude our main theorem. The smallness of $\mathcal{E}$ depends on $\mathcal{M}$. Without loss of generality, we assume $\mathcal{E}\leq 1$, $\mathcal{M}>1$.

In the definition \eqref{eq:def4NkF} for $M_2$, the main contribution is $E_0^2[\bar F]$ with $\bar F$ the chargeless part of the Maxwell field. As the scalar field $\phi$ solves the linear equation $\Box_A\phi=0$, we derive from the definition \eqref{eq:def4Nkphi} for $\mathcal{E}_2[\phi]$ that $\mathcal{E}_2[\phi]=E_0^2[\phi]$. The definition for $E_0^k[\bar F]$, $E_0^k[\phi]$ has been given in \eqref{eq:def4E0kFphi}. To proceed, we need to bound $E_0^2[\bar F]$, $\mathcal{E}_2[\phi]$ in terms of $\mathcal{M}$ and $\mathcal{E}$ which is shown in the following lemma.
\begin{lem}
 \label{lem:IDbd}
 We can bound $E_0^2[\bar F]$, $E_0^2[\phi]$ as follows:
 \begin{equation*}
  E_0^2[\bar F]\les \mathcal{M},\quad E_0^2[\phi]\les_{\mathcal{M}} \mathcal{E}.
 \end{equation*}
\end{lem}
\begin{proof}
To define the norm $E^k_0[\phi]$, we need to know the connection field $A$ at least on the initial hypersurface $t=0$. As the norm $E_0^k[\phi]$ is gauge invariant, we may choose a particular gauge. Let $\bar A=(A_1, A_2, A_3)(0, x)$, $A_0=A_0(0, x)$. We want to choose a particular connection field $( A_0, \bar A)$ on the initial hypersurface to define the gauge invariant norm $E_0^k[\phi]$ so that we are able to prove this Lemma.

It is convenient to choose Coulomb gauge to make use of the divergence free part $E^{df}$ and the curl free part $E^{cf}$ of $E$. More precisely on the initial hypersurface $t=0$, we choose $( A_0, \bar A)$ so that $\div(\bar A)=0$.
Then the compatibility condition \eqref{eq:comp:cond} is equivalent to
\[
\Delta  A_0=-\Im(\phi_0\cdot \overline{\phi_1})=-J_0(0),\quad \nabla \times \bar A=H.
\]
Define the weighed Sobolev space
\[
W^{p}_{s, \delta}:=\{f|\sum\limits_{|\b|\leq s}\|(1+|x|)^{\delta+|\b|}|\pa^\b f\|_{L^p}<\infty\}.
\]
For the special case $p=2$, let $H_{s, \delta}=W^2_{s, \delta}$. Denote $\tilde{Z}=\{\Om, \pa_j\}$, $\delta=\frac{1+\ga_0}{2}$. By the definition of $\mathcal{M}$:
\[
\|\mathcal{L}_{\tilde{Z}}^k H\|_{H_{0, \delta}}\les \mathcal{M}^{\f12},\quad k\leq 2.
\]
Then from Theorem 0 of \cite{McOwen:behaviorLaplacian} or Theorem 5.1 of \cite{christodoulou:ellipticHs}, we conclude that
\begin{equation*}
\|\mathcal{L}_{\tilde{Z}}^k\bar A\|_{H_{1, \delta-1}}\les \mathcal{M}^{\f12}, \quad k\leq 2.
\end{equation*}
This is the desired estimate for the gauge field $\bar A$. With this connection field $\bar A$, we then can define the covariant derivative $\tilde{D}=\nabla+\sqrt{-1}\bar A$ in the spatial direction. Therefore
\begin{align*}
\|D\phi(0, \cdot)\|_{H_{0, \delta}}=\|\tilde{D}\phi_0\|_{H_{0, \delta}}+\|\phi_1\|_{H_{0, \delta}}\les \mathcal{E}^{\f12}+\|\bar A\|_{W^3_{0, \delta}}\| \phi_0\|_{W^6_{0, \delta}}\les\mathcal{E}^{\f12}(1+\mathcal{M}^{\f12})\les\mathcal{E}^{\f12}\mathcal{M}^{\f12}.
\end{align*}
By the same argument and commute the equations with $D_{\tilde Z}$, we obtain same estimates for $D_{\tilde{Z}}\phi$:
\begin{equation*}
\|DD_{\tilde{Z}}^k\phi(0, \cdot)\|_{H_{0, \delta}}\les \mathcal{E}^{\f12}\mathcal{M}^{\f12},\quad k\leq 2.
\end{equation*}
To define the covariant derivative $D_0$, we need estimates for $A_0$. The difficulty is the nonzero charge. Take a cut off function $\chi(x)=\chi(|x|))$ such that $\chi=1$ when $|x|\geq R$ and vanishes for $|x|\leq \frac{R}{2}$. Denote the chargeless part of $A_0$ and $J_0$ as follows:
\[
\bar A_0=A_0+\chi q_0r^{-1}, \quad \bar J_0(0):=J_0-\Delta (\chi q_0r^{-1}).
\]
By the definition of the charge $q_0$, we then have
\[
\Delta \bar A_0=-\bar J_0(0),\quad \int_{\mathbb{R}^3}\bar J_0(0)dx=0.
\]
Recall that $J_0(0)=\Im(\phi_0\cdot \overline{\phi_1})$. By using Sobolev embedding, we can bound
\begin{align*}
\|\bar J_0(0)\|_{W^{\frac{3}{2}}_{0, 2\delta}}\les |q_0|+\|\phi_1\|_{W^{2}_{0, \delta}}\|\phi_0\|_{W^6_{0, \delta}}\les |q_0|+\|\phi_1\|_{W^{2}_{0, \delta}}\|\phi_0\|_{W^2_{1, \delta}}\les \mathcal{E}.
\end{align*}
Then from Theorem 0 of \cite{McOwen:behaviorLaplacian}, we conclude that
\begin{equation*}
\|\bar A_0\|_{W^{\frac{3}{2}}_{2, 2\delta-2}}\les \mathcal{E}.
\end{equation*}
Here the condition that $\bar A_0$ is chargeless guarantees $\bar A_0$ to belong to the above weighted Sobolev space. Then by using Gagliardo-Nirenberg interpolation inequality, we derive that
\begin{align*}
\|\nabla\bar A_0\|_{H_{0, 2\delta-\frac{1}{2}}}\les \|\nabla \bar A_0\|_{W^{\frac{3}{2}}_{0, 2\delta-1}}^\f12\cdot\|\nabla\nabla \bar A_0\|_{W^{\frac{3}{2}}_{0, 2\delta}}^\f12\les \mathcal{E} .
\end{align*}
By definition $E=\pa_t \bar A-\nabla A_0$. By our gauge choice, $\pa_t \bar A$ is divergence free, $\nabla A_0$ is curl free. In particular we derive that $E^{df}=\pa_t \bar A$, $E^{cf}=-\nabla A_0$. Take the chargeless part. We obtain that $\bar E^{cf}=\nabla \bar A_0$ when $|x|\geq R$. Therefore we can bound the weighted Sobolev norm of the chargeless part of the Maxwell field $\bar F$ on the initial hypersurface as follows:
\begin{align*}
\|\bar F\|_{H_{0, \delta}}&\leq \|F\chi_{\{|x|\leq R\}}\|_{H_{0, \delta}}+\|(\bar E, H)\chi_{\{|x|\geq R\}}\|_{H_{0, \delta}}\\
&\les \|F\chi_{\{|x|\leq R\}}\|_{H_{0, \delta}}+\|(E^{df}, H)\chi_l{\{|x|\geq R\}}\|_{H_{0, \delta}}+\|\bar E^{cf}\chi_{\{|x|\geq R\}}\|_{H_{0, \delta}}\\
&\les \mathcal{M}^{\f12}+\|\nabla\bar A_0\|_{H_{0, \delta}}\les \mathcal{M}^{\f12}.
\end{align*}
Similarly we have the same estimates for $\mathcal{L}_{\tilde Z}^k \bar F$, $k\leq 2$, that is
\begin{equation*}
\|\mathcal{L}_{\tilde Z}^k\bar F\|_{H_{0, \delta}}\les \mathcal{M}^{\f12},\quad k\leq 2.
\end{equation*}
To derive estimates for $D_{Z}^k\phi$, $\mathcal{L}_{Z}^k \bar F$ on the initial hypersurface, we use the equations:
\[
\pa_t E-\nabla\times H=\Im(\phi\cdot \tilde{D}\phi),\quad \pa_t H+\nabla\times E=0,\quad D_t\phi_1=\tilde{D}\tilde{D}\phi
\]
to replace the time derivatives with the spatial derivatives. The inhomogeneous term $\Im(\phi\cdot \tilde{E}\phi)$ or the commutators $[D_t, \tilde{D}]$ could be controlled by using Sobolev embedding together with H\"older's inequality. The lemma then follows.
\end{proof}
The above lemma then leads to the following:
\begin{cor}
\label{cor:Data:0}
Let $(\phi, F)$ be the solution of \eqref{EQMKG}. Under the bootstrap assumption \eqref{eq:btstrap:assum}, we have
\begin{equation*}
M_2\les \mathcal{M},\quad \mathcal{E}_2[\phi]\les_{\mathcal{M}} \mathcal{E}.
\end{equation*}
\end{cor}
\begin{proof}
The corollary follows from the definition of $M_2$, $\mathcal{E}_2[\phi]$ in line \eqref{eq:def4NkF}, \eqref{eq:def4Nkphi} together lemma \ref{lem:IDbd}.
\end{proof}
From now on, we allow the implicit constant $\les$ also depends on $\mathcal{M}$, that is, $B\les K$ means that $B\leq C K$ for some constant $C$ depending on $\ga_0$, $R$, $\ep$ and $\mathcal{M}$. The rest of this section is devoted to improve the bootstrap assumption.

To improve the bootstrap assumption, we need to estimate $m_2$ defined in line \eqref{eq:def4NkF}. On the finite region $r\leq R$, the null structure of $J$ is not necessary as the weights of $r$ is bounded above. When $r\geq R$, the null structure of $J$ plays a crucial role. Note that $J_{L}$, $\J=(J_{e_1}, J_{e_2})$ are easy to control as they already contain "good" terms $\D\phi$ or $D_L(r\phi)$. The difficulty is to exploit the null structure of the component $J_{\Lb}$ which is not a standard null form defined in \cite{klNullc}, \cite{klNull}. The null structure of the system is that $J_{\Lb}$ does not interact with the "bad" component $\ab$ of the Maxwell field.

For nonnegative integers $k$, denote $\phi_k=D_Z^k\phi$, $\psi_k=r\phi_k$, $F_k=\mathcal{L}_Z^k$ in this section. First we expand the second order derivative of $J=\Im(\phi\cdot \overline{D\phi})$.
\begin{lem}
\label{lem:J:Z2}
Let $X$ be $L$, $\Lb$, $e_1$, $e_2$. Then we have
\begin{equation*}
\begin{split}
|\mathcal{L}_Z^2 J|+|\nabla\mathcal{L}_Z J|&\les |D\phi_1||D\phi|+|\phi_1||D^2\phi|+|\phi||D^2\phi_{1}|+|\pa F||\phi|^2+|F||D\phi||\phi|,\quad |x|\leq R;\\
r^2|\mathcal{L}_Z^2 J_X|&\les \sum\limits_{k\leq 2} |\psi_k||D_X\psi_{2-k}|+\sum\limits_{l_1+l_2+l_3\leq 1}|\mathcal{L}_Z^{l_1} F_{ZX}||\psi_{l_2}||\psi_{l_3}|,\quad |x|>R.
\end{split}
\end{equation*}
\end{lem}
\begin{proof}
By the definition of the Lie derivative $\mathcal{L}_Z$, we can compute
\begin{align*}
\mathcal{L}_Z J_X= Z(J_X)-J_{\mathcal{L}_Z X}&=\Im(D_Z\phi \cdot \overline{D_X\phi}+\phi\cdot\overline{D_Z D_X\phi}-\phi\cdot \overline{D_{[Z, X]\phi}})\\
&=\Im(\phi_1 \cdot \overline{D_X\phi}+\phi\cdot\overline{D_X\phi_1}+\phi\cdot \overline{([D_Z, D_X]-D_{[Z, X]})\phi})\\
&=\Im(\phi_l\cdot\overline{D_X\phi_{1-l}})-F_{ZX}|\phi|^2.
\end{align*}
Here we note that $[D_Z, D_X]-D_{[Z, X]}=\sqrt{-1}F_{ZX}$ for any vector fields $Z$, $X$ and we omitted the summation sign for $l=0,\, 1$. Take one more derivative $\nabla$. The estimate on the region $\{r\leq R\}$ then follows.

Similarly the second order derivative expands as follows:
\begin{align*}
\mathcal{L}_Y\mathcal{L}_Z J_X&= Y \mathcal{L}_Z J_X- \mathcal{L}_Z J_{[Y, X]}\\
&= Y\Im(\phi_l\cdot\overline{D_X}\phi_{1-l})-Y (F_{ZX}|\phi|^2)-\Im(\phi_l\cdot\overline{D_{[Y,X]}}\phi_{1-l})+F_{Z[Y, X]}|\phi|^2\\
&= \Im(\phi_k\cdot\overline{D_X}\phi_{2-k})-(\mathcal{L}_Y F_{ZX}+F_{[Y, Z]X})|\phi|^2+\Im(\sqrt{-1}\phi_l\cdot\overline{F_{YX}\phi_{1-l}})-F_{ZX}Y|\phi|^2
\end{align*}
for any $Y\in Z$. Here we omitted the summation sigh for $k=0, 1, 2$ and $l=0, 1$. Note that
\[
\Im(\phi\cdot \overline{D_X\phi})=r^{-2}\Im(r\phi\cdot \overline{D_X (r\phi)}),\quad [Y, Z]=0 \textnormal { or }\in Z.
\]
The estimate on the region $r\geq R$ then follows. We thus finished the proof of the lemma.
\end{proof}

Next we use the above bound for $J$ to improve the bootstrap assumption.
\begin{prop}
\label{prop:btstrap:imp}
We have
\begin{equation}
\label{eq:btstrap:imp}
m_2\leq C \mathcal{E}^2
\end{equation}
for some constant $C$ depending on $\mathcal{M}$, $\ep$, $R$ and $\ga_0$.
\end{prop}

\begin{proof}
Since $M_2\les \mathcal{M}$, all the estimates in the previous section hold. In particular we have the energy flux and the $r$-weighted energy decay estimates  for the scalar field and the chargeless part of the Maxwell field up to second order derivatives. Moreover the pointwise estimates in Proposition \ref{prop:Est4F:in:R}, \ref{prop:supF}, \ref{prop:supphi}, \ref{prop:Est4phi:in:R} hold.

Let's first consider the estimate of $|J_{\Lb}|r^{-2}$ on the exterior region. We have the simple bound that $|J_{\Lb}|\leq |D_{\Lb}\phi||\phi|$. We can control $D_{\Lb}\phi$ by using the energy flux through the incoming null hypersurface $\Hb_{v}$ and $\phi$ by the $L^\infty$ norm. In particular for any $\tau<0$ we can show that
\begin{align*}
\iint_{\mathcal{D}_{\tau}^{-\infty}}|J_{\Lb}|r^{-2}dxdt&\les\int_{-\tau^*}^{\infty}\left(\int_{\Hb_{v}}|D_{\Lb}\phi|^2r^2 dud\om\right)^{\f12}\cdot \left(\int_{\Hb_{v}}|\phi|^2r^{-2}dud\om\right)^{\f12}dv \\
&\les \mathcal{E}\int_{-\tau^*}^{\infty}\tau_+^{-\frac{1+\ga_0}{2}}\left(\int_{\Hb_{v}}r^{-4}\tau_+^{-\ga_0}dud\om\right)^{\f12}dv\\
&\les\mathcal{E}\int_{-\tau^*}^{\infty}\tau_+^{-\frac{1+2\ga_0}{2}}r^{-\frac{3}{2}}dv\les \mathcal{E}\tau_+^{-1-\ga_0}.
\end{align*}
We remark here that we can not use the integrated local energy to bound the above term due to the exact total decay rate of $|J_{\Lb}|r^{-2}$. As the charge $|q_0|\les \mathcal{E}$, we therefore obtain
 \[
 |q_0|\sup\limits_{\tau\leq 0}\tau_+^{1+\ga_0}\iint_{\mathcal{D}_{\tau}^{-\infty}}|J_{\Lb}|r^{-2}dxdt\les \mathcal{E}^2,\quad \forall \tau\leq 0.
 \]
 Next we consider the estimates on the compact region $r\leq 2R$. As $|\phi_1|=|D_Z\phi|\les |D\phi|$ when $|x|\leq R$, we can bound $\phi_1$, $\phi$, $D\phi$ and $F$ by the $L^\infty$ norm obtained in \eqref{eq:Est4F:in:R:p}, \eqref{eq:Est4phi:in:R:p}. Then $D^2\phi_1$, $\pa F$ can be controlled by using the integral decay estimates \eqref{eq:Est4F:in:R}, \eqref{eq:Est4phi:in:R} on $r\leq R$. To derive estimates for $D^2\phi_k$ or $\pa F$ on the region $\{R\leq r\leq 2R\}$, we use the equation \eqref{EQMKG}. From Lemma \ref{lem:EQ4sca:null} and Lemma \ref{lem:nullMKG}, we can show that
\begin{align*}
|D^2\phi_1|+\mathcal{E}|\pa F|&\les |D\phi_2|+|D_{L}D_{\Lb}\psi_1|+|F||\phi_1|+\mathcal{E}(|\mathcal{L}_Z F|+|L(r^2\rho, r^2\si, r\ab)|+|L(r\ab)|)\\
&\les |D\phi_2|+|\Box_A\phi_1|+|F||\phi_1|+\mathcal{E}(|\mathcal{L}_Z F|+|J|).
\end{align*}
Here we omitted the easier lower order terms. On the region $\{R\leq R\leq 2R\}$, $Z$ only miss one derivative which could be recovered from the equation. From Lemma \ref{lem:J:Z2}, we can show that
\begin{align*}
&I^0_{1+\ga_0+2\ep}[|\mathcal{L}_Z^2 J|+|\nabla \mathcal{L}_Z J|](\{r\leq 2R\})\\
&\les \mathcal{E}\int_{0}^{\infty}\tau_+^{2\ep}\int_{r\leq 2R}|D^2\phi_1|^2+\mathcal{E}|\pa F|^2+|D\phi|^2 dxd\tau\\
&\les \mathcal{E}^2+\mathcal{E}\int_{0}^{\infty}\tau_+^{2\ep}\int_{R\leq r\leq 2R}|D\phi_2|^2+|\Box_A\phi_1|^2+|F|^2|\phi_1|^2+\mathcal{E}(|\mathcal{L}_Z F|^2+|J|^2)dxd\tau\\
&\les \mathcal{E}^2+\mathcal{E}I^{0}_{2\ep}[\Box_A\phi_1](\{r\geq R\})+\mathcal{E}^2I^{0}_{2\ep}[J](\{r\geq R\})\les \mathcal{E}^2.
\end{align*}
Here the implicit constant also depends on $\mathcal{M}$ and we only consider the highest order terms. The second last step follows as the integral from time $\tau_1$ to $\tau_2$ decays in $\tau_1$. Hence the spacetime integral is bounded by using Lemma \ref{lem:simplint}. The bound for $\Box_A\phi_1$ follows from Proposition \ref{prop:bd4EDZphi} and the spacetime norm for $J$ is controlled by the bootstrap assumption.

Next, we consider the case when $|x|\geq R$ where the null structure of $J$ plays a crucial role.
For $|\mathcal{L}_Z^2 J_{\Lb}|$, Lemma \ref{lem:J:Z2} implies that
\[
r^2|\mathcal{L}_Z^2 J_{\Lb}|\les |\psi_k||D_{\Lb}\psi_{2-k}|+(|r\mathcal{L}_Z^{l_1}\ab|+|\mathcal{L}_Z^{l_1}\rho|)|\psi_{l_2}||\psi_{1-l_1-l_2}|.
\]
Here the indices $k$, $2-k$, $1-l_1-l_2$, $l_1$, $l_2$ are nonnegative integers and we only consider the highest order term as the lower order terms are easier and could be bounded in a similar way. On the right hand side of the above inequality, after using Sobolev embedding on the sphere, we can bound $|\psi|$ by using Lemma \ref{lem:Est4phipWE} and $D_{\Lb}\psi$, $|\ab|$, $\rho$ by using the integrated local energy estimates. In deed we can show that
\begin{align*}
&I^{1-\ep}_{1+\ga_0+2\ep}[\mathcal{L}_Z^2 J_{\Lb}](\{r\geq R\})=\int_{\tau}\int_{H_{\tau^*}}r_+^{-\ep-1}\tau_+^{1+\ga_0+2\ep}|r^2\mathcal{L}_Z^2J_{\Lb}| ^2dud\om d\tau\\ &\les\int_{\tau}\int_{v}r_+^{-1-\ep}\tau_+^{1+\ga_0+2\ep} \int_{\om}|\psi_2|^2d\om\cdot \int_{\om}|D_{\Lb}\psi_{2}|^2d\om +\int_{\om}|r\mathcal{L}_Z^{2}\ab|^2+|\mathcal{L}_Z^{2}\bar\rho|^2d\om(\int_{\om}|\psi_{2}|^2d\om)^2 dv d\tau\\
&\qquad+\int_{\tau\leq 0}\int_{v}r^{-1-\ep}|q_0|^2r^{-4} \tau_+^{1+\ga_0+2\ep}(\int_{\om}|\psi_{2}|^2d\om)^2 dvd\tau\\
&\les \mathcal{E}\int_{\tau}\tau_+^{1+2\ep}\int_{H_{\tau^*}}\frac{|\bar D\phi_2|^2}{r_+^{1+\ep}}dxd\tau+\mathcal{E}^2\int_{\tau}\tau_+^{1+2\ep}\int_{H_{\tau^*}}\frac{|\mathcal{L}_Z^2 \bar F|^2}{r_+^{1+\ep}}dxd\tau
+\mathcal{E}^2|q_0|^2\int_{\tau\leq 0}\int_{v}r_+^{-4+\ep+\ga_0}dvd\tau\\
&\les \mathcal{E} I^{-1-\ep}_{1+2\ep}[\bar D \phi_2](\{t\geq 0\})+\mathcal{E}^2 I^{-1-\ep}_{1+2\ep}[\mathcal{L}_Z^2\bar F](\{t\geq 0\})+|q_0|^2\mathcal{E}^2\les \mathcal{E}^2.
\end{align*}
Here after using Sobolev embedding on the sphere, we dropped the lower order terms like $\psi_1$, $\psi$. In the above estimate, we have used the decay estimates $\int_{\om}|\psi_k|^2d\om\les \mathcal{E}\tau_+^{-\ga_0}$ by Lemma \ref{lem:Est4phipWE}. The last step follows from the integrated local energy decay (see e.g. estimate \eqref{eq:Enerdecay:sca:in}) and Lemma \ref{lem:simplint}. We also note that in the exterior region $r_+\geq \frac{1}{2}\tau_+$.

For $J_{L}$, Lemma \ref{lem:J:Z2} indicates that
 \[
r^2|\mathcal{L}_Z^2 J_{L}|\les  |\psi_k||D_L\psi_{2-k}|+(|r\mathcal{L}_Z^{l_1}\a|+|\mathcal{L}_Z^{l_1}\rho|)|\psi_l{l_2}||\psi_{1-l_1-l_2}|.
\]
Similarly, after using Sobolev embedding, we control $\psi_k$ by using Lemma \ref{lem:Est4phipWE}. Then for $D_{L}\psi_k$, $|\mathcal{L}_Z^k\a|$ we can apply the $r$-weighted energy estimates. For $\rho$, we split it into the charge part $q_0r^{-2}$ and the chargeless part which can be bounded by using the energy flux decay estimates. More precisely for $\ep\leq p\leq 1+\ga_0$ we can show that
\begin{align*}
&I^{1+p}_{1+\ga_0+\ep-p}[\mathcal{L}_Z^2 J_{L}](\{r\geq R\})=\int_{\tau}\int_{H_{\tau^*}}r_+^{p-1}\tau_+^{1+\ga_0+\ep-p}|r^2\mathcal{L}_Z^2J_{L}| ^2dvd\om d\tau\\ &\les\mathcal{E}\int_{\tau}\int_{H_{\tau^*}}r_+^{p}\tau_+^{\ep-p} |D_{L}\psi_{2}|^2d\om  dvd\tau +\mathcal{E}^2\int_{\tau}\int_{H_{\tau^*}}r_+^{p}\tau_+^{\ep-p-\ga_0}(|r\mathcal{L}_Z^{2}\a|^2+|\mathcal{L}_Z^{2}\bar\rho|^2)d\om dv d\tau\\
&\qquad+\mathcal{E}^2\int_{\tau\leq 0}\int_{v}r^{p-1}|q_0|^2r^{-4} \tau_+^{1-\ga_0+\ep-p} dvd\tau\\
&\les \mathcal{E}^2\int_{\tau}\tau_+^{\ep-p-1-\ga_0+p}+\tau_+^{\ep-p-\ga_0-1-\ga_0+p}+\tau_+^{\ep-p-1-\ga_0}d\tau+\mathcal{E}^2|q_0|^2\les \mathcal{E}^2.
\end{align*}
Here for $\psi_k$, we have used the estimate $r^{-1}\int_{\om}|\psi_k|^2d\om \les \mathcal{E}\tau_+^{-1-\ga_0}$.

Next for $\J$, Lemma \ref{lem:J:Z2} shows that
\[
r^2|\mathcal{L}_Z^2 \J|\les |\psi_k||\D\psi_{2-k}|+(|r\mathcal{L}_Z^{l_1} \si|+|\mathcal{L}_Z^{l_1} \a|+|\mathcal{L}_Z^{l_1} \ab|)|\psi_{l_2}||\psi_{1-l_1-l_2}|.
\]
Like the previous estimates for $J_{\Lb}$, $J_{L}$, for all $\ep\leq p\leq \ga_0$ we can show that
\begin{align*}
&I^{1+p}_{1+\ga_0+\ep-p}[\mathcal{L}_Z^2 \J_{L}](\{r\geq R\})=\int_{\tau}\int_{H_{\tau^*}}r_+^{p-1}\tau_+^{1+\ga_0+\ep-p}|r^2\mathcal{L}_Z^2 \J|^2dvd\om d\tau\\ &\les\mathcal{E}\int_{\tau}\int_{H_{\tau^*}}r_+^{p}\tau_+^{\ep-p} |\D\psi_{2}|^2d\om  dvd\tau +\mathcal{E}^2\int_{\tau}\int_{H_{\tau^*}}r_+^{p}\tau_+^{\ep-p-\ga_0}(|r\mathcal{L}_Z^{l_1} \si|^2+|\mathcal{L}_Z^{l_1} (\ab,\a)|^2)d\om dv d\tau\\
&\les\mathcal{E}\int_{\tau}\int_{H_{\tau^*}}r_+^{\ga_0} (|\D\psi_{2}|^2+\mathcal{E}|r\mathcal{L}_Z^{l_1} (\si,\a)|^2)d\om  dvd\tau +\mathcal{E}^2\int_{\tau}\int_{H_{\tau^*}}r_+^{1-\ep}|\mathcal{L}_Z^{l_1} \ab|^2d\om dv d\tau\\
&\les \mathcal{E}^2.
\end{align*}
Here $l_1\leq 1$. The last term is bounded by using the integrated local energy estimates. This relies on the assumption that $\ga_0\leq 1-\ep<1$. For $\ga_0\geq 1$, we then can use the improved integrated local energy estimate for the angular derivatives of $\phi$ or $\si$ or we can move the $r$ weights to $\phi_k$.

Combining the above estimates, we have \eqref{eq:btstrap:imp}.
\end{proof}

By choosing $\mathcal{E}$ sufficiently small depending only on $\mathcal{M}$, $\ep$, $R$ and $\ga_0$, we then can improve the bootstrap assumption \eqref{eq:btstrap:assum}. To prove theorem \ref{thm:dMKG:small}, we can choose $R=2$. Then for sufficiently small $\mathcal{E}$, we can bound $m_2$ and $M_2$. The pointwise estimates in the main Theorem \ref{thm:dMKG:small} follow from Propositions \ref{prop:Est4F:in:R}, \ref{prop:supF}, \ref{prop:supphi}, \ref{prop:Est4phi:in:R}.

\bibliography{shiwu}{}
\bibliographystyle{plain}

\bigskip

DPMMS, Centre for Mathematical Sciences, University of Cambridge,
Wilberforce Road, Cambridge, UK CB3 0WA

\textsl{Email address}: S.Yang@dpmms.cam.ac.uk
\end{document}